\documentclass[11pt]{amsart}
\usepackage{amsmath,amsthm,amscd,amsfonts, amssymb, mathrsfs}
\usepackage{graphicx}
\usepackage{tikz}
\usepackage{color}
\include{diagxy}
\usepgflibrary{shapes}
\newcommand{\sect}[1]{\section{#1}\setcounter{equation}{0}}
\newcommand{\subsect}[1]{\subsection{#1}}

\font\mbn=msbm10 scaled \magstep1
\font\mbs=msbm7 scaled \magstep1
\font\mbss=msbm5 scaled \magstep1
\newfam\mbff
\textfont\mbff=\mbn
\scriptfont\mbff=\mbs
\scriptscriptfont\mbff=\mbss   

\newcommand{\Di}      {\mathbb{D}}
\newcommand{\RR}       { \mathbb{R}}

\newcommand{\N}       { \mathbb{N}}
\newcommand{\Z}        {\mathbb{Z}  } 
\newcommand\Co           {{\mathbb C}}

\newtheorem{Th}{Theorem}[section]
\newtheorem{Lm}[Th]{Lemma}
\newtheorem{C}[Th]{Corollary}

\newtheorem{Proposition}[Th]{Proposition}
\newtheorem{R}[Th]{Remark}
\newtheorem*{Problem}{Sz.-Nagy Problem}
\newtheorem*{Convention}{Convention}
\newtheorem{E}[Th]{Example}

\newtheorem*{P2}{Completion Problem}

\begin{document}
\title[Oka Principle on  the Maximal Ideal Space of  ${\mathbf H^\infty}$]{Oka Principle on the Maximal Ideal Space of  ${\mathbf H^\infty}$}

\author{Alexander Brudnyi} 
\address{Department of Mathematics and Statistics\newline
\hspace*{1em} University of Calgary\newline
\hspace*{1em} Calgary, Alberta\newline
\hspace*{1em} T2N 1N4}
\email{abrudnyi@ucalgary.ca}

\thanks{Research supported in part by NSERC}

\begin{abstract}
The classical Grauert and Ramspott theorems constitute the foundation of the Oka principle on Stein spaces. In this paper we establish analogous results on the maximal ideal space $M(H^\infty)$ of the Banach algebra $H^\infty$ of bounded holomorphic functions on the open unit disk $\Di\subset\Co$. We illustrate our results by some examples and applications
to the theory of operator-valued $H^\infty$ functions.
 \end{abstract}

\date{}

\maketitle

\sect{Introduction}
Let $H^\infty$ be the Banach algebra of bounded holomorphic functions in the open unit disk $\Di\subset\Co$ equipped with pointwise multiplication and supremum norm. In this paper, following our earlier work \cite{Br1}, \cite{Br2}, \cite{Br3}, we investigate further the relationship between certain analytic and topological objects on the {\em maximal ideal space} of $H^\infty$. The subject is intertwined with the area of the so-called {\em Oka Principle} which in a broad sense means that on Stein spaces (closed complex subvarieties of complex coordinate  spaces) cohomologically formulated analytic problems have only topological obstructions (for recent advances in the theory see, e.g., survey \cite{FL}).
This principle can be transferred, to some extent, to the theory of commutative Banach algebras to reveal (via the Gelfand transform) some connections between algebraic structure of a Banach algebra and topological properties of its maximal ideal space. (The most general results there are due to Novodvorski, Taylor and Raeburn, see, e.g., \cite{R} and references therein.)

Recall that for a unital commutative complex Banach algebra $A$ the maximal
ideal space $M(A)$ is the set of all nonzero homomorphisms $A\rightarrow\Co$.
Since norm of each $\varphi\in M(A)$ is at least one, 
$M(A)$ is a subset of the closed unit ball of the dual space $A^{*}$. It is a compact Hausdorff space in the weak$^*$ topology induced by
$A^{*}$ (called the {\em Gelfand topology}). Let $C(M(A))$ be
the Banach algebra of continuous complex-valued functions on $M(A)$ equipped with 
supremum norm. An element $a\in A$ can be thought of as a
function in $C(M(A))$ via the {\em Gelfand transform}
$\, \hat{} : A\rightarrow
C(M(A))$, $\hat a(\varphi):=\varphi(a)$. Map
$\,\hat{}\ $ is a nonincreasing-norm morphism
of Banach algebras. Algebra $A$ is called {\em uniform} if the Gelfand transform is an isometry (as for $H^\infty$).

For $H^\infty$ evaluation at a point of $\Di$ is an element of $M(H^\infty)$, so $\Di$ is naturally embedded into $M(H^\infty)$ as an open subset. The famous Carleson corona theorem \cite{C} asserts that $\Di$ is dense in $M(H^\infty)$.
In general, if a Stein space $X$ is embedded as an open dense subset into a normal topological space $\bar X$, it is natural to introduce an analog of the complex analytic structure on $\bar X$ regarding it as a ringed space with the structure sheaf $\mathcal O_{\bar X}$  of germs of complex-valued continuous functions on open subsets $U$ of $\bar X$ whose restrictions to $U\cap \bar X$ are holomorphic.  
Then one can ask whether analogs of classical results of complex analysis (including
the Oka Principle) are valid on $(\bar X,\mathcal O_{\bar X})$. Some results in this direction for $\bar X$ being a fiberwise compactification of an unbranched covering $X$ of a Stein manifold have been obtained in \cite{BK1}, \cite{BK2}.  Another example, $X=\Di$ and $\bar X=M(H^\infty)$, has been the main subject of papers \cite{Br1}, \cite{Br2}, \cite{Br3} where the Stein-like theory analogous to the classical complex function theory on Stein spaces (see, e.g., \cite{GR}) has been developed and then applied to
the celebrated Sz.-Nagy operator corona problem \cite{SN}.  In the present paper we continue this line of research and establish (in the framework of the Oka Principle on $M(H^\infty)$) the Grauert and Ramspott type theorems (see \cite{Gra}, \cite{Ra}) for holomorphic principal bundles on $M(H^\infty)$ and holomorphic maps of $M(H^\infty)$ into complex homogeneous spaces. Then we formulate and prove some applications of the obtained results, in particular, to the theory of operator-valued $H^\infty$ functions. \smallskip

The paper is organized as follows.\smallskip

Section 2 is devoted to the Oka Principle for holomorphic principal bundles. We work in a more general setting of unital subalgebras $H_I^\infty\subseteq H^\infty$ generated by closed ideals $I\subset H^\infty$. In Proposition \ref{prop1.1} we describe the topological structure of maximal ideal spaces $M(H_I^\infty)$. Then we define holomorphic principal bundles on $M(H_I^\infty)$ with fibres complex Banach Lie groups and prove Grauert-Bungart-type theorems for them (Theorems \ref{teo1.7}, \ref{teo2.9}) asserting that  inclusion of the category of holomorphic principal bundles on $M(H_{I}^\infty)$ to the category of topological ones induces a bijective map between the corresponding sets of isomorphism classes of bundles. 
Further, in Theorem \ref{teo1.10} we prove  that  holomorphic principal bundles on $M(H_I^\infty)$ are isomorphic to holomorphic bundles pulled back from certain Stein domains by maps of the inverse limit construction describing $M(H_I^\infty)$ as the inverse limit of the inverse system of maximal ideal spaces of finitely generated subalgebras of $H_I^\infty$ (see, e.g., \cite{Ro}). This establishes a connection between our intrinsic analytic theory of $M(H_I^\infty)$ and that of Novodvorski, Taylor and Raeburn.\smallskip

Section 3 deals with the Oka Principle for holomorphic maps from $M(H_I^\infty)$ to complex Banach homogeneous spaces.

In Section~3.1 we prove a Ramspott-type theorem (Theorem \ref{teo3.5}) providing a one-to-one correspondence between sets of path-connected components of spaces of holomorphic and continuous maps of $M(H_I^\infty)$ into complex Banach homogeneous spaces $X$.  In turn,  Theorem \ref{teo3.6}
asserts that up to homotopy equivalence holomorphic maps $M(H_I^\infty)\rightarrow X$  can be uniformly approximated by holomorphic maps into $X$ pulled back from certain Stein domains by the inverse limit projections of Theorem \ref{teo1.10}.

Section~3.2 is devoted to a Runge-type theorem 
(Theorem \ref{nonlinrunge}) for holomorphic maps from some subsets of $M(H_I^\infty)$ to complex Banach homogeneous spaces $X$ and a nonlinear interpolation problem for such maps  (Theorem \ref{te2.2}) going back to Carleson \cite{C} ($X=\Co$) and Treil \cite{T6} ($X=\Co^*\, (:=\Co\setminus\{0\}$)). 
For instance, we prove that if $X$ is simply connected, then a holomorphic map into $X$ defined on a neighbourhood of a holomorphically convex subset $K\subset M(H_I^\infty)$ can be uniformly approximated on $K$ by holomorphic maps $M(H_I^\infty)\rightarrow X$ and that if $X$ is a complex Banach Lie group and $K$ is the zero locus of the image under the Gelfand transform of $I$ in $C(M(H_I^\infty))$, then for each holomorphic map $F:U\rightarrow X$ defined on a neighbourhood $U$ of $K$ there is a holomorphic map $\tilde F: M(H_I^\infty)\rightarrow X$ such that $\tilde F|_K=F|_K$.
\smallskip

Section 4 contains some applications and examples of the obtained results. 

In Section 4.1 we describe the structure of spaces of holomorphic maps of $M(H_I^\infty)$ into complex flag manifolds and tori, the latter by means of  BMOA functions (Theorem \ref{bmoa1}). 

In Section 4.2 we prove some results about holomorphic maps from $M(H_I^\infty)$ into the space of idempotents of a complex unital Banach algebra $\mathfrak A$ (Theorem \ref{teo4.1}).
In particular, we show that in some cases (e.g., if $\mathfrak A$ is the algebra of bounded linear operators on a Hilbert or $\ell^p$ space) such
$\mathfrak A$-valued holomorphic idempotents can be transformed by appropriate holomorphic similarity transformations to constant idempotents of $\mathfrak A$. 

In Section 4.3 we study analogs of the Sz.-Nagy operator corona problem for holomorphic maps of $M(H_I^\infty)$ into the space of left-invertible elements of a complex Banach algebra $\mathfrak A$ with unit $1_{\mathfrak A}$. In particular, we solve a general problem considered by Vitse \cite{V} proving that 
a holomorphic map $F$ on $\Di$ with a relatively compact image in $\mathfrak A$ has a holomorphic left inverse with a relatively compact image $G$ (i.e., such that $G(z)H(z)=1_{\mathfrak A}$ for all $z\in \Di$) if and only if for every $z\in \Di$ there is a left inverse $G_z\in\mathfrak A$ of $F(z)$ such that the family $\{G_z\}_{z\in\Di}\subset\mathfrak A$ is uniformly bounded.

Finally, in Example \ref{ex4.11} we discuss a more general than the Sz.-Nagy problem, the so-called Completion Problem, asking about extension of a bounded $H^\infty$ operator-valued function to an invertible one (see, e.g., \cite{T5} and references therein).\smallskip

Sections 5--12 contain proofs of results of the paper.

\sect{Oka Principle for Principal Bundles}
\subsection{Maximal Ideal Spaces of Algebras ${\mathbf H_I^\infty}$.} We work in a more general setting of (uniform) Banach algebras $H_I^\infty:=\Co+I$, where $I\subset H^\infty$ is a closed ideal. Such algebras arise naturally in the theory of bounded holomorphic functions in balls and polydisks, see \cite{AM1}, \cite{AM2}. They were also studied in the framework of the theory of univariate $H^\infty$ functions, see \cite{Gam}, \cite{MSW}. 

The corona theorem for algebra $H_I^\infty$ can be derived from the Carleson corona theorem, see \cite[Th.\,1.6]{MSW}. It states that for a $n$-tuple of functions  $f_1,\dots, f_n\in H_I^\infty$, $n\in\mathbb N$, satisfying the corona condition
\begin{equation}\label{e1.1}
\sum_{j=1}^n |f_j(z)|\ge\delta>0\quad {\rm for\ all}\quad z\in\Di,
\end{equation}
there exist functions $g_1,\dots, g_n\in H_I^\infty$ such that 
\begin{equation}\label{bezout}
\sum_{j=1}^n f_j g_j=1\quad {\rm on}\quad \Di.
\end{equation}
From here using some basic results due to Su\'{a}rez \cite{S1} and Treil \cite{T6} on the structure of $M(H^\infty)$ one obtains the following topological description of $M\bigl(H_{I}^\infty\bigr)$.

For an ideal $I\subset H^\infty$, we define 
\begin{equation}\label{e1.2}
{\rm hull}(I):=\{x\in M(H^\infty)\, :\,  \hat f(x)=0\quad \forall f\in I\}.
\end{equation}
\begin{Proposition}\label{prop1.1}
{\rm (a)} There is a continuous surjective map $Q_Z: M(H^\infty)\rightarrow M(H_{I}^\infty)$, $Z:={\rm hull}(I)$, sending $Z$ to a point and one-to-one outside of $Z$.\smallskip

\noindent {\rm (b)} Covering dimension ${\rm dim}\, M(H_{I}^\infty)=2$. \smallskip
 
\noindent {\rm (c)} \v{C}ech cohomology group $H^2(M(H_{I}^\infty),\mathbb Z)=0$.
 \end{Proposition}
\noindent (Recall that for a normal space $X$, $\text{dim}\, X\le n$ if every finite open cover
of $X$ can be refined by an open cover whose order $\le n + 1$. If $\text{dim}\, X\le n$ and
the statement $\text{dim}\, X\le n-1$ is false, we say that $\text{dim}\, X = n$.)
\begin{R}\label{rem1.2}
{\rm
(1) Part (a) of the proposition says that $M(H_{I}^\infty)$ is homeomorphic to the (Alexandroff) one-point compactification of space $M(H^\infty)\setminus Z$ and $\Di\setminus Z$ is an open dense subset of $M(H_{I}^\infty)$. 
\smallskip

\noindent (2)  Proposition \ref{prop1.1} implies that algebra $H_I^\infty$ is {\em projective free} (i.e., every projective $H_I^\infty$-module is free), see, e.g., \cite[Cor.\,1.4]{BS}.
}
\end{R}

A closed subset $Z\subset M(H^\infty)$ such that $Z={\rm hull}(I)$  for an ideal $I\subset H^\infty$ is called a {\em hull}. For a hull $Z$ by ${\mathscr A}_Z$ we denote the partially ordered by inclusion set of all algebras $H_I^\infty$ for which ${\rm hull}(I)=Z$.  E.g., if $Z=\emptyset$, then $\mathscr A_Z=\{H^\infty\}$. But, in general, set ${\mathscr A}_Z$ may be even uncountable.
\begin{E}\label{ex1.3}
{\rm Let $v$ be a (unbounded) holomorphic function on $\Di$ such that $u:=e^v\in H^\infty$ is an inner function.
For each $\alpha\in (0,1)$ we define the inner function $u_\alpha:=e^{\alpha v}\in H^\infty$. Let $I(u_\alpha)\subset H^\infty $ be the principal ideal generated by $u_\alpha$.
Then ${\rm hull}(I(u_\alpha))={\rm hull}(I(u_1))=:Z$ is a nonempty compact subset of $M(H^\infty)\setminus\Di$.  Moreover, all ideals $I(u_\alpha)$ are closed and $ I(u_{\alpha})\subsetneq I(u_\beta)$ for $\beta<\alpha$. Thus, the corresponding set ${\mathscr A}_Z$ contains a subset of the cardinality of the continuum. }
\end{E}
Each set $\mathscr A_Z$ contains the unique maximal subalgebra $H^\infty_{I(Z)}$, where
$I(Z):=\{f\in H^\infty\,:\,\hat f(x)=0\quad \forall x\in Z\}$.
For instance, if $Z$ is a single point, then $I(Z)\subset H^\infty$ is a maximal ideal and $H^\infty_{I(Z)}= H^\infty$. In turn, if $Z$ is the zero locus of $\hat b$, where $b$ is a Blaschke product with simple zeros, then $H^\infty_{I(Z)}=H^\infty_{I(b)}$. 

\begin{Convention} {\rm In what follows we assume that map $Q_Z$ of Proposition \ref{prop1.1} satisfies \penalty-10000
$Q_Z|_{M(H^\infty)\setminus Z}={\rm id}$. Then 
maximal ideal spaces of algebras in $\mathscr A_Z$ coincide with $M(H_{I(Z)}^\infty)$. This particular space will be denoted by $M(\mathscr A_Z)$.}
\end{Convention}

\subsection{ Oka Principle for Holomorphic Principal Bundles on  $\mathbf{M({\mathscr A}_Z)}$.}  
Let $U\subset M({\mathscr A}_Z)$ be an open subset and $X$ be a complex Banach manifold (i.e., a complex manifold modelled on a complex Banach space).  
\smallskip

A  continuous map $f:U\rightarrow X$ is said to be {\em holomorphic} (written as $f\in \mathcal O(U,X)$) if restriction $f|_{U\cap( \Di\setminus Z)}$ is a holomorphic map of complex manifolds. 

For $X=\Co$ we set $\mathcal O(U):=\mathcal O(U,\Co)$.\smallskip

Using \cite[Prop.\,1.3]{Br1} one obtains the following description of $X$-valued holomorphic maps on $U\cap( \Di\setminus Z)$ having continuous extensions to $U$.
\begin{Proposition}\label{prop1.4}
A map $f\in\mathcal O(U\cap( \Di\setminus Z),X)$ extends to a map in $\mathcal O(U,X)$ if and only if there exist open covers $(U_\alpha)_{\alpha\in A}$ of $U$  and $(V_\beta)_{\beta\in B}$ of $X$ and a map $\tau :A\rightarrow B$ such that 
\begin{itemize}
\item[(a)] for each $\beta\in B$ holomorphic functions on $V_\beta$ separate points and
\[
f(U_\alpha\cap( \Di\setminus Z))\Subset V_{\tau(\alpha)}\quad {\rm for\ all}\quad \alpha\in A;
\]
\item[(b)]
\[
\bigcap_{W\subset Q_Z^{-1}(U)\,:\, \mathring{W}\cap Z\ne\emptyset}\overline{f\bigl(W\cap( \Di\setminus Z)}\ne\emptyset.
\]
\end{itemize}
\end{Proposition}
\noindent (Here for topological spaces $S,Y$ the implication $S\Subset Y$ means that the closure $\bar S$ of $S$ in $Y$ is compact and $\mathring{S}$ stands for the interior of $S$.)

\begin{R}\label{rem2.5}
{\rm (1) Condition (a) implies that map $Q_Z^*f\, (:=f\circ Q_Z)$ extends continuously to $M(H^\infty)$ and then (b) guarantees  that  this extension is constant on $Z$. If $Z=\overline{Z\cap\Di}$, then instead of (b) one can assume that $Q_Z^*f$ extends to a continuous map on $\Di$ taking  the same value on $Z$.\smallskip

\noindent (2) If $X$ is a complex submanifold of a complex Banach space,  $f\in\mathcal O( \Di\setminus Z,X)$ extends to a map in $\mathcal O(M(\mathscr A_Z),X)$ iff $f( \Di\setminus Z)\Subset X$ and for each $g\in\mathcal O(X)$ the Gelfand transform $\widehat{g\circ f}$ is constant on $Z$.  (Note that $g\circ f\in\mathcal O(\Di\setminus Z)$ extends to a function in $H^\infty$ by the Riemann extension theorem.)
Here the first condition implies that $Q_Z^*f$ extends continuously to $M(H^\infty)$ by \cite[Prop.\,1.3]{Br1}, and the second one that this extension is constant on $Z$.
}
\end{R}
 
 Let $U$ be an open subset of  $M(\mathscr A_Z)$.
A topological principal $G$-bundle $\pi: P\rightarrow U$ with fibre a complex Banach Lie group is called {\em holomorphic} if it is defined on an open cover $(U_i)_{i\in I}$ of $U$ by a cocycle $\{g_{ij}\in \mathcal O(U_i\cap U_j, G)\}_{i,j\in I}$.  In this case, $P|_{U\cap (\Di\setminus Z)}$ is a holomorphic principal $G$-bundle on $U\cap (\Di\setminus Z)$ in the usual sense.\smallskip

Recall that $P$ is defined as the quotient space of disjoint union $\sqcup_{i\in I}\,U_i\times G$ by the equivalence relation:
\begin{equation}\label{equ1.3}
U_j\times G\ni u\times g\sim u\times g g_{ij}(u)\in U_i\times G. 
\end{equation}
The projection $\pi:P\rightarrow U$ is induced by natural projections $U_i\times G\rightarrow U_i$, $i\in I$.\smallskip

A bundle  isomorphism $\varphi : (P_1, G, \pi_1)\rightarrow (P_2, G, \pi_2)$ of holomorphic principal $G$-bundles on $U$ is called {\em holomorphic} if $\varphi|_{U\cap (\Di\setminus Z)}: P_1|_{U\cap (\Di\setminus Z)}\rightarrow P_2|_{U\cap (\Di\setminus Z)}$ is a biholomorphic map of complex Banach manifolds.  \smallskip

We say that a holomorphic principal $G$-bundle $(P,G,\pi)$ on $U$ is {\em trivial} if it is holomorphically isomorphic to the trivial bundle
$U\times G$. (For basic facts of the theory of bundles, see, e.g., \cite{Hus}.)\smallskip

For a  complex Banach Lie group $G$ by $G_0$ we denote the connected component containing   unit $1_G\in G$. Then $G_0$ is a clopen normal subgroup of $G$. By $q:G\rightarrow G/G_0=:C(G)$ we denote the continuous quotient homomorphism onto the discrete group of connected components of $G$. Let $\pi: P\rightarrow  U\, (\subset M(\mathscr A_Z))$ be a holomorphic principal $G$-bundle  defined on an open cover $\mathfrak U=(U_i)_{i\in I}$ of $U$ by a cocycle $g=\{g_{ij}\in \mathcal O (U_i\cap U_j, G)\}_{i,j\in I}$. By $P_{C(G)}$ we denote the principal bundle on $U$ with discrete fibre $C(G)$ defined on cover $\mathfrak U$ by  locally constant cocycle $q(g)=\bigl\{q(g_{ij})\in C(U_i\cap U_j, C(G))\bigr\}_{i,j\in I}$. Let $K\subset U$ be  a compact subset.

The next result used in the applications in particular characterizes trivial holomorphic bundles on $M(\mathscr A_Z)$.

\begin{Th}\label{te1.4}
$P$ is trivial over a neighbourhood of $K$ if and only if the associated bundle $P_{C(G)}$ is topologically trivial over $K$. 
\end{Th}
\begin{C}\label{cor1.5}
$P$ is trivial over a neighbourhood of $K$ in one of the following cases:
\begin{itemize}
\item[(1)]
Group $G$ is connected;
\item[(2)]
$U= M(H^\infty)$ and images of all maps $q(g_{ij})$ belong to a finite subgroup of  $C(G)$ (e.g., this is true if $G$ has finitely many connected components);
\item[(3)]
$P$ is trivial over $K$ in the category of topological principal $G$-bundles.
\end{itemize}
\end{C}
In turn, not all principal bundles
on $M(\mathscr A_Z)$ are trivial:
\begin{Proposition}\label{prop1.6}
Let $G$ be a complex Banach Lie group such that $C(G)$ has a nontorsion element. Then there exists a nontrivial holomorphic principal $G$-bundle on $M(\mathscr A_Z)$. 
\end{Proposition}

Let $\mathscr P_G^{\mathcal O}$ and $\mathscr P_G^C$ be the sets of isomorphism classes of holomorphic and topological principal $G$-bundles on $M(\mathscr A_Z)$, respectively. 

The following two results (analogous to the classical results of Grauert \cite{Gra} and Bungart \cite{Bu}) show that the natural map
$i:\mathscr P_G^{\mathcal O}\rightarrow \mathscr P_G^C$ induced by inclusion of the category of holomorphic principal bundles on $M(\mathscr A_Z)$ to the category of topological ones is a bijection. These constitute the Oka Principle for holomorphic principal bundles on $M(\mathscr A_Z)$.
\begin{Th}[Injectivity of $i$]\label{teo1.7}
If two holomorphic principal $G$-bundles on $M(\mathscr A_Z)$ are isomorphic as  topological bundles, they are holomorphically isomorphic.
\end{Th}

\begin{Th}[Surjectivity of $i$]\label{teo2.9}
Each topological principal $G$-bundle on $M(\mathscr A_Z)$ is  isomorphic to a holomorphic one.
\end{Th}

Let $A\in\mathscr A_Z$, i.e., $A=H_I^\infty$ for some closed ideal $I\subset H^\infty$ such that ${\rm hull}(I)=Z$. Recall that we assumed that $M(A)$ coincides with $M(\mathscr A_Z)$.

Let $D$ be the set of all finite subsets of $A$ directed by inclusion. If $\alpha=\{f_1,\dots,f_n\}\in D$   we let $A_\alpha$ be the unital closed subalgebra of $A$ generated by $\alpha$. By $M(A_\alpha)$ we denote the maximal ideal space of $A_\alpha$. It is naturally identified with the polynomially convex hull of the image of $F_\alpha: M(\mathscr A_Z)\rightarrow \Co^{n}$, $F_\alpha(x):=(\hat f_1(x),\dots, \hat f_n(x))$ (here $\hat{\,}$ is the Gelfand transform for algebra $H_{I(Z)}^\infty$). If $\alpha,\beta\in D$ with $\alpha\supseteq\beta$, then linear map $F_\beta^\alpha:\Co^{\#\alpha}\rightarrow\Co^{\#\beta}$, $\Co^{\#\alpha}\ni (z_1,\dots, z_{\#\alpha})\mapsto (z_1,\dots, z_{\#\beta})\in\Co^{\#\beta}$, sends $M(A_\alpha)$ to $M(A_\beta)$. Thus we obtain the inverse system of compacta $\{M(A_\alpha), F_\beta^\alpha\}$ whose limit is naturally identified with $M(\mathscr A_Z)$ and the limit projections coincide with maps $F_\alpha$
(see, e.g., \cite{Ro} for details). 

Let $P_i$ be holomorphic principal $G$-bundles defined on open neighbourhoods $O_i\Subset\Co^{\#\alpha}$ of $M(A_\alpha)$, $i=1,2$. We say that $P_1$ and $P_2$ are isomorphic if they are holomorphically isomorphic on an open neighbourhood $O\subset O_1\cap O_2$ of $M(A_\alpha)$.
By $(\mathscr P^{\mathcal O}_G)_\alpha$ we denote the set of isomorphism classes of holomorphic principal $G$-bundles defined on neighbourhoods of $M(A_\alpha)$. Projections $F_\beta ^\alpha$ induce maps $(\mathscr F_\beta^\alpha)^*:  (\mathscr P^{\mathcal O}_G)_\beta\rightarrow (\mathscr P^{\mathcal O}_G)_\alpha$ assigning to the isomorphism class of a bundle $P$ the isomorphism class of its pullback $(F_\beta^\alpha)^*P$. Thus we obtain the direct system of sets $\{(\mathscr P^{\mathcal O}_G)_\beta,(\mathscr F_\beta^\alpha)^*\}$. Similarly, limit projection $F_\alpha$ induces a map $\mathscr F_\alpha^*:  (\mathscr P^{\mathcal O}_G)_\alpha\rightarrow \mathscr P^{\mathcal O}_G$ assigning to the isomorphism class of a bundle $P$ the isomorphism class of its pullback $F_\alpha^*P$. Since $\mathscr F_\beta^*=\mathscr F_\alpha^*\circ (\mathscr F_\beta^\alpha)^*$ for all $\alpha\supseteq\beta$ in $D$, the family of maps $\{\mathscr F_\alpha^*\}_{\alpha\in D}$ induces a map $\mathscr F_A$ of the direct limit $\displaystyle\lim_{\longrightarrow}(\mathscr P^{\mathcal O}_G)_\alpha$ to $\mathscr P^{\mathcal O}_G$. 
\begin{Th}\label{teo1.10}
Map $\mathscr F_A: \displaystyle \lim_{\longrightarrow}(\mathscr P^{\mathcal O}_G)_\alpha\rightarrow \mathscr P^{\mathcal O}_G$ is a bijection.
\end{Th}
In particular, 
\[
\mathscr P^{\mathcal O}_G=\bigcup_{\alpha\in D}\mathscr F_\alpha^*((\mathscr P^{\mathcal O}_G)_\alpha).
\]
The statement consists of two parts:\smallskip

(1) ({\em Surjectivity of $\mathscr F_A$}). For each holomorphic principal $G$-bundle  $P$ on $M(\mathscr A_Z)$ there exist $\alpha\in D$ and a holomorphic principal $G$-bundle $\tilde P$ defined on a neighbourhood of $M(A_\alpha)$ such that bundles $P$ and $F_\alpha^*\tilde P$ are holomorphically isomorphic. \smallskip

(2)  ({\em Injectivity of $\mathscr F_A$}). If holomorphic principal $G$-bundles $P_1,P_2$ defined on a neighbourhood of $M(A_\beta)$ are such that  $F_\beta^*P_1$ and $F_\beta^*P_2$ are holomorphically isomorphic bundles, then there exist $\alpha\supseteq\beta$ and a neighbourhood $U$ of $M(A_\alpha)$ such that  bundles $(F_\beta^\alpha)^*P_1$ and $(F_\beta^\alpha)^*P_2$ are defined on  $U$ and  holomorphically isomorphic.

\sect{Oka Principle for Maps to Complex Homogeneous Spaces}

\subsect{Holomorphic Maps from ${\mathbf M(\mathscr A_Z)}$ to Banach Homogeneous Spaces}
Let $K\subset M(\mathscr A_Z)$ be a compact subset and $X$ a complex Banach manifold. By $\mathcal O(K,X)\subset C(K,X)$ we denote the set of continuous maps holomorphic on neighbourhoods of $K$ and by $\mathcal A(K,X)$ the closure of $\mathcal O(K,X)$ in the topology of uniform convergence of $C(K,X)$.
If $G$ is a complex Banach Lie group with Lie algebra $\mathfrak g$ and exponential map $\exp_G:\mathfrak g\rightarrow G$, set $\mathcal A(K, G)$ is a complex Banach Lie group with respect to pointwise product of maps with Lie algebra  $\mathcal A(K, \mathfrak g)$  and the exponential map being the composition of $\exp_G$ with elements of $\mathcal A(K, \mathfrak g)$.

The following result interesting in its own right will be used in the subsequent proofs.
\begin{Th}\label{teo3.4}
If $G$ is simply connected, then $\mathcal A(K, G)$ is path-connected.
\end{Th}

To formulate further results we invoke the definition of a complex Banach homogeneous space (see, e.g., \cite[Sec.\,1]{R}).\smallskip

Suppose a complex Banach Lie group $G$ with unit $e$ acts {\em holomorphically} and {\em transitively} on a complex Banach manifold $X$, i.e., there is a holomorphic map $(g,p)\mapsto g\cdot p$ of $G\times X$ onto $X$ satisfying
$g_1\cdot (g_2\cdot p)=(g_1g_2)\cdot p$ and $e\cdot p=p$ for all $g_1,g_2\in G$, $p\in X$,
and  for each pair $p,q\in X$ there is some $g\in G$ with $g\cdot p=q$.  For $p\in X$ consider holomorphic surjective map $\pi^p: G\rightarrow X$, 
$\pi^p(g):=g\cdot p$, $g\in G$. Let
$G_{e}\, (\cong\mathfrak g)$ and $X_p$ be tangent spaces of $G$ at $e$ and $X$ at $p$. By $d\pi_e^p: G_e\rightarrow X_p$ we denote the differential of $\pi^p$ at $e$.\smallskip

 $X$ is called a {\em Banach homogeneous space} under the action of $G$ if
there is some $p\in X$ such that
${\rm ker}\, d\pi_e^p$ is a complemented subspace of $G_e$ and $d\pi_e^p: G_e\rightarrow X_p$ is surjective.\smallskip

By $[M(\mathscr A_Z),X]$  we denote the set of homotopy classes of continuous maps from $M(\mathscr A_Z)$ to $X$. Maps $f_0,f_1\in  \mathcal O(M(\mathscr A_Z),X)$ are said to be homotopic in $\mathcal O(M(\mathscr A_Z),X)$ if there is a homotopy $F: M(\mathscr A_Z)\times [0,1]\rightarrow X$ connecting them such that $f_t:=F(t,\cdot)\in \mathcal O(M(\mathscr A_Z),X)$ for all $t\in [0,1]$. This homotopy relation is an equivalence relation with the set of equivalence classes $[M(\mathscr A_Z),X]_{\mathcal O}$
consisting of path-connected components of $\mathcal O(M(\mathscr A_Z),X)$. Next, embedding $\mathcal O(M(\mathscr A_Z),X)\hookrightarrow C(M(\mathscr A_Z),X)$ induces a map \penalty-10000 $\mathscr O: [M(\mathscr A_Z),X]_{\mathcal O}\rightarrow [M(\mathscr A_Z),X]$.

\begin{Th}\label{teo3.5}
Let $X$ be a complex Banach homogeneous space. Then $\mathscr O$ is a bijection.
\end{Th}
The statement consists of two parts:

\begin{itemize}
 \item[(1)] ({\em Injectivity of $\mathscr O$}).  If two  maps in $  \mathcal O(M(\mathscr A_Z),X)$ are homotopic in $C(M(\mathscr A_Z),X)$,  
they are homotopic in $\mathcal O(M(\mathscr A_Z),X)$.
\item[(2)] ({\em Surjectivity of $\mathscr O$}).
Every map in $C(M(\mathscr A_Z),X)$ is homotopic to a map in \penalty-10000 $\mathcal O(M(\mathscr A_Z),X)$.
\end{itemize}

Next, recall that a path-connected topological space $X$ is  $n$-{\em simple} if for each $x\in X$ the fundamental group $\pi_1(X,x_0)$ acts trivially on the $n$-homotopy group $\pi_n(X,x)$ (see, e.g., \cite[Ch.\,IV.16]{Hu1} for the corresponding definitions and results).  For instance, $X$ is $n$-simple if group $\pi_n(X)=0$ and
$1$-simple if and only if group $\pi_1(X)$ is abelian. Also, it is worth noting that every path-connected topological group is $n$-simple for all $n$.
\begin{C}\label{cor3.6}
Let $X$ be a complex Banach homogeneous space $n$-simple for all $n\le 2$. Then there is a natural one-to-one correspondence between elements of $[M(\mathscr A_Z),X]_{\mathcal O}$ and the \v{C}ech cohomology group $H^1(M(\mathscr A_Z),\pi_1(X))$. In particular, 
if $X$ is simply connected, space $\mathcal O(M(\mathscr A_Z),X)$ is path-connected.
\end{C}
\begin{R}\label{rem3.7}
{\rm (1) The correspondence of the corollary is described in Remark \ref{rem10.1}.\smallskip

\noindent (2) Group $H^1(M(\mathscr A_Z),\Z)$ is always nontrivial, see  Lemma \ref{lem7.4}.}
\end{R}

Recall that for each $A\in\mathscr A_Z$ space $M(\mathscr A_Z)$ can be presented as the inverse limit of the inverse limit system $\{M(A_\alpha),F_\beta^\alpha\}$, $\alpha\in D$, see Section~2.2. Two maps $M(A_\alpha)\rightarrow X$   holomorphic in a neighbourhood $U$ of $M(A_\alpha)$ are said to be {\em homotopic in} $\mathcal O(M(A_\alpha),X)$ if their restrictions to 
a neighbourhood $V\subset U$  of $M(A_\alpha)$ can be joined by a  path in $\mathcal O(V,X)$. This homotopy relation is an equivalence relation with the set of equivalence classes denoted by $[M(A_\alpha), X]_{\mathcal O}$.

Projections $F_\beta ^\alpha$ induce maps $(\mathfrak F_\beta^\alpha)^*:  [A_\beta, X]_{\mathcal O}\rightarrow [A_\alpha, X]_{\mathcal O}$ assigning to the homotopy class of a  map $f\in\mathcal O(M(A_\beta), X)$ the homotopy class of its pullback $(F_\beta^\alpha)^*f$. Thus we obtain the direct system of sets $\bigl\{[M(A_\beta),X]_{\mathcal O}, (\mathfrak F_\beta^\alpha)^*\bigr\}$. Similarly, limit projection $F_\alpha$ induces a map $\mathfrak F_\alpha^*:  [M(A_\alpha), X]_{\mathcal O}\rightarrow [M(\mathscr A_Z),X]$ assigning to the homotopy class of $f\in\mathcal O(M(A_\alpha), X)$ the homotopy class of its pullback $F_\alpha^*f$. Since $\mathfrak F_\beta^*=\mathfrak F_\alpha^*\circ (\mathfrak F_\beta^\alpha)^*$ for all $\alpha\supseteq\beta$ in $D$, the family of maps $\{\mathfrak F_\alpha^*\}_{\alpha\in D}$ induces a map $\mathfrak F_A$ of the direct limit $\displaystyle\lim_{\longrightarrow}\,[M(A_\alpha), X]_{\mathcal O}$ to $[M(\mathscr A_Z),X]_{\mathcal O}$. 
\begin{Th}\label{teo3.6}
Let $X$ be a complex Banach homogeneous space. Then
$\mathfrak F_A$ is a bijection.
\end{Th}
In particular, 
\[
[M(\mathscr A_Z),X]_{\mathcal O}=\bigcup_{\alpha\in D}\mathfrak F_\alpha^*\bigl([M(A_\alpha), X]_{\mathcal O}\bigr).
\]
To prove the result we establish the following:\smallskip

(1) ({\em Surjectivity of $\mathfrak F_A$}). For each holomorphic map $f: M(\mathscr A_Z)\rightarrow X$ there exist $\alpha\in D$ and a holomorphic map into $X$  defined on a neighbourhood of $M(A_\alpha)$ such that maps $f$ and $F_\alpha^*f$ are homotopic in $\mathcal O(M(\mathscr A_Z),X)$. \smallskip

(2)  ({\em Injectivity of $\mathfrak F_A$}). If holomorphic maps $f_1,f_2$ into $X$ defined on a neighbourhood of $M(A_\beta)$ are such that  $F_\beta^*f_1$ and $F_\beta^*f_2$ are homotopic in $\mathcal O(M(\mathscr A_Z),X)$, then there exist $\alpha\supseteq\beta$ and a neighbourhood $U$ of $M(A_\alpha)$ such that  maps $(F_\beta^\alpha)^*f_1$ and $(F_\beta^\alpha)^*f_2$ are defined on  $U$ and  homotopic in $\mathcal O(U,X)$.

Let $X$ be a complex Banach homogeneous manifold under the action of a complex Banach Lie group $G$. It is known, see, e.g., \cite[Prop.\,1.4]{R}, that the stabilizer of a point  $p\in X$,
$G(p):=\{g\in G\, :\, \pi^p(g)=p\}$, is a closed complex Banach Lie subgroup of $G$ and stabilizers of different points are conjugate in $G$ by inner automorphisms.

The following result will be used in  applications.
\begin{Th}\label{teo3.7}
Let $X$ be a complex Banach homogeneous space under the action of a complex Banach Lie group $G$. Assume that group $G(p)\subset G$, $p\in X$, is connected. Then for every $f\in\mathcal O(M(\mathscr A_Z),X)$ and each $p\in X$ there is a map $\tilde f_p\in\mathcal O(M(\mathscr A_Z),G)$ such that $f(x)=\tilde f_p(x)\cdot p$ for all $x\in M(\mathscr A_Z)$.
\end{Th}
\subsection{Nonlinear Approximation and Interpolation Problems}
A compact subset $K\subset M(\mathscr A_Z)$ is called {\em holomorphically convex} if for any $x\notin K$ there is $f\in \mathcal O(M(\mathscr A_Z))$ such that
\[
\max_{K}|f|<|f(x)|.
\]
Note that for a natural number $l$ any subset of $M(\mathscr A_Z)$ of the form
\[
\{x\in M(\mathscr A_Z)\, :\, \max_{1\le j\le l}|f_j(x)|\le 1,\ f_j\in\mathcal O(M(\mathscr A_Z)),\ 1\le j\le l\}
\]
is holomorphically convex and each holomorphically convex  $K\subset M(\mathscr A_Z)$ is intersection of such subsets.

Let $U$ be an open neighbourhood of a holomorphically convex set $K\subset M(\mathscr A_Z)$ and $X$ be a complex Banach homogeneous space.
The following result is a nonlinear analog of the Runge approximation theorem.
\begin{Th}\label{nonlinrunge}
\begin{itemize}
\item[(1)]
Suppose a  map in $\mathcal O(U,X)$ can be uniformly approximated on  $K$ by maps in $C(M(\mathscr A_Z),X)$. Then
it can be uniformly approximated on  $K$ by maps in $\mathcal O(M(\mathscr A_Z),X)$.
\item[(2)]
If $X$ is simply connected, then each map in $\mathcal O(U,X)$ can be uniformly approximated on  $K$ by maps in $\mathcal O(M(\mathscr A_Z),X)$.
\end{itemize}
\end{Th}

Let $Z\subset M(H^\infty)$ be a hull and $U$ be an open neighbourhood of $Z$.
\begin{Th}\label{te2.2}
Let $X$ be a complex Banach homogeneous space and $f\in \mathcal O(U, X)$. 
\begin{itemize}
\item[(1)]
Restriction $f|_Z\in C(Z,X)$ extends to a map in
$\mathcal O(M(H^\infty), X)$  iff it extends to a map in $C(M(H^\infty),X)$.
\item[(2)]
$f|_Z$ extends to a map in $\mathcal O(M(H^\infty), X)$ if $X$ is $n$-simple for all $n\le 2$.
\end{itemize}
\end{Th}
\begin{R}\label{rem3.2}
{\rm 
(1) The quantitative version of the second part of Theorem \ref{te2.2} for $X =\Co$ and $Z$ the zero locus of the Gelfand transform of
a Blaschke product was proved by Carleson \cite{C}. For $X=\Co^*$ the result of part (2) follows from Treil's theorem \cite{T6} via its cohomological interpretation due to Su\'{a}rez \cite[Th.\,1.3]{S1} and the Arens-Royden theorem \cite{A2}, \cite{Ro}. For $X$ a complex Banach space the results of second parts of Theorems \ref{nonlinrunge}, \ref{te2.2} were established by the author \cite[Th.\,1.7,\,1.9]{Br1}, see also Theorems \ref{teo5.2}, \ref{runge} in Section~6 below.\smallskip

\noindent (2) Relation between spaces $\mathcal O(U, X)$ and $\mathcal O(U\cap\Di, X)$  is given by Proposition \ref{prop1.4}: \smallskip

 \noindent  \, ($\circ$) {\em a map $f\in\mathcal O(U\cap\Di,X)$ extends to a map in $\mathcal O(U,X)$ iff there exist open covers $(U_\alpha)_{\alpha\in A}$ of $U$ and  $(V_\beta)_{\beta\in B}$ of $X$ and a map $\tau: A\rightarrow B$   such that for each $\beta \in B$ holomorphic functions on $V_\beta$ separate points and $f(U_\alpha\cap \Di)\Subset V_{\tau(\alpha)}$ for all $\alpha\in A$.}
 
{\em If $X$ is a complex submanifold of a complex Banach space, then it suffices to assume that $f(U\cap\Di)\Subset X$, see \cite[Prop.\,1.3]{Br1}.}
 }
\end{R}
\begin{E}\label{eq3.3}
{\rm Let $F\subset \{z\in\mathbb C\, :\, |z|=1\}$ be a closed subset of Lebesgue measure zero. By $I_F\subset H^\infty$ we denote the ideal generated by all functions of the {\em disk-algebra} $A(\Di)\, (=H^\infty\cap C(\bar{\Di}))$ equal zero on $F$. According to the Rudin-Carleson theorem  for each $f\in C(F)$ there exists a function $\tilde f\in A(\Di)$ such that $\tilde f|_F=f$ and $\|\tilde f\|_{A(\Di)}=\|f\|_{C(F)}$. The map transposed to embedding $A(\Di)\hookrightarrow H^\infty$ induces a continuous surjection of the maximal ideal spaces $p:M(H^\infty)\rightarrow\bar{\Di}\,(=M(A(\Di))$ such that $Z:={\rm hull}(I_F)=p^{-1}(F)$. Let $O\subset\bar{\Di}$ be a relatively open subset containing $F$. Then $U:=p^{-1}(O)$ is an open neighbourhood of $Z$ and $U\cap\Di=O\cap\Di$. Theorem \ref{te2.2} implies the following nonlinear version of the Rudin-Carleson theorem: {\em Let $G$ be a connected complex Banach Lie group and $f\in \mathcal O(U\cap\Di, G)$ satisfy $(\circ)$. Then
there exists a map $\tilde f\in \mathcal O(\Di,G)$ satisfying $(\circ)$ such that $\displaystyle \lim_{z\rightarrow x}\tilde f(z) f(z)^{-1}=1_G$ for all $x\in F$; here $1_G\in G$ is the unit.}
}
\end{E}

\sect{Applications and Examples}
\subsection{Holomorphic Maps of ${\mathbf M(\mathscr A_Z)}$ into Flag Manifolds and Complex Tori}
For basic results of the Lie group theory, see, e.g., \cite{OV}.

Recall that the maximal connected solvable Lie subgroup $B$ of a connected (finite-dimensional) complex Lie group $G$ is called  a {\em Borel subgroup}. A Lie subgroup $P\subset G$ containing some Borel subgroup is called {\em parabolic}. The set of orbits $X=G/P$ under the action of a parabolic subgroup $P\subset G$ on $G$ by right multiplications has the natural structure of a complex homogeneous space and is called the {\em flag manifold}. Since $P$ contains the radical of $G$, one may assume that in the definition of $X$ group $G$ is semisimple.
Each flag manifold is simply connected and thus by Corollary \ref{cor3.6} and  Theorem \ref{teo3.7}  space $\mathcal O(M(\mathscr A_Z),
X)$ {\em is path-connected and consists of all maps of the form $f\cdot p_{0}$, where $f\in\mathcal O(M(\mathscr A_Z),G)$ and $p_0$ is the orbit of the unit of $G$.}
Since $G$ admits a faithful linear representation, we also obtain (see Remark \ref{rem2.5}(2)) that {\em $g\in\mathcal O(\Di\setminus Z,X)$ extends to a map in $\mathcal O(M(\mathscr A_Z),X)$ iff $g=f\cdot p_{0}$ for some $f\in \mathcal O(\Di\setminus Z,G)$ such that $f(\Di\setminus Z)\Subset G$ and for every $h\in\mathcal O(G)$ the Gelfand transform $\widehat{h\circ f}$ is constant on $Z$.}
\begin{E}\label{eq3.9}
{\rm The complex {\em flag  manifold} ${\mathbf F}(d_1,\dots, d_k)$ is the space of all flags of type $(d_1, \dots, d_k)$ in $\Co^{n}$, $n:=d_k$, i.e., of increasing sequences of subspaces
\[
\{0\}=V_{0}\subset V_{1}\subset V_{2}\subset \cdots \subset V_{k}=\Co^n\qquad {\rm with}\qquad {\rm dim}_\Co\, V_i = d_i.
\]
It has the natural structure of a simply connected compact complex homogeneous space under the action of group $GL_n(\Co)$. The stabilizer of a flag is the connected subgroup of $GL_n(C)$ isomorphic to the group of nonsingular block upper triangular matrices with dimensions of blocks $d_{i}-d_{i-1}$ with $d_0:=0$. Thus space $\mathcal O(M(\mathscr A_Z),
{\mathbf F}(d_1,\dots, d_k))$ is path-connected and  consists of all maps of the form $f\cdot p_{0}$, where $f\in\mathcal O(M(\mathscr A_Z),GL_n(\Co))$ and $p_0$ is the flag of subspaces $V_i^0:=\{z=(z_1,\dots, z_n)\in\Co^n\, :\, z_1=z_2=\cdots =z_{n-d_i}=0\}$, $0\le i\le k-1$, $V_k:=\Co^n$. Here the image of $V_i^0$ under $f(x)$, $x\in M(\mathscr A_Z)$, is the subspace of $\Co^n$ generated by the last $d_i$ column vectors of matrix $f(x)$. 

For example, in the case of ${\mathbf F}(1,n)$, the $(n-1)$-dimensional complex projective space,  we obtain that space $O(M(\mathscr A_Z),{\mathbf F}(1,n))$ consists of all maps $f$ of the form \penalty-10000 $f(x)=[f_1(x):f_2(x):\dots : f_n(x)]$, $x\in M(\mathscr A_Z)$,  where $f_i\in\mathcal O(M(\mathscr A_Z))$ and satisfy the corona condition  on $M(\mathscr A_Z)$ (cf. \eqref{e1.1}). (Note that the column vector composed of such $f_i$  extends automatically to an invertible matrix with entries in $\mathcal O(M(\mathscr A_Z))$ because of  projective freeness of this algebra, see Remark \ref{rem1.2}(2).)}
\end{E}

Let $\Co\mathbb T^n$ be the complex torus obtained as the quotient of $\Co^n$ by a lattice $\Gamma\, (\cong\Z^{2n})$ (acting on $\Co^n$ by translations). By $\pi:\Co^n\rightarrow \Co\mathbb T^n$ we denote the holomorphic projection map.
In the next result we describe the structure of space $\mathcal O(M(\mathscr A_Z),\Co\mathbb T^n)$. 

Recall that a holomorphic function $f$ on $\Di$ belongs to the BMOA space if it has boundary values a.e. on $\mathbb T:=\{z\in\Co\, :\, |z|=1\}$ satisfying
\begin{equation}\label{bmoa}
\sup_{I}\frac{1}{|I|}\int_I |f(\theta)-f_I|\,d\theta<\infty,
\end{equation}
where $I\Subset \mathbb T$ is an open arc of arclength $|I|$  and
$f_I:=\frac{1}{|I|}\int_I f(\theta)d\theta$.
\begin{Th}\label{bmoa1}
A  map $f\in \mathcal O(\Di,\Co\mathbb T^n)$ extends continuously to $M(H^\infty)$ iff it is  factorized as $f=\pi\circ \tilde f$ for some $\tilde f=(\tilde f_1,\dots, \tilde f_n)\in\mathcal O(\Di,\Co^n)$ with all $\tilde f_i\in BMOA$. 
\end{Th}
\begin{R}
{\rm (1) Let $Q_Z:M(H^\infty)\rightarrow M(\mathscr A_Z)$ be the continuous map transposed to embedding $H_I^\infty\hookrightarrow H^\infty$. Clearly, $f\in\mathcal O(\Di\setminus Z,\Co\mathbb T^n)$ extends continuously to $M(\mathscr A_Z)$ iff $Q_Z^*f$ extends to a map $C(M(H^\infty),\Co\mathbb T^n)$ constant on $Z$.\smallskip

\noindent (2) It follows from the proof of Theorem \ref{teo3.5} that if $f,g\in \mathcal O(M(\mathscr A_Z),\Co\mathbb T^n)$ are homotopic, then there is a map $h\in\mathcal O(M(\mathscr A_Z),\Co^n)$ such that $f(x)=g(x)+(\pi\circ h)(x)$ for all $x\in M(\mathscr A_Z)$. (Here $+$ stands for addition on $\Co\mathbb T^n$ induced from that on $\Co^n$.)
Space $\mathcal O(M(\mathscr A_Z),\Co\mathbb T^n)$ is an abelian complex Banach Lie group under the pointwise addition of maps.  Thus, according to Corollary \ref{cor3.6}, $H^1(M(\mathscr A_Z),\Z^{2n})$ is naturally isomorphic to the quotient of group $\mathcal O(M(\mathscr A_Z),\Co\mathbb T^n)$ by the connected component of its unit $\pi\bigl(\mathcal O(M(\mathscr A_Z),\Co^n)\bigr)$.
}
\end{R}

\subsection{Structure of Spaces of ${\mathbf H^\infty}$ Idempotents}
Let ${\mathfrak A}$ be a complex Banach algebra with unit $1_{\mathfrak A}$.
By ${\rm id}\, \mathfrak A=\{a\in \mathfrak A\, :\, a^2=a\}$ we denote the set of idempotents of $\mathfrak A$. It is a closed complex Banach submanifold of $\mathfrak A$ which is a discrete union of connected complex Banach homogeneous spaces, see \cite[Cor.\,1.7]{R}. Specifically,  let $\mathfrak A^{-1}_0$ be the connected component 
containing the unit $1_{\mathfrak A}$ of the complex Banach Lie group $\mathfrak A^{-1}$ of invertible elements of $\mathfrak A$. Then each connected component of ${\rm id}\, \mathfrak A$ is a complex Banach homogeneous space under the action $\mathfrak A^{-1}_0\times {\rm id}\, \mathfrak A\rightarrow {\rm id}\, \mathfrak A$ by similarity transformations $(a,p)\mapsto apa^{-1}$.

Analogously for unital complex  Banach algebras $C(M(\mathscr A_Z), \mathfrak A)$ and $\mathcal O(M(\mathscr A_Z), \mathfrak A)$ their sets of idempotents ${\rm id}\, C(M(\mathscr A_Z), \mathfrak A)$ and  ${\rm id}\, \mathcal O(M(\mathscr A_Z), \mathfrak A)$ coincide with $C(M(\mathscr A_Z),{\rm id}\, \mathfrak A)$ and  $\mathcal O(M(\mathscr A_Z),{\rm id}\, \mathfrak A)$ (because $M(\mathscr A_Z)$ is connected) and are discrete unions of connected complex Banach homogeneous spaces.
 Each connected component of $C (M(\mathscr A_Z),{\rm id}\, \mathfrak A)$ or $\mathcal O(M(\mathscr A_Z),{\rm id}\, \mathfrak A)$ is  the complex homogeneous space under the action of the complex Banach Lie group 
$C(M(\mathscr A_Z), \mathfrak A)_0^{-1}$ or 
 $\mathcal O(M(\mathscr A_Z), \mathfrak A)_{0}^{-1}$ by similarity transformations. 
\begin{Th}\label{teo4.1}
\begin{itemize}
\item[(a)]
Embedding $\mathcal O(M(\mathscr A_Z),{\rm id}\, \mathfrak A)\hookrightarrow C(M(\mathscr A_Z),{\rm id}\, \mathfrak A)$ induces bijection between connected components of these complex Banach manifolds.\smallskip
\item[(b)]
If the stabilizer  $\mathfrak A_0^{-1}(p)\subset \mathfrak A_0^{-1}$ of a point $p\in {\rm id}\, \mathfrak A$ is connected, then for each $f\in \mathcal O(M(\mathscr A_Z),{\rm id}\, \mathfrak A)$ whose image belongs to the connected component  
containing $p$ there is $g\in \mathcal O(M(\mathscr A_Z),\mathfrak A_0^{-1})$ such that $f=gpg^{-1}$. \smallskip

\item[(c)]
If $A\in\mathscr A_Z$, then for each $f\in\mathcal O(M(\mathscr A_Z),{\rm id}\,\mathfrak A)$ 
there are $\alpha\in D$, a map $h\in \mathcal O(M(\mathscr A_Z), {\rm id}\,\mathfrak A)$ uniformly approximated by maps of the algebraic tensor product $\widehat{A}_\alpha\otimes\mathfrak A$  and a map $g\in \mathcal O(M(\mathscr A_Z),\mathfrak A_0^{-1})$ such that
$gfg^{-1}=h$.

\noindent Here $\widehat{A}_\alpha$ is the image of $A_\alpha$ under the Gelfand transform $\,\hat\, : A\rightarrow C(M(\mathscr A_Z))$.
\end{itemize}
\end{Th}
\begin{R}\label{rem4.2}
{\rm (1) If $\mathfrak A_0^{-1}$ is simply connected and
the hypothesis of part (b) holds for all $p\in {\rm id}\, \mathfrak A$, then, since in this case $\mathcal O(M(\mathscr A_Z),\mathfrak A_0^{-1})$ coincides with 
$\mathcal O(M(\mathscr A_Z), \mathfrak A)_{0}^{-1}$ (cf. Theorem \ref{teo3.4}),
embedding ${\rm id}\, \mathfrak A\hookrightarrow \mathcal O(M(\mathscr A_Z),{\rm id}\, \mathfrak A)$ assigning to each $p$ the map of the constant value $p$
induces bijection between sets of connected components of ${\rm id}\, \mathfrak A$ and $\mathcal O(M(\mathscr A_Z),{\rm id}\, \mathfrak A)$.\smallskip

\noindent (2) Part (c) follows from Theorem \ref{teo3.6} and the fact that  $F_\alpha^*\bigl(\mathcal O(M(A_\alpha),\mathfrak A)\bigr)$ coincides with $A_\alpha\otimes_{\varepsilon}\mathfrak A$, the injective tensor product of these algebras (cf. the arguments in Section~12). \smallskip
 
 \noindent  (3) Let $H_Z^\infty({\mathfrak A})$  be the Banach algebra of bounded maps $F\in \mathcal O(\Di,\mathfrak A)$ equipped with  pointwise product and addition  with norm $\|F\|:=\sup_{z\in\Di}\|F(z)\|_{\mathfrak A}$ such that for each $\varphi\in\mathfrak A^*$ the Gelfand transform of $\varphi\circ F\in H^\infty$ is constant on $Z$. Let $H_{Z\, {\rm comp}}^\infty({\mathfrak A})\subset H_Z^\infty({\mathfrak A})$ be the closed subalgebra of maps with relatively compact images.  (If $Z$ is empty or a single point we omit $_Z$ in the indices.) Each map in $H_{Z\, {\rm comp}}^\infty({\mathfrak A})$ admits a continuous extension to $ M(H^\infty)$ constant on $Z$, see, e.g., \cite[Prop.\,1.3]{Br1}. This implies that algebras $H_{Z\, {\rm comp}}^\infty({\mathfrak A})$ and $\mathcal O(M(\mathscr A_Z),\mathfrak A)$ are isometrically isomorphic. Thus Theorem \ref{teo4.1} 
describes the structure of set ${\rm id}\,H_{Z\, {\rm comp}}^\infty({\mathfrak A})$. 
}
\end{R}
\begin{E}\label{ex4.3}
{\rm Let $L(X)$ be the Banach algebra of bounded linear operators on a complex Banach space $X$ equipped with the operator norm. By $I_X\in L(X)$ we denote the identity operator and by $GL(X)\subset L(X)$ the set of invertible bounded linear operators on $X$. Clearly, $L(X)^{-1}:=GL(X)$.  By $GL_0(X)\subset GL(X)$ we denote the connected component of $I_X$.
Each $P\in {\rm id}\, L(X)$ determines a direct sum decomposition $X=X_0\oplus X_1$, where $X_0:={\rm ker}\,(P)$ and $X_1:={\rm ker}\,(I_X-P)$. It is easily seen that the stabilizer
$GL_0(X)(P)\subset GL_0(X)$ consists of all operators $B\in GL_0(X)$ such that $B(X_k)\subset X_k$, $k=1,2$. In particular, restrictions of operators in $GL_0(X)(P)$ to $X_k$ determine a monomorphism of complex Banach Lie groups $S_P: GL_0(X)(P)\rightarrow GL(X_1)\oplus GL(X_2)$. Moreover, $S_P$ is an isomorphism if $GL(X_i)$, $i=1,2$ are connected. Now, Theorem \ref{teo4.1} leads to the following statement:}\smallskip

\noindent $(1)$ Suppose $P\in {\rm id}\, L(X)$ is such that groups $GL(X_1)$ and $GL(X_2)$ are connected. Then for each $F\in \mathcal O(M(\mathscr A_Z),{\rm id}\, L(X))$ whose image belongs to the connected component containing $P$ there is $G\in \mathcal O(M(\mathscr A_Z), GL_0(X))$ such that $GFG^{-1}=P$.\smallskip

{\rm In particular, the result is valid for $X$ being one of the spaces: a finite-dimensional space, a Hilbert space, $c_0$ or $\ell^p$, $1\le p\le\infty$. Indeed, $GL(X)$ is connected if ${\rm dim}_{\Co}X<\infty$ and contractible for  other listed above spaces, see, e.g., \cite{Mi} and references therein. Moreover, each subspace of a Hilbert space is Hilbert, and each
infinite-dimensional complemented subspace of $X$ being either $c_0$ or $\ell^p$, $1\le p\le\infty$, is isomorphic to $X$, see  \cite{Pe}, \cite{Lin}. This gives the required conditions. 

It is worth noting that there are complex Banach spaces $X$ for which groups $GL(X)$ are not connected, see, e.g., \cite{D}. \smallskip

In turn, part (c) of the theorem implies in this case:}\smallskip

\noindent $(2)$ Let $A\in\mathscr A_Z$. For each $F\in \mathcal O(M(\mathscr A_Z),{\rm id}\, L(X))$ there exist a finitely generated unital subalgebra  $B\subset A$ and maps $H\in {\rm id}\, \widehat B\otimes_{\varepsilon}L(X)$, $G\in\mathcal O(M(\mathscr A_Z),GL_0(X))$ such that $GFG^{-1}=H$.

\end{E}

\subsection{Extension of Operator-valued ${\mathbf H^\infty}$ Functions}
Let $\mathfrak A$ be a complex Banach algebra with unit $1_{\mathfrak A}$. Let
$\mathfrak A^{-1}_l=\{a\in\mathfrak A\, :\ \exists\, b\in\mathfrak A\ {\rm such\ that}\ ba=1_{\mathfrak A}\}$ be the set of left-invertible elements of $\mathfrak A$. Clearly, $\mathfrak A^{-1}_l$ is an open subset of $\mathfrak A$ and complex Banach Lie group $\mathfrak A^{-1}_0$ acts holomorphically on each connected component of $\mathfrak A^{-1}_l$ by left multiplications: $(g,a)\mapsto ga$.
 In fact, we have
\begin{Proposition}\label{prop4.4}
Each connected component of $\mathfrak A^{-1}_l$ is a complex Banach homogeneous space under the action of $\mathfrak A_0^{-1}$.
\end{Proposition}
Next, similarly to Theorem \ref{teo4.1} the following result holds.
\begin{Th}\label{teo4.5}
\begin{itemize}
\smallskip
\item[(a)]
If the stabilizer  $\mathfrak A_0^{-1}(p)\subset \mathfrak A_0^{-1}$ of a point $p\in \mathfrak A_l^{-1}$ is connected, then for each $f\in \mathcal O(M(\mathscr A_Z), \mathfrak A_l^{-1})$ with image in the connected component containing $p$ there is $g\in \mathcal O(M(\mathscr A_Z),\mathfrak A_0^{-1})$ such that $f=gp$. \smallskip
 \item[(b)]
If $A\in\mathscr A_Z$, then for each $f\in\mathcal O(M(\mathscr A_Z),\mathfrak A_l^{-1})$ 
there are $\alpha\in D$, a map $h\in \mathcal O(M(\mathscr A_Z), \mathfrak A_l^{-1})$  uniformly approximated by maps of the algebraic tensor product $\widehat{A}_\alpha\otimes\mathfrak A$  and a map $g\in \mathcal O(M(\mathscr A_Z),\mathfrak A_0^{-1})$ such that
$gf=h$.
\end{itemize}
\end{Th}
The analog of Theorem \ref{teo4.1}(a) for Banach algebras $C(M(\mathscr A_Z), \mathfrak A)$ and $\mathcal O(M(\mathscr A_Z), \mathfrak A)$ is also true.

\begin{Th}\label{teo4.6}
\[
C(M(\mathscr A_Z), \mathfrak A)_l^{-1}=C(M(\mathscr A_Z), \mathfrak A_l^{-1})\quad {\rm and }\quad \mathcal O(M(\mathscr A_Z), \mathfrak A)_l^{-1}=\mathcal O(M(\mathscr A_Z), \mathfrak A_l^{-1}).
\]
Thus every connected component of 
$C(M(\mathscr A_Z), \mathfrak A_l^{-1})$ or
$\mathcal O(M(\mathscr A_Z), \mathfrak A_l^{-1})$ is a complex Banach homogeneous space under the group action of 
$C(M(\mathscr A_Z), \mathfrak A)^{-1}_0$
or   $\mathcal O(M(\mathscr A_Z), \mathfrak A)^{-1}_0$ by left multiplications, and embedding 
\[
\mathcal O(M(\mathscr A_Z), \mathfrak A_l^{-1})\hookrightarrow C(M(\mathscr A_Z),\mathfrak A_l^{-1})
\]
induces bijection between connected components of these complex Banach manifolds.
\end{Th}

\begin{R}\label{rem4.7}
{\rm (1) If $\mathfrak A_0^{-1}$ is simply connected and the hypothesis of part (b) of Theorem \ref{teo4.5} holds for all $p\in \mathfrak A_l^{-1}$, then embedding $\mathfrak A_l^{-1}\hookrightarrow \mathcal O(M(\mathscr A_Z),\mathfrak A_l^{-1})$ assigning to each $p$ the map of the constant value $p$ induces bijection between
sets of connected components of $\mathfrak A_l^{-1}$ and $ \mathcal O(M(\mathscr A_Z),\mathfrak A_l^{-1})$, (cf. Remark \ref{rem4.2}(1)).\smallskip

\noindent (2) While the first identity of Theorem \ref{teo4.6} for algebra $C(M(\mathscr A_Z),\mathfrak A)$ is true in general with a compact Hausdorff space substituted instead of $M(\mathscr A_Z)$, the second  one was previously unknown even for $H^\infty$, see, e.g., \cite{V}.
For $H_{Z\, {\rm comp}}^\infty$, see Remark \ref{rem4.2}(3), the second identity of Theorem \ref{teo4.6} can be reformulated as follows:\smallskip

\noindent ($\star$) {\em A map $F\in H_{Z\,\rm comp}^\infty(\mathfrak A)$ has a left inverse $G\in  H_{Z\,\rm comp}^\infty(\mathfrak A)$ if and only if for every $z\in\Di$ there exists a left inverse $G_z$ of $F(z)$ such that $\sup_{z\in\Di}\|G_z\|_{\mathfrak A}<\infty$.}\smallskip

This result is related to the classical
Sz.-Nagy operator corona problem \cite{SN}:
\begin{Problem}
Let $H_1, H_2$ be separable Hilbert spaces and $F\in H^\infty(L(H_1,H_2))$ be such that for some $\delta>0$ and all $x\in H_1, z\in\Di$, $\|F(z)x\|\ge\delta\|x\|$. Does there exist $G\in H^\infty (L(H_2,H_1))$ such that $G(z)F(z)=I_{H_1}$ for all $z\in\Di$?
\end{Problem}
This problem is of great interest in operator theory (angles between invariant
subspaces, unconditionally convergent spectral decompositions), as well as in control theory. It is also related
to the study of submodules of $H^\infty$ and to many other subjects of analysis, see \cite{N1}, \cite{N2}, \cite{T1}, \cite{T2}, \cite{Vi} and references therein.   In general, the answer is known to be negative (see \cite{T3}, \cite{T4}, \cite{TW}); 
it is positive as soon as ${\rm dim}\, H_1<\infty$ or 
$F$ is a ``small'' perturbation of a left invertible function $F_0\in H^\infty(L(H_1,H_2))$ (e.g., if $F -F_0$ belongs to $H^\infty(L(H_1,H_2))$ with values in the class 
of Hilbert Schmidt operators), see \cite{T5}, or 
$F\in H_{\rm comp}^\infty(L(H_1,H_2))$, see \cite[Th.\,1.5]{Br2} and $(\star)$ above.

It is worth noting that the proof of statement $(\star)$ for $H_{\rm comp}^\infty(\mathfrak A)$ would be shorter if we knew that $H^\infty$ has the {\em Grothendieck approximation property}, cf. \cite[Th.\,2.2]{V}. However, this long-standing problem remains unsolved (for  some developments see, e.g., \cite[Th.\,9]{BR} and \cite[Th.\,1.21]{Br1}).
}
\end{R}

\begin{E}\label{ex4.11}
{\rm Let $L(X)$ be the Banach algebra of bounded linear operators on a complex Banach space $X$.  
Each $A\in L(X)_l^{-1}$ determines 
complemented subspace $X_1:={\rm ran}\,A\subset X$ isomorphic to $X$.
Then the stabilizer
$GL_0(X)(A)\subset GL_0(X)$ of $A$ consists of all operators $B\in GL_0(X)$ such that $B|_{X_1}=I_{X_1}$. If $X_2\subset X$ is a complemented subspace to $X_1$, then each $B\in GL_0(X)(A)$ has a form
\begin{equation}\label{equ4.1}
B=\left(
\begin{array}{ccc}
I_{X_1}&C\\
0&D
\end{array}
\right),\quad {\rm where}\quad D\in GL(X_2)\ \, {\rm and}\  \, C\in L(X_2,X_1).
\end{equation}
Thus $GL_0(X)(A)$ is homotopy equivalent to the subgroup of $GL(X_2)\, (\cong GL(X/X_1))$ consisting of all operators $D$ such that ${\rm diag}(I_{X_1}, D)\in GL_0(X)$. In particular, this subgroup coincides with $GL(X_2)$ if the latter is connected.

Now, Theorem \ref{teo4.5} leads to the following statement:\smallskip

\noindent $(1)$ {\em Suppose $A\in L(X)_l^{-1}$ is such that group $GL(X/X_1)$ is connected. Then for each $F\in \mathcal O(M(\mathscr A_Z), L(X)_l^{-1})$ whose image belongs to the connected component containing $A$ there is $G\in \mathcal O(M(\mathscr A_Z), GL_0(X))$ such that $F=GA$.}\smallskip

Let us identify $X$ with $X_1$  by means of $A$ and regard $F(x)$, $x\in M(\mathscr A_Z)$,  and $A$ as operators in $L(X_1,X_1\oplus X_2)$.  Then we obtain
\[
F=
\left(
\begin{array}{ccc}
F_1\\
F_2
\end{array}
\right),\qquad
G=
\left(
\begin{array}{ccc}
G_{11}&G_{12}\\
G_{21}&G_{22}
\end{array}
\right)\quad {\rm and}\quad
A=
\left(
\begin{array}{cc}
I_{X_{1}}\\
0
\end{array}
\right),
\]
where $F_i\in \mathcal O(M(\mathscr A_Z), L(X_1,X_i))$, $i=1,2$, and  $G_{ii}\in \mathcal O(M(\mathscr A_Z), L(X_i))$, $i=1,2$,
$G_{ij}\in \mathcal O(M(\mathscr A_Z), L(X_j,X_i))$, $i,j\in\{1,2\}$, $i\ne j$.
Now identity $F=GA$ implies that
\[
F_1=G_{11},\quad F_2=G_{21},
\]
that is, $G$ extends $F$ to an invertible holomorphic operator-valued function.

Thus  from (1) we obtain the following generalization of \cite[Th.\,1.4]{Br2}.\smallskip

\noindent (1$'$)  Suppose $F\in \mathcal O(M(\mathscr A_Z),L(Y_1,Y_2))$, where $Y_i$, $i=1,2$, are complex Banach spaces is such that for each $z\in \Di\setminus Z$ there exists a left inverse $G_z$ of $F(z)$ satisfying 
\[
\sup_{z\in\Di\setminus Z}\|G_z\|<\infty.
\]
Let $Y:={\rm ker}\, G_0$.
 Assume that $GL(Y)$ is connected. Then {\em
there exist maps $H\in \mathcal O(M(\mathscr A_Z),L(Y_1\oplus Y, Y_2))$, $G\in \mathcal O(M(\mathscr A_Z), L(Y_2, Y_1\oplus Y))$ such that for all $x\in M(\mathscr A_Z)$,}
\[
H(x)G(x)=I_{Y_2},\qquad G(x)H(x)=I_{Y_1\oplus Y}\qquad {\rm and}\qquad H(x)|_{Y_1}=F(x).
\]

Note that group $GL(Y)$ is connected in the following cases (see, e.g.,  \cite[Cor.\,]{Br3} for the references): (1) ${\rm dim}_{\Co}Y<\infty$; (2) $Y_2$ is isomorphic to a Hilbert space or $c_0$ or one of the spaces $\ell^p$, $1\le p\le\infty$; (3) $Y_2$ is isomorphic to one of the spaces $L^p[0,1]$, $1< p<\infty$, or $C[0,1]$ and $Y_1$ is not isomorphic to $Y_2$.

%In particular, the result is valid for spaces $Y$ with contractible group $GL(Y)$.The class of such spaces include infinite-dimensional Hilbert spaces, spaces $\ell^p$ and $L^p[0,1]$, $1\le p\le \infty$, $c_0$ and $C[0,1]$, spaces $L_p(\Omega,\mu)$, $1<p<\infty$, of $p$-integrable measurable functions on an arbitrary
%measure space $\Omega$, and some classes of reflexive symmetric function spaces; the class of spaces $X$ with connected but not simply connected group $GL(X)$ include finite dimensional Banach spaces, finite direct products of James spaces etc., see, e.g., \cite{M} and references therein. There are also Banach spaces $X$ whose linear groups $GL(X)$ are not connected. E.g., the groups of connected components of spaces $\ell^{p}\times\ell^q$, $1\le p<q<\infty$, are isomorphic to $\Z$, see  \cite{D}.

\smallskip

Statement  (1$'$) is related to the   following
\begin{P2}
Let $F\in H^\infty(L(H_1,H_2))$ satisfy the hypotheses of the Sz.-Nagy problem and $H_3:=\bigl(F(0)(H_1)\bigr)^{\perp}$.  Do there exist functions $D\in H^\infty(L(H_1\oplus H_3, H_2))$ and $E\in H^\infty(L(H_2, H_1\oplus H_3))$ such that 
\[
D(z)E(z)=I_{H_2},\quad E(z)D(z)=I_{H_1\oplus H_3}\quad\text{ and}\quad D(z)|_{H_1}=F(z)\quad\text{ for all}\quad z\in\Di.
\]
\end{P2}

Seemingly much stronger, this problem is equivalent to the Sz.-Nagy problem. This result, known as the Tolokonnikov lemma, is proved in full generality in \cite{T5}. \smallskip

Finally, part (b) of Theorem \ref{teo4.5} implies in our case:}\smallskip

\noindent $(2)$ Let $A\in\mathscr A_Z$. For each $F\in \mathcal O(M(\mathscr A_Z), L(X)_l^{-1})$ there exist a finitely generated unital subalgebra  $B\subset A$ and maps $H\in \widehat B\otimes L(X)$,   $G\in\mathcal O(M(\mathscr A_Z),GL_0(X))$ such that $GF=H$.
\end{E}

\sect{Proofs of Propositions \ref{prop1.1} and \ref{prop1.4}}
\begin{proof}[Proof of Proposition \ref{prop1.1}]
(a) By $Q_Z: M(H^\infty)\rightarrow M(H_I^\infty)$ we denote the continuous map transposed to the embedding homomorphism $H_I^\infty\hookrightarrow H^\infty$. Each $f\in H_I^\infty$ has a form $f=c+h$ for some $c\in\Co$ and $h\in I$ so that $\hat f(\xi):=\xi(f)=c$ for all $\xi\in  Z:={\rm hull}(I)\subset M(H^\infty)$. (Here and below $\hat{\,}$ is the Gelfand transform for $H^\infty$.) Hence, $Q_Z$ maps $Z$ to a point. Further, if $\xi_1,\xi_2\in M(H^\infty)\setminus Z$ are distinct, there exists a function $f\in H^\infty$ such that $\hat f(\xi_1)=1$ and $\hat f(\xi_2)=0$,  Also, by the definition of a hull, there exists $g\in I$ such that $\hat g(\xi_1)=1$. Function $fg\in H_I^\infty$ and $Q_Z(\xi_i)(fg):=\widehat{fg}(\xi_i)=\hat f(\xi_i)\hat g(\xi_i)=2-i$. Thus functions in $H_I^\infty$ separate points of $M(H^\infty)\setminus Z$, i.e., $Q_Z$ is one-to-one outside $Z$.
Finally, the corona theorem for $H_I^\infty$ implies that the image of $\Di$ under $Q_Z$ is dense in $M(H_I^\infty)$ (cf. \eqref{e1.1}, \eqref{bezout}). Hence, $Q_Z$ is a surjection.\smallskip

(b) Due to Suarez's theorem \cite[Th.\,4.5]{S1}, ${\rm dim}\, M(H^\infty)=2$; hence, ${\rm dim}\, K\le 2$ for all compact $K\subset M(H^\infty)\setminus Z$. Since $Q_Z|_K$ is a homeomorphism, ${\rm dim}\, Q_Z(K)={\rm dim}\, K\le 2$ as well.
Note that each compact subset of $M(H_I^\infty)\setminus \{z\}$, $z:=Q_Z(Z)$, has a form $Q_Z(K)$ for some $K$ as above. Thus ${\rm dim}\, K'\le 2$ for all compact $K'\subset M(H_I^\infty)\setminus \{z\}$. Moreover, ${\rm dim}\, \{z\}=0$. These two facts and the standard result of the dimension theory, see, e.g., 
\cite[Ch.\,2, Th.\,9--11]{N}, imply that ${\rm dim}\, M(H_I^\infty)=2$.\smallskip

(c) According to the cohomology interpretation due to Su\'{a}rez \cite[Th.\,1.3]{S1} of 
Treil's theorem \cite{T6} stating that the Bass stable rank of algebra $H^\infty$ is one, the (\v{C}ech) cohomology exact sequence of the pair $(M(H^\infty),Z)$ acquires the form
\[
\!\rightarrow H^1(M(H^\infty),\mathbb Z)\rightarrow H^1(Z,\mathbb Z)\stackrel{0\ }{\rightarrow}H^2(M(H^\infty),Z,\,\mathbb Z)
\rightarrow H^2(M(H^\infty),\mathbb Z)\rightarrow H^2(Z,\mathbb Z)\rightarrow 0.
\]
From here by another Su\'{a}rez's result \cite[Cor.\,3.9]{S1} asserting that $H^2(M(H^\infty),\mathbb Z)$=0 we obtain that
$H^2\bigl(M(H^\infty),Z,\,\mathbb Z\bigr)=0$. Since $Q_Z: M(H^\infty)\rightarrow M(H_I^\infty)$ maps $Z$ to a point and is one-to-one outside of $Z$,
due to the {\em strong excision property} for cohomology, see, e.g., \cite[Ch.\,6, Th.\,5]{Sp}, the pullback map $Q_Z^*$ induces isomorphism of the \v{C}ech cohomology groups $H^2(M(H_I^\infty),\Z)\cong H^2(M(H^\infty),Z,\,\Z)$.  In particular, $H^2(M(H_I^\infty),\mathbb Z)=0$.
\end{proof}
%====
\begin{proof}[Proof of Proposition \ref{prop1.4}]
Suppose  for $f\in\mathcal O(U\cap (\Di\setminus Z),X)$ there exist open covers $(U_\alpha)_{\alpha\in A}$ of $U\subset M(\mathscr A_Z)$ and $(V_\beta)_{\beta\in B}$ of $X$ and a map $\tau: A\rightarrow B$ such that 
$\mathcal O(V_\beta)$ (algebras of holomorphic functions on $V_\beta$, $\beta\in B$) separate points,   $f(U_\alpha\cap (\Di\setminus Z))\Subset V_{\tau(\alpha)}$, $\alpha\in A$, and
\begin{equation}\label{equ5.0}
\bigcap_{W\subset Q_Z^{-1}(U)\,:\, \mathring{W}\cap Z\ne\emptyset}\overline{f\bigl(W\cap (\Di\setminus Z))}\ne\emptyset.
\end{equation}
We require to show that $f$ extends to a map in $\mathcal O(U,X)$.
\begin{Lm}\label{lem5.1}
For each $\alpha\in A$ map $Q_Z^*f: Q_Z^{-1}(U_\alpha)\cap( \Di\setminus Z)\rightarrow V_{\tau(\alpha)}$ extends to  a map in  $C(Q_Z^{-1}(U_\alpha), V_{\tau(\alpha)})$.
\end{Lm}
Recall that  $Q_Z|_{M(H^\infty)\setminus Z}={\rm id}$, see Convention in Section~2.1.
\begin{proof}
According to the hypotheses  for each $g\in \mathcal O(V_{\tau(\alpha)})$  function $g\circ f\in H^\infty(U_\alpha\cap ( \Di\setminus Z))$. Hence, due to \cite[Th.\,3.2]{S1},\smallskip
 
\noindent ($\circ$)\quad
{\em function $g\circ Q_Z^*f$ extends continuously to  $Q_Z^{-1}(U_\alpha)$.}\smallskip 

By $C_g\subset \mathbb C$ we denote the range of $g\in \mathcal O(V_{\tau(\alpha)})$. Let us consider space
\[
T_\alpha:=\prod_{g\in\mathcal O(V_{\tau(\alpha)})}C_g
\] 
equipped with the product topology and
continuous map $\Psi_\alpha: V_{\tau(\alpha)}\rightarrow T_\alpha$, 
\[
\Psi_\alpha(z)(g):=g(z), \quad z\in V_{\tau(\alpha)},\ \, g\in \mathcal O(V_{\tau(\alpha)}).
\]
Since $\mathcal O(V_{\tau(\alpha)})$ separates points,
$\Psi_\alpha$ is one-to-one. Therefore it maps compact set $K:=\overline{f(U_\alpha\cap (\Di\setminus Z))}$ homeomorphically onto its image  $K_\alpha:=\Psi_\alpha(K)$. Let $\psi_\alpha^{-1}: K_\alpha\rightarrow K$ be the continuous inverse of $\Psi_\alpha|_K$.
Due to ($\circ$) map $\Psi_\alpha\circ f\circ Q_Z|_{Q_Z^{-1}(U_\alpha)\cap( \Di\setminus Z)}$ extends continuously to $Q_Z^{-1}(U_\alpha)$ and so has range in $K_\alpha$. Thus
$Q_Z^*f=\psi_\alpha^{-1}\bigl(\Psi_\alpha\circ f\circ Q_Z|_{Q_Z^{-1}(U_\alpha)\cap( \Di\setminus Z)}\bigr)$ extends continuously to $Q_Z^{-1}(U_\alpha)$ and takes its values in $K\subset V_{\tau(\alpha)}$, as required.
\end{proof}

Since $(U_\alpha)_{\alpha\in A}$ is an open cover of $U$, Lemma \ref{lem5.1} implies that $Q_Z^*f$ extends continuously to a map $\tilde F\in\mathcal O(Q_Z^{-1}(U),X)$.

Next, suppose $Z\subset Q_Z^{-1}(U)$. For points $x_1,x_2\in  Z$, let $\mathcal B(x_i)$ be local bases of relatively compact open neighbourhoods of $x_i$ in $Q_Z^{-1}(U)$, $i=1,2$. Then condition \eqref{equ5.0} and continuity of $\tilde F$ imply that
\[
\{\tilde F(x_1)\}\cap\{\tilde F(x_2)\}=\bigcap_{i=1,2}\,\bigcap_{O_i\in\mathcal B(x_i)} \tilde F(\bar O_i)\ne\emptyset .
\]
Hence, $\tilde F(x_1)=\tilde F(x_2)$ for all distinct $x_1,x_2\in Z$ and so $\tilde F|_Z$ is constant. In turn, there is a map $F\in \mathcal O(U,X)$ such that $\tilde F=Q_Z^*F$. By the definition of $Q_Z$, $F|_{U\cap (\Di\setminus Z)}=f$ as required.\smallskip

The converse statement saying that $f=F|_{U\cap (\Di\setminus Z)}$ with $F\in \mathcal O(U,X)$ satisfies   conditions (a) and (b) of the proposition follows easily by continuity of $F$.
\end{proof}

\sect{Auxiliary Results}
In this part we present some results used in the proofs of theorems of Section~2.2. 

\subsection{Maximal Ideal Space of ${\mathbf H^\infty}$}
Recall that the {\em pseudohyperbolic metric} on $\Di$ is defined by
\[
\rho(z,w):=\left|\frac{z-w}{1-\bar w z}\right|,\qquad z,w\in\Di.
\]
For $x,y\in  M(H^\infty)$ the formula
\[
\rho(x,y):=\sup\{|\hat f(y)|\, :\, f\in H^\infty,\, \hat f(x)=0,\, \|f\|_{H^\infty}\le 1\}
\]
gives an extension of $\rho$ to $M(H^\infty)$.
The {\em Gleason part} of $x\in M(H^\infty)$ is then defined by $\pi(x):=\{y\in M(H^\infty)\, :\, \rho(x,y)<1\}$. For $x,y\in M(H^\infty)$ we have
$\pi(x)=\pi(y)$ or $\pi(x)\cap\pi(y)=\emptyset$. Hoffman's classification of Gleason parts \cite{H} shows that there are only two cases: either $\pi(x)=\{x\}$ or $\pi(x)$ is an analytic disk. The former case means that there is a continuous one-to-one and onto map $L_x:\Di\rightarrow\pi(x)$ such that
$\hat f\circ L_x\in H^\infty$ for every $f\in H^\infty$. Moreover, any analytic disk is contained in a Gleason part and any maximal (i.e., not contained in any other) analytic disk is a Gleason part. By $M_a$ and $M_s$ we denote sets of all nontrivial (analytic disks) and trivial (one-pointed) Gleason parts. It is known that $M_a\subset M(H^\infty)$ is open. Hoffman proved that $\pi(x)\subset M_a$ if and only if $x$ belongs to the closure of an interpolating sequence in $\Di$. The base of topology on $M_a$ consists of sets of the form $\{x\in M_a\, :\, |\hat b(x)|<\varepsilon\}$, where $b$ is an interpolating Blaschke product. This is because for a sufficiently small  $\varepsilon$ set $b^{-1}(\Di_\varepsilon)\subset\Di$, $\Di_\varepsilon:=\{z\in\Co\, :\, |z|<\varepsilon\}$, is biholomorphic to $\Di_\varepsilon\times b^{-1}(0)$, see \cite[Ch.~X, Lm.~1.4]{Ga}. Hence,
$\{x\in M_a\, :\, |\hat b(x)|<\varepsilon\}$ is biholomorphic to $\Di_\varepsilon\times \hat b^{-1}(0)$.

Su\'{a}rez  \cite{S2} proved that the set of trivial Gleason parts is totally disconnected, i.e., ${\rm dim}\, M_s=0$ (as  $M_s$ is compact). 
%=============

\subsection{Holomorphic Banach vector bundles on $M(\mathscr A_Z)$} Let $E$ be a Banach vector bundle on $M(\mathscr A_Z)$ with fibre a complex Banach space $X$. We say that $E$ is {\em holomorphic} if it is defined on an open cover $(U_i)_{i\in I}$ of $M(\mathscr A_Z)$ by a holomorphic cocycle $g=\{g_{ij}\in \mathcal O(U_i\cap U_j, GL(X))\}_{i,j\in I}$.  In this case, $E|_{\Di\setminus Z}$ is a holomorphic Banach vector bundle on $\Di\setminus Z$ in the usual sense.
Recall that $E$ is defined as the quotient space of disjoint union $\sqcup_{i\in I}\,U_i\times X$ by the equivalence relation:
\begin{equation}\label{eq5.1}
U_j\times X\ni u\times x\sim u\times g_{ij}(u)x\in U_i\times X.
\end{equation}
The projection $p:E\rightarrow M(\mathscr A_Z)$ is induced by natural projections $U_i\times X\rightarrow U_i$, $i\in I$.

A morphism $\varphi : (E_1, X_1, p_1)\rightarrow (E_2, X_2, p_2)$ of holomorphic Banach vector bundles on $M(\mathscr A_Z)$ is a continuous map which sends each vector space $p_1^{-1}(w)\cong X_1$ linearly to vector space $p_2^{-1}(w)\cong X_2$, $w\in M(\mathscr A_Z)$, and such that $\varphi|_{\Di\setminus Z}: E_1|_{\Di\setminus Z}\rightarrow E_2|_{\Di\setminus Z}$ is a holomorphic map of complex Banach manifolds. If, in addition, $\varphi$ is bijective, it is called an  isomorphism. 

We say that a holomorphic Banach vector bundle on $M(\mathscr A_Z)$ is {\em trivial} if it is isomorphic to the  bundle
$E_X:=M(\mathscr A_Z)\times X$. 

The following result was established in \cite{Br3} by  techniques developed in \cite{Br1}, \cite{Br2}.
\begin{Th}[\cite{Br3}, Theorem 3.2]\label{teo5.1}
For a complex Banach space $X$ there exists a complex Banach space $Y$ such that for each
holomorphic Banach vector bundle $E$ on $M(H^\infty)$ with fibre $X$ the Whitney sum $E\oplus E_Y$ is trivial.  
\end{Th}
Let $E$ be a holomorphic Banach vector bundle on $M(\mathscr A_Z)$.
For an open subset $U\subset M(\mathscr A_Z)$ by $\Gamma_{\mathcal O}(U,E)$  we denote the complex vector space of holomorphic sections $s$ of $E$ over $U$ (i.e., such that  $s_{U\cap (\Di\setminus Z)}$ are holomorphic sections of the bundle $E|_{U\cap (\Di\setminus Z)}$ in the usual sense).  We equip $\Gamma_\mathcal O(U,E)$ with the topology $\tau_c$ of uniform convergence on compact subsets of $U$. It is generated by all subsets of the form $O_{K,V}:=\{f\in \Gamma_\mathcal O(U,E)\, :\, f(K)\subset V\}$, where $K\subset U$ is compact and $V\subset E$ is open subsets.

Using Theorem \ref{teo5.1} we prove
\begin{Th}\label{teo5.2}
Let $E$ be a holomorphic Banach vector bundle on $M(H^\infty)$ and $U$ an open neighbourhood of $Z$. There is a continuous linear operator $L: \Gamma_\mathcal O(U,E)\rightarrow \Gamma_\mathcal O(M(\mathscr A_Z),E)$ such that 
\[
Lf|_{Z}=f|_Z\quad {\rm for\ all}\quad f\in \Gamma_\mathcal O(U,E).
\]
\end{Th}
\begin{proof}
For $E$ a trivial bundle the result is established in \cite[Th.\,1.9]{Br1}. In this case, we may assume without loss of generality that $E=E_X$. Then $(\Gamma_\mathcal O(U,E),\tau_c)$ is isomorphic to the topological vector space $\mathcal O(U,X)$  
with the topology of uniform convergence on compact subsets of $U$. The corresponding continuous linear operator $L$ is constructed as in the proof of Theorem 1.7 in \cite[pp.\,213--214]{Br1}. The general case is reduced to that one by means of Theorem \ref{teo5.1}. 

Indeed, in the notation of the theorem, let $i: E\rightarrow E\oplus E_Y$  and $q: E\oplus E_Y\rightarrow E$,  $q\circ i={\rm id}$, be continuous embedding and projection. They induce continuous linear maps $i_*: \Gamma_\mathcal O(U,E)\rightarrow \Gamma_\mathcal O(U,E\oplus E_Y)$, $i_*f:= i\circ f$, and $q_*: \Gamma_\mathcal O(M(\mathscr A_Z),E\oplus E_Y)\rightarrow \Gamma_\mathcal O(M(\mathscr A_Z),E)$, $q_*g:=q\circ g$. Since $E\oplus E_Y$ is trivial, by the above result of \cite{Br1} there is a continuous linear operator $L': \Gamma_\mathcal O(U,E\oplus E_Y)\rightarrow  \Gamma_\mathcal O(M(\mathscr A_Z),E\oplus E_Y)$  such that $L'f|_Z=f|_Z$ for all $f\in \Gamma_\mathcal O(U,E\oplus E_Y)$. We set
\[
L:=q_*L' i_*: \Gamma_\mathcal O(U,E)\rightarrow \Gamma_\mathcal O(M(\mathscr A_Z),E).
\]
Then $L$ is a continuous linear operator and for all $x\in Z$
\[
(Lf)(x)=q((L'(i_*f))(x))=q((i_*f)(x))=(q\circ i)(f(x))=f(x)
\]
as required.
\end{proof}

\subsection{Runge-type approximation theorem}
Recall that a compact subset $K\subset M(\mathscr A_Z)$ is called {\em holomorphically convex} if for any $x\notin K$ there is $f\in \mathcal O(M(\mathscr A_Z))$ such that
\[
\max_{K}|f|<|f(x)|.
\]
\begin{Th}\label{runge}
Each holomorphic section
of a holomorphic Banach vector bundle $E$ on $M(\mathscr A_Z)$ defined on a neighbourhood of a holomorphically convex set $K\subset M(\mathscr A_Z)$ can be uniformly approximated on $K$ by sections from $\Gamma_\mathcal O(M(\mathscr A_Z),E)$. 
\end{Th}
For $E$ a trivial bundle on $M(H^\infty)$ the result was proved in \cite[Th.\,1.7]{Br1}. 
\begin{proof}
Let $Q_Z: M(H^\infty)\rightarrow M(\mathscr A_Z)$ be the continuous surjection transposed to the embedding $H_{I(Z)}^\infty\hookrightarrow H^\infty$ (cf. Proposition \ref{prop1.1}). We set $z:=Q_Z(Z)$.
\begin{Lm}\label{lem5.4}
If $K\subset M(\mathscr A_Z)$ is  holomorphically convex, $K\cup\{z\}$ is holomorphically convex. 
\end{Lm}
\begin{proof}
Without loss of generality we may assume that $z\not\in K$. Let $x\not\in K\cup\{z\}$. Then there exists a function $f\in\mathcal O(M(\mathscr A_Z))$ such that
\[
m_f:=\max_K |f|<|f(x)|=1.
\]
Let $g\in\mathcal O(M(\mathscr A_Z))$ be such that $|g(x)|=1$ and $g(z)=0$. We set
\[
m_g:=\max_K |g|.
\]
Replacing $f$ by $f^n$ with a sufficiently large $n\in\N$, if necessary, without loss of generality we may assume that
\[
m_f\, m_g<1.
\]
Then for the function $fg\in\mathcal O(M(\mathscr A_Z))$ we get
\[
\max_{K\cup\{z\}}|fg|\le m_f\, m_g<|f(x)g(x)|=1.
\]
This shows that set $K\cup\{z\}$ is holomorphically convex.
\end{proof}
Suppose that $s\in \Gamma_\mathcal O(U,E)$, where $U$ is an open neighbourhood of a holomorphically convex set $K\subset M(\mathscr A_Z)$. Due to Lemma \ref{lem5.4} without loss of generality we may assume that $z\in K$. (Indeed, for otherwise, there are open disjoint sets $U_1\subset U$ and $U_2\ni z$. Then we will prove Theorem \ref{runge} for $K$ replaced by $K\cup\{z\}$, $U$ by $U_1\cup U_2$ and $s$ by the section in $\Gamma_\mathcal O(U_1\cup U_2,E)$ equals $s$ on $U_1$ and $0$ on $U_2$.) Thus
$\widetilde K:=Q_Z^{-1}(K)$ is a holomorphically convex subset of $M(H^\infty)$ containing $Z$.
In turn, $\widetilde E:=Q_Z^*E$ is a holomorphic Banach vector bundle on $M(H^\infty)$ and $\tilde s:=Q_Z^*s\in \Gamma_\mathcal O(\widetilde U,E)$ with $\widetilde U:=Q_Z^{-1}(U)$ an open neighbourhood of $\widetilde K$. Note that $\tilde s$ is constant on $Z$.

Let $E_Y:=M(H^\infty)\times Y$ be such that bundle $\widetilde E\oplus E_Y$
is trivial (see Theorem \ref{teo5.1}). By $i: \widetilde E\rightarrow \widetilde E\oplus E_Y$  and $q: \widetilde E\oplus E_Y\rightarrow \widetilde E$, $q\circ i={\rm id}$, we denote  continuous embedding and projection. Then $i\circ \tilde s\in \Gamma_\mathcal O(\widetilde U, \widetilde E\oplus E_Y)$ and is constant on $Z$.
\begin{Lm}\label{lem5.5}
Section $i\circ \tilde s$ can be uniformly approximated on $\widetilde K$ by holomorphic sections from $\Gamma_\mathcal O(M(H^\infty),\widetilde E\oplus E_Y)$ constant on $Z$.
\end{Lm}
\begin{proof}
By \cite[Lm.\,5.1]{Br1} there exists a holomorphically convex set $V\subset\widetilde U$ whose interior
$\mathring{V}$
contains $\widetilde K$. Since $\widetilde E\oplus E_Y$ is  trivial on $M(H^\infty)$, \cite[Th.\,1.7]{Br1} implies that there exists a sequence $\{s_j\}_{j\in\N}\subset \Gamma_\mathcal O(M(H^\infty),\widetilde E\oplus E_Y)$ converging to $i\circ \tilde s$ uniformly on $V$. Hence, the sequence $t_j:=(i\circ\tilde s-s_j)|_{\mathring{V}}$, $j\in\N$, of sections in $\Gamma_\mathcal O(\mathring{V},\widetilde E\oplus E_Y)$ converges to the zero section uniformly on $\mathring{V}$. Let $L: \Gamma_\mathcal O(\mathring{V},\widetilde E\oplus E_Y)\rightarrow \Gamma_\mathcal O(M(H^\infty),\widetilde E\oplus E_Y)$ be a continuous linear operator such that $Lf|_Z=f|_Z$ (see Theorem \ref{teo5.2}). We set
$\tilde t_j:=Lt_i\in \Gamma_\mathcal O(M(H^\infty),\widetilde E\oplus E_Y)$, $i\in\N$. Then $\tilde t_j|_Z=t_j|_Z$, $j\in\N$, and $\{\tilde t_j\}_{j\in\N}$ converges to the zero section uniformly on $M(H^\infty)$. Finally, defining
$u_j:=s_j+\tilde t_j$, $j\in\N$, we obtain that $\{u_j\}_{j\in\N}\subset \Gamma_\mathcal O(M(H^\infty),\widetilde E\oplus E_Y)$ converges to $i\circ\tilde s$ uniformly on $\widetilde K$ and $u_j|_Z=(i\circ \tilde s)|_Z$ (constant) for all $j$. 

The proof of the lemma is complete.
\end{proof}
In the notation of the lemma  sequence $\{q\circ u_j\}_{j\in\N}\subset \Gamma_\mathcal O(M(H^\infty),\widetilde E)$ converges to $\tilde s$ uniformly on $\widetilde K$. Since $(q\circ u_j)|_Z=\tilde s|_Z$ for all $j$, there is a sequence $\{v_j\}_{j\in\N}\subset S_\mathcal O(M(\mathscr A_Z),E)$ such that $v_j\circ Q_Z=q\circ u_j$, $j\in\N$, converging to $s$ uniformly on $K$.

This concludes the proof of the theorem.
\end{proof}
\subsection{Cousin-type lemma}
Let $E$ be a holomorphic Banach vector bundle with fibre $X$ on $M(\mathscr A_Z)$ defined on a finite open cover $\mathfrak U=(U_i)_{i\in I}$ of $M(\mathscr A_Z)$ by a holomorphic cocycle $g=\{g_{ij}\in\mathcal O(U_i\cap U_j,GL(X))\}_{i,j\in I}$, see \eqref{eq5.1}.
A continuous section of $E$ defined on a compact subset $K\subset M(\mathscr A_Z)$ is called {\em holomorphic} if it is the restriction of a holomorphic section of $E$ defined in an open neighbourhood of $K$. The space of such sections is denoted by $\Gamma_\mathcal O(K,E)$.   
Let $\Gamma_C(K,E)$ be the topological vector space of continuous sections of $E$ on $K$ equipped with the topology of uniform convergence.  This space is normable,  a norm on $\Gamma_C(K,E)$ compatible with topology can be defined as follows. 

Let us fix a finite refinement $\mathfrak V=(V_l)_{l\in L}$ of $\mathfrak U$ consisting of compact subsets. Let $\tau: L\rightarrow I$ be the  refinement map, i.e., $V_l\Subset U_{\tau(l)}$ for all $l\in L$. Each $s\in \Gamma_C(K,E)$ in the local coordinates on $V_l$, $l\in L_K:=\{m\in L\, :\, V_m\cap K\ne\emptyset\}$, determined by \eqref{eq5.1} is represented by a pair $(v,s_l(v))$, $v\in K\cap V_l$, $s_l\in C(K\cap V_l,X)$, such that $s_l(v)=g_{\tau(l)\tau(m)}(v)s_m(v)$ for all $v\in K\cap V_l\cap V_m\ne\emptyset$, $l, m\in L_K$. We set
\begin{equation}\label{eq5.2}
\|s\|_{K;E}^{\mathfrak V}:=\max_{l\in L_K}\max_{v\in K\cap V_l}\|s_l(v)\|_X .
\end{equation}
It is readily seen that $(\Gamma_C(K,E),\|\cdot\|_{K;E}^{\mathfrak V})$ is a complex Banach space and the norm topology on $\Gamma_C(K,E)$ coincides with the topology of uniform convergence.\smallskip

By $\Gamma_\mathcal A(K,E)$ we denote the closure of $\Gamma_\mathcal O(K,E)$ in the Banach space $\Gamma_C(K,E)$. \smallskip

Suppose open $U_1, U_2\Subset M(\mathscr A_Z)$ are such that (a) $z\not\in \bar U_i\setminus U_i$, $i=1,2$; (b)  $Q_Z^{-1}(\bar U_1\cap \bar U_2)\subset M_a\setminus Z$. Let $W_i\Subset U_i$, $i=1,2$, be compact subsets  and $W:=W_1\cup W_2$.

\begin{Th}[Cousin-type Lemma]\label{cousin}
The bounded linear operator of complex Banach spaces $A:  \Gamma_\mathcal A(\bar U_1\cap W, E)\oplus \Gamma_\mathcal A(\bar U_2\cap W, E)\rightarrow  \Gamma_\mathcal A(\bar U_1\cap \bar U_2\cap W, E)$,
\[
A(f_1,f_2):=f_1|_{\bar U_1\cap \bar U_2\cap W}+f_2|_{\bar U_1\cap \bar U_2\cap W},\quad f_i\in \Gamma_\mathcal A(\bar U_i\cap W, E),\ i=1,2,
\]
is surjective.
\end{Th}

\begin{proof}
For $E$ a trivial bundle on $M(H^\infty)$ the result is established in \cite[Th.\,3.1]{Br2}. The general case is reduced to that one. \smallskip

We must show that \smallskip

\noindent ($\circ$) For each $f\in \Gamma_\mathcal A(\bar U_1\cap \bar U_2\cap W, E)$ there exist
$f_i\in \Gamma_\mathcal A(\bar U_i\cap W, E),\ i=1,2$, such that $f_1+f_2=f$ on $\bar U_1\cap \bar U_2\cap W$. \smallskip

Let $\widetilde E:=Q_Z^*E$ be the pullback of $E$ to $M(H^\infty)$ and $E_Y:=M(H^\infty)\times Y$ be such that $\widetilde E\oplus E_Y$ is trivial. As before, by $i: \widetilde E\rightarrow \widetilde E\oplus E_Y$  and $q: \widetilde E\oplus E_Y\rightarrow \widetilde E$, $q\circ i={\rm id}$, we denote continuous embedding and projection.  We set 
\[
\begin{array}{l}
\displaystyle
\tilde f:=f\circ Q_Z\in \Gamma_\mathcal A(Q_Z^{-1}(\bar U_1\cap \bar U_2\cap W), \widetilde E),\quad V_i:=Q_Z^{-1}(U_i),\quad \widetilde W_i:=Q_Z^{-1}(W_i),\quad i=1,2,\medskip\\
\displaystyle
\widetilde W:=Q_Z^{-1}(W)\, (=\widetilde W_1\cup \widetilde W_2).
\end{array}
\] 
Since $Q_Z$ maps $Z$ to $z\in M(\mathscr A_Z)$ and is one-to-one on $M(H^\infty)\setminus S$,  assumptions of the theorem imply that  $\bar V_i= Q_Z^{-1}(\bar U_i)$, $i=1,2$. Thus, $i\circ \tilde f\in \Gamma_\mathcal A(\bar V_1\cap\bar V_2\cap\widetilde W,\widetilde E\oplus E_Y)$. 
Since $\widetilde E\oplus E_Y$ is trivial, \cite[Th.\,3.1]{Br2} yields  sections
$f_i'\in \Gamma_\mathcal A(\bar V_i\cap \widetilde W, \widetilde E\oplus E_Y)$  such that
\begin{equation}\label{eq5.4}
\begin{array}{c}
f_1'+f_2'=i\circ \tilde f\quad\text{on}\quad \bar V_1\cap \bar V_2\cap \widetilde W. 
\end{array}
\end{equation}
Hence, for $f_i'':=q\circ f_i'\in \Gamma_\mathcal A(\bar V_i\cap \widetilde W, \widetilde E)$, $i=1,2$,  
\begin{equation}\label{eq5.5}
\begin{array}{c}
f_1''+ f_2''=\tilde f\quad\text{on}\quad \bar V_1\cap \bar V_2\cap \widetilde W.
\end{array}
\end{equation}
If $Z\cap  ((\bar V_1\cup \bar V_2)\cap\widetilde W)=\emptyset$, then there exist some $f_i\in \Gamma_\mathcal A(\bar U_i\cap W,E)$ such that $f_i\circ Q_Z=f_i''$, $i=1,2$. Due to \eqref{eq5.5} these sections satisfy ($\circ$). For otherwise, by the hypothesis of the theorem, $Z\subset V_{i_0}\cap \widetilde W$ for some $i_0\in\{1,2\}$. Without loss of generality we may assume that $i_0=1$. Then according to Theorem \ref{teo5.2} there is $g\in \Gamma_\mathcal O(M(H^\infty),\widetilde E)$ such that $g|_Z=f_1''|_Z$. We set 
\[
\tilde f_1:=f_1''-g,\qquad \tilde f_2:=f_2''+g.
\]
By the definition and due to \eqref{eq5.5},
\[
\tilde f_1|_Z=0\quad {\rm and}\quad \tilde f_2|_Z=\tilde f|_Z=const.
\]
Hence, there are $f_i\in \Gamma_\mathcal A(\bar U_i\cap W,E)$ such that $f_i\circ Q_Z=\tilde f_i$, $i=1,2$. In view of \eqref{eq5.5}, these sections satisfy ($\circ$).
 \end{proof}
 \subsection{Cartan-type Lemma}
 For basic facts of the Banach Lie group theory, see, e.g., \cite{M}.
 
Let $G$ be a complex Banach Lie group with Lie algebra $\mathfrak g$. The latter has the structure of a complex Banach space  naturally identified with the tangent space of $G$ at unit $1_G$. 
 By $\exp_G: \mathfrak g\rightarrow G$ we denote the corresponding holomorphic exponential map which maps a neighbourhood of $0\in\mathfrak g$ biholomorphically onto a neighbourhood of $1_G\in G$.
 
Let $E$ be a topological bundle on $M(\mathscr A_Z)$ with fibre $G$ defined on a finite open cover $\mathfrak U=(U_i)_{i\in I}$ of $M(\mathscr A_Z)$ by holomorphic transition functions
$f_{ij}\in\mathcal O(U_i\cap U_j\times G,G)$, $i,j\in I$ (i.e., restrictions of $f_{ij}$ to $\bigl((U_i\cap U_j)\cap (\Di\setminus Z)\bigr)\times G$
are holomorphic in the usual sense), such that\smallskip
\begin{itemize}
\item[(1)]
$f_{ij}(x,f_{jk}(x,g))=f_{ik}(x,g),\quad x\in U_i\cap U_j\cap U_k,\ g\in G;$\medskip
\item[(2)] $f_{ij}(x,\cdot)$ is an automorphism of $G$ for each $x\in U_i\cap U_j$, $i,j\in I$.
\end{itemize}
Thus, $E$ is the quotient of disjoint union $\sqcup_{i\in I} U_i\times G$ by the equivalence relation:
\[
U_j\times G\ni (x,g)\sim (x,f_{ij}(x,g))\in U_i\times G.
\]
By $\Gamma_C(Y,E)$ we denote the set of continuous sections of $E$ over $Y\subset M(\mathscr A_Z)$. It has the natural group structure induced by the product on $G$ with the unit the restriction to $K$ of the section $1_{E}$ assigning to each $x\in M(\mathscr A_Z)$ the unit $1_G$ of the fibre of $E$ over $x$.

A section in $\Gamma_C(W,E)$,  $W\subset M(\mathscr A_Z)$ is open, is called holomorphic if its restriction to $W\cap (\Di\setminus Z)$ is the holomorphic section of the holomorphic fibre bundle $E|_{\Di\setminus Z}$. The set of such sections is denoted by $\Gamma_\mathcal O(W,E)$. For a compact set $K\subset M(\mathscr A_Z)$ a section of $E$ over $K$ is called holomorphic if it is the restriction of a holomorphic section of $E$ defined on an open neighbourhood of $K$. This set of sections is denoted by $\Gamma_\mathcal O(K,E)$. \smallskip

Associated with $E$ is the holomorphic Banach vector bundle $T(E)$  with fibre $\mathfrak g$ of complex tangent spaces to fibres of $E$ at $1_G$. It is defined on the cover $\mathfrak U$ by cocycle
$\bigl\{({\rm d}_2f_{ij})_{1_G}\in\mathcal O(U_i\cap U_j,{\rm Aut}(\mathfrak g))\bigr\}_{i,j\in I}$, where ${\rm d}_2$ is the differential with respect to the variable $g\in G$.
Map $\exp_G$ induces a holomorphic map $\exp_E:T(E)\rightarrow E$ of holomorphic fibre bundles sending an open neighbourhood $U$ of the image of the zero section $0_{T(E)}$ of $T(E)$ biholomorphically onto a  neighbourhood $V$ of the image of section $1_E$ in $E$.  The complex vector space $\Gamma_C(K,T(E))$ of continuous sections of $T(E)$ on a compact set $K\subset M(H_{I(Z)}^\infty)$ equipped with norm \eqref{eq5.2} defined with respect to a fixed refinement $\mathfrak V$ of the cover $\mathfrak U$ is a complex Banach space. It has the structure of a complex (Banach) Lie algebra with Lie product induced by the Lie product on $\mathfrak g$.
In turn, map $\exp_E$ induces a map $(\exp_E)_*:\Gamma_C(K,T(E))\rightarrow \Gamma_{C}(K,E)$, $s\mapsto\exp_E\circ s$, sending 
$\Gamma_\mathcal O(K,T(E))$ to $\Gamma_{\mathcal O}(K,E)$. Since $1_E|_K\in \Gamma_\mathcal O(K,E)$, the latter is a subgroup of group $\Gamma_C(K,E)$. We equip $\Gamma_C(K,E)$ with the topology $\tau_u$ of uniform convergence.
Then $(\Gamma_C(K,E),\tau_u)$ has the structure of a complex Banach Lie group with  Lie algebra $\Gamma_C(K,T(E))$ and with  exponential map $(\exp_E)_*$.

Finally, by $\Gamma_\mathcal A(K,E)$ we denote the closure of $\Gamma_\mathcal O(K,E)$ in $\Gamma_C(K,E)$. It consists of continuous sections of $E$ on $K$ that can be uniformly approximated by sections of $\Gamma_\mathcal O(K,E)$. Space 
 $(\Gamma_\mathcal A(K,E),\tau_u)$ is a complex Banach Lie subgroup of $\Gamma_C(K,E)$  with Lie algebra $\Gamma_\mathcal A(K,T(E))$.\smallskip
 
By $\widetilde U_K\subset \Gamma_\mathcal A(K,T(E))$ and $\widetilde V_K\subset \Gamma_\mathcal A(K,E)$ we denote open neighbourhoods of $0_{T(E)}|_K$ and $1_E|_K$ consisting of sections $s$ such that $s(K)\subset U$ and $s(K)\subset V$, respectively. Thus, $(\exp_E)_*:\widetilde U_K\rightarrow \widetilde V_K$ is a biholomorphic map. 

It is readily seen (as $M(\mathscr A_Z)$  is compact and $\exp_E$ is continuous) that there is some
$r\in (0,\infty)$ such that 
for each compact $K\subset M(\mathscr A_Z)$  the open ball $B_K(r)\subset \Gamma_\mathcal A(K,T(E))$ of 
radius $r$ centred at $0_{T(E)}|_K$ is a subset of $\widetilde U_K$ and   
\begin{equation}\label{equ5.7}
(\exp_E)_*(s_1)\cdot(\exp_E)_*(s_2)\in \widetilde V_K\quad {\rm for\ all}\quad s_i\in  B_K(r),\ i=1,2.
\end{equation}

In the next two results we retain notation of Theorem \ref{cousin}, that is,\smallskip

\noindent ($\circ$) $U_1, U_2\Subset M(\mathscr A_Z)$ are open such that (a) $z\not\in \bar U_i\setminus U_i$, $i=1,2$; (b)   $Q_Z^{-1}(\bar U_1\cap \bar U_2)\subset M_a\setminus Z$, and $W_i\Subset U_i$, $i=1,2$, are compact, $W:=W_1\cup W_2$. 

Our next result solves the analog of the 
``probl\`{e}me fondamental'' of Cartan \cite{Ca}.

\begin{Proposition}[Cartan-type Lemma]\label{small}
There is $0<r_1\le r$ such that for every $F\in \Gamma_\mathcal A(\bar U_{1}\cap\bar U_{2}\cap W, E)$ with $(\exp_E)_*^{-1}(F)\in
B_{\bar U_{1}\cap \bar U_{2}\cap  W}(r_1)$ there exist $F_j\in \Gamma_\mathcal A(\bar U_{j}\cap W, E)$, $j=1,2$, such that 
\[
F_{1}\, F_{2}=F\quad {\rm on}\quad \bar U_{1}\cap\bar U_{2}\cap W.
\]
\end{Proposition}
\begin{proof}
We set for brevity
\[
Y_0:=\bar U_{1}\cap \bar U_{2}\cap  W\quad {\rm and} \quad Y_j:=\bar U_{j}\cap W,\quad j=1,2.
\]
By \eqref{equ5.7} the holomorphic map
$H :B_{Y_1}(r)\times B_{Y_2}(r)\rightarrow \Gamma_\mathcal A(Y_0,T(E))$,
\[
H(s_1,s_2):=(\exp_E)_*^{-1}\bigl((\exp_E)_*(s_{1}|_{Y_0})\cdot (\exp_E)_*(s_{2}|_{Y_0})\bigr),
\]
is well-defined. Its differential at $0_{T(E)}|_{Y_1}\times 0_{T(E)}|_{Y_2}$ is the bounded linear operator
$A: \Gamma_\mathcal A(Y_1,T(E))\times \Gamma_\mathcal A(Y_2,T(E))\rightarrow \Gamma_\mathcal A(Y_0,T(E))$,
\[
A(s_1,s_2):=s_{1}|_{Y_0}+s_{2}|_{Y_0}.
\] 
Due to Theorem \ref{cousin}, $A$ is surjective; hence by the implicit function theorem (see, e.g., \cite{L}) there exists a continuous right inverse of $H$ defined on a ball $B_{Y_0}(r_1)$, $r_1\in (0,r)$.
\end{proof}

Recall that a path in the complex Banach Lie group $\Gamma_\mathcal A(K, E)$ is a continuous map $[0,1]\rightarrow \Gamma_\mathcal A(K, E)$. The set of all sections in $\Gamma_\mathcal A(K, E)$ that can be joined by paths in $\Gamma_\mathcal A(K, E)$ with  section $1_E|_{K}$ forms the {\em connected component} of the unit $1_E|_K$.

Our next result generalizes \cite[Th.4.1]{Br2}.

\begin{Th}\label{cartan}
Assume that $W_1$ is holomorphically convex and $W_1\cap W_2\ne\emptyset$.
Let $F\in \Gamma_\mathcal A(\bar U_1\cap \bar U_2,E)$ belong to the connected component of $1_E|_{\bar U_1\cap \bar U_2}$. Then there exist 
$F_i\in \Gamma_\mathcal A(W_i, E)$, $i=1,2$, such that $F_1\, F_2=F$ on $W_1\cap W_2$.
\end{Th}
\begin{proof}
An open {\em polyhedron} $\Pi\subset M(\mathscr A_Z)$ is the set of the form
\[
\Pi:=\{x\in M(\mathscr A_Z)\, :\, \max_{1\le j\le l}|f_j(x)|<1,\ f_j\in\mathcal O(M(\mathscr A_Z)),\ 1\le j\le l\}.
\]
(Here $l$ can be any natural number.)

In turn, the closed polyhedron in $M(\mathscr A_Z)$  corresponding to $\Pi$ is defined as
\[
\Pi^c:=\{x\in M(\mathscr A_Z)\, :\, \max_{1\le j\le l}|f_j(x)|\le 1,\ f_j\in\mathcal O(M(\mathscr A_Z)),\ 1\le j\le l\}.
\]
(Note that the closure of $\Pi$ is not necessarily $\Pi^c$.)

The proof of the following result repeats literally the proof of Lemma 5.1 in \cite{Br1}.
\begin{Lm}\label{lem5.9}
Let $N\subset M(\mathscr A_Z)$ be a neighbourhood of a holomorphically convex set $K$. Then there exists an open polyhedron $\Pi$ such that   $K\subset\Pi\subset N$.
\end{Lm}

Now, since $W_1$ is holomorphically convex, according to Lemma \ref{lem5.9}, there exists a sequence of open polyhedrons $\{U_{1,i}\}_{i\in\N}$ containing $W_1$ such that $U_{1,i}^c\Subset U_1$ and $U_{1,i+1}^c\Subset U_{1,i}$, $i\in\N$. (Observe that by definition all $U_{1, i}^c$ are holomorphically convex.) Let us choose a sequence of open 
sets $U_{2,i}\Subset U_2$ containing $W_2$ such that $U_{2,i+1}\Subset U_{2,i}$ for all $i\in\N$.
Then Proposition \ref{small} is also valid with $U_j$ replaced by $U_{j,i}$, $j=1,2$,  $W:=W_1\cup W_2$  by $W^i:=U_{1,i+1}^c\cup\bar U_{2,i+1}$ and $r_1$  by some $r_{1,i}\in (0,r)$.

Further, since $F$ in the statement of the theorem belongs to the connected component $\mathcal C_0$  of the group $\Gamma_\mathcal A(\bar U_1\cap \bar U_2, E)$ containing $1_E|_{\bar U_1\cap \bar U_2}$, there exists a path $\gamma : [0,1]\rightarrow \mathcal C_0$ such that $\gamma(0)=1_E|_{\bar U_1\cap \bar U_2}$ and $\gamma(1)=F$. Continuity of $\gamma$ implies that there is a partition $0=t_0<t_1<\cdots <t_k=1$ of $[0,1]$ such that 
\begin{equation}\label{eq4.1}
(\exp_E)_*^{-1}\bigl(\gamma(t_i)^{-1}\,\gamma(t_{i+1})\bigr)|_{\bar U_{1,1}\cap \bar U_{2,1}\cap W^1}\in B_{\bar U_{1,1}\cap \bar U_{2,1}\cap W^1}(r_{1,1})\ \text{ for all}\ \,i.
\end{equation}
Applying Proposition \ref{small} to each $\gamma(t_i)^{-1}\,\gamma(t_{i+1})$ we obtain
\begin{equation}\label{eq4.2}
\gamma(t_i)^{-1}\,\gamma(t_{i+1})=F_{1}^i\, F_{2}^i\quad\text{on}\quad \bar U_{1,1}\cap \bar U_{2,1}\cap W^1
\end{equation}
for some $F_{l}^i\in \Gamma_\mathcal A(\bar U_{l,1}\cap W^1, E)$ such that $(\exp_E)_*^{-1}(F_l^i)\in B_{\bar U_{l,1}\cap W^1}(r)$, $l=1,2$, cf. \eqref{equ5.7}.

Next, we have
\[
F=\prod_{i=0}^{k-1}\gamma(t_i)^{-1}\,\gamma(t_{i+1}).
\]

To prove the theorem we use induction on $j$ for $1\le j\le k$. The induction hypothesis is {\em If
\[
F^j:=\prod_{i=0}^{j-1}\gamma(t_i)^{-1}\,\gamma(t_{i+1}),
\]
then for some $F_{l,j}\in \Gamma_\mathcal A(\bar U_{l,j}\cap W^j, E)$, $l=1,2$,}
\[
F^j=F_{1,j}\, F_{2,j}\quad\text{on}\quad \bar U_{1,j}\cap \bar U_{2,j}\cap W^j.
\]

By \eqref{eq4.2} the statement is true for $j=1$. Assuming that it is true for $j-1$ let us prove it for $j$.

By the induction hypothesis, 
\[
F^j=F^{j-1}\, \gamma(t_j)^{-1}\,\gamma(t_{j+1})= F_{1,j-1}\, F_{2,j-1}\, F_{1}^{j-1}\, F_{2}^{j-1}\quad\text{on}\quad
\bar U_{1,j-1}\cap\bar U_{2,j-1}\cap W^{j-1}.
\]
Since $U_{1,j}^c\subset  U_{1,j-1}\cap W^{j-1}$ is holomorphically convex, by Theorem \ref{runge}, we can approximate
$(\exp_{E})_*^{-1}( F_{1}^{j-1})$ uniformly on $U_{1,j}^c$ by sections in $\Gamma_\mathcal O(M(\mathscr A_Z),T(E))$. Thus for every $\varepsilon>0$ there is
$F_{\varepsilon}\in \Gamma_\mathcal O(M(\mathscr A_Z),E)$ such that
\begin{equation}\label{eq4.3}
(\exp_E)_*^{-1}(F_{1}^{j-1}|_{U_{1,j}^c}\cdot (F_{\varepsilon}|_{U_{1,j}^c})^{-1})\in B_{U_{1,j}^c}(\varepsilon) .
\end{equation}
We write
\[
F_{2,j-1}\, F_{1}^{j-1}\, F_{2}^{j-1}=[F_{2,j-1},F_{1}^{j-1}\,F_{\varepsilon}^{-1}]\, (F_{1}^{j-1}\, F_{\varepsilon}^{-1})\, F_{2,j-1}\, F_{\varepsilon}\,  F_{2}^{j-1}
\]
on $\bar U_{1,j-1}\cap\bar U_{2,j-1}\cap W^{j-1}$. (Here $[A_1,A_2]:=A_1 A_2 A_1^{-1} A_2^{-1}$.)

According to \eqref{eq4.3} for a sufficiently small $\varepsilon$, 
\[
(\exp_E)_*^{-1}\bigl(([F_{2,j-1},F_{1}^{j-1}\, F_{\varepsilon}^{-1}]\, (F_{1}^{j-1}\, F_{\varepsilon}^{-1}))|_{\bar U_{1,j}\cap \bar U_{2,j}\cap W^j}\bigr)\in B_{\bar U_{1,j}\cap \bar U_{2,j}\cap W^j}(r_{1,j}) .
\]
Hence, by Proposition \ref{small}, there exist $H_{l}\in \Gamma_\mathcal A(\bar U_{l,j}\cap W^j, E)$, $l=1,2$, such that 
\[
[F_{2,j-1},F_{1}^{j-1}\, F_{\varepsilon}^{-1}]\, (F_{1}^{j-1}\, F_{\varepsilon}^{-1})=H_1\, H_2\quad\text{on}\quad \bar U_{1,j}\cap \bar U_{2,j}\cap W^j.
\]
In particular, we obtain
\[
F^j=F_{1,j-1}\, F_{2,j-1}\, F_{1}^{j-1}\, F_{2}^{j-1}=F_{1,j-1}\, H_1\, H_2\, F_{2,j-1}\, F_{\varepsilon}\,  F_{2}^{j-1}\quad\text{on}\quad \bar U_{1,j}\cap \bar U_{2,j}\cap W^j.
\]
We set
\[
F_{1,j}:=F_{1,j-1}\,H_1\quad\text{on}\quad \bar U_{1,j}\cap W^j,\qquad F_{2,j}:=H_2\, F_{2,j-1}\, F_{\varepsilon}\,  F_{2}^{j-1}\quad\text{on}\quad \bar U_{2,j}\cap W^j.
\]
Then 
\[
F^j=F_{1,j}\, F_{2,j}\quad\text{on}\quad \bar U_{1,j}\cap \bar U_{2,j}\cap W^j.
\]
This completes the proof of the induction step.

Using this result for $j:=k-1$ together with the fact $W_l\subset \bar U_{l,k+1}$, $l=1,2$, we obtain that there exist 
$F_i\in \Gamma_\mathcal A(W_i, E)$ such that $F_1\, F_2=F$ on $W_1\cap W_2$.
\end{proof}
\subsection{Blaschke Sets}

Any open set of the form
$$
O_{b,\varepsilon}:=\{x\in M(H^\infty)\, :\, |\hat b(x)|<\varepsilon\},
$$
where $b$ is an interpolating Blaschke product and
$\epsilon$ is so small that $O_{b,\varepsilon}\cap\Di$ is biholomorphic to $\Di_\varepsilon\times  b^{-1}(0)$, cf. Section~6.1, will be called an open {\em Blaschke set}. For such $\varepsilon$ the closure $\bar O_{b,\varepsilon}$ is given by $\{x\in M(H^\infty)\, :\, |\hat b(x)|\le\varepsilon\}$ and so it is holomorphically convex.

Note that by the normality argument the biholomorphism in the definition of  $O_{b,\varepsilon}$ extends (by means of the Gelfand transform) to a homeomorphism between $O_{b,\varepsilon}$ and $\Di_\varepsilon\times \hat b^{-1}(0)$.

Let $G$ be a complex Banach Lie group with Lie algebra $\frak g$ and exponential map $\exp_G$. By $\mathcal O(K,G)$, $K\subset M(\mathscr A_Z)$ compact, we denote the group with respect to the pointwise multiplication of restrictions to $K$ of holomorphic maps into $G$ defined on open neighbourhoods of $K$ equipped the topology of uniform  convergence. By $\mathcal A(K,G)$ we denote the closure of $\mathcal O(K,G)$ in the group of continuous maps $K\rightarrow G$ equipped with  the topology of uniform convergence. Then $\mathcal A(K,G)$ is a complex Banach Lie group with Lie algebra $\mathcal A(K,\frak g)$ and exponential map $(\exp_{G})_*(f):=\exp_G\circ f$, $f\in \mathcal A(K,\frak g)$.  
\begin{Th}\label{compactif}
Suppose $G$ is simply connected and $N$ is a  compact subset of the Blaschke set $O_{b,\varepsilon}$. Then the complex Banach Lie group $\mathcal A(K,G)$ is connected.
\end{Th}
\begin{proof}
Let ${\mathbf 1}_G$  be the constant map $M(H^\infty)\rightarrow G$ whose value is the unit $1_G$ of $G$. Then ${\mathbf 1}_G|_K$ is the unit of 
$\mathcal A(K,G)$. We must show that each $F\in \mathcal A(K,G)$ can be joined by a path with ${\mathbf 1}_G|_K$. Since $\mathcal O(K,G)$ is dense in the complex Banach manifold $\mathcal A(K,G)$, it suffices to prove the result for $F\in \mathcal O(K,G)$. So assume that $F$ is holomorphic in an open neighbourhood $U\Subset O_{b,\varepsilon}$ of $N$. Without loss of generality we may assume that $b$ is not finite and identify $O_{b,\varepsilon}$ with $\Di_\varepsilon\times \beta\N$. (Here $\beta\N$ is the Stone-\v{C}ech compactification of $\N$.)
\begin{Lm}\label{lem5.11}
 $N$ can be covered by open sets $S_j\times Y_j\Subset U$, $1\le j\le k$, such that $Y_j\subset \beta\N$ are clopen and pairwise disjoint, and each $S_j\Subset\Di_\varepsilon$  has finitely many connected components that are finitely connected domains.
\end{Lm} 
\begin{proof}
 Since $\beta\N$ is totally disconnected, it is homeomorphic to the inverse limit of a family of finite sets, see, e.g., \cite{N}. Hence, $O_{b,\varepsilon}$ can be obtained as the inverse limit of a family of sets of the form $\Di_{\varepsilon}\times {\mathcal F}$, where ${\mathcal F}$ is a finite set. In turn, $N$ can be obtained as the inverse limit of the family of compact subsets $K\times {\mathcal F}\subset \Di_{\varepsilon}\times {\mathcal F}$. By the definition of the inverse limit topology, there exists an open neighbourhood $V\Subset \Di_\varepsilon\times\mathcal F$ of one of such $K\times {\mathcal F}$ whose preimage under the continuous limit projection $\pi: \Di_\varepsilon\times\beta N\rightarrow \Di_\varepsilon\times\mathcal F$ (identical in the first coordinate) of the inverse limit construction is contained in $U$.
 Note that $V=\cup_{s\in\mathcal F}(V_s\times \{s\})$, where $V_s\Subset \Di_\varepsilon$ are open. Replacing each $V_s$ by its relatively compact subset $V_s'$ with analytic boundary so that $V':=\cup_{s\in\mathcal F}(V_s'\times \{s\})$ still contains $K$, we get an open neighbourhood $\pi^{-1}(V')\Subset U$ of $N$ of the required form.
 \end{proof}
 
Each set of the form $\cup_{ j=1}^k \,S_j\times Y_j$ with $S_j$ and $Y_j$ as in Lemma \ref{lem5.11} will be called {\em elementary}. For an elementary  set $X=\cup_{j=1}^k\, S_j\times Y_j$ its elementary subset of the form $\cup_{j=1}^k\!\cup_{l=1}^m S_j\times Y_{jl}$
will be called a {\em refinement} of $X$. By $\mathcal O_f(X,G)\subset \mathcal O(X,G)$ we denote the set of holomorphic maps $f$ such that each $f|_{S_j\times Y_j}$ is independent of the second coordinate, $1\le j\le k$.
\begin{Lm}\label{lem5.12}
Let $X=\cup_{j=1}^k\, S_j\times Y_j$ with $S_j\times Y_j$ as in Lemma \ref{lem5.11}. There exist refinements $X_n$ of $X$ and  holomorphic maps $F_n\in\mathcal O_f(X_n,G)$, $n\in\N$, such that sequence $\{F_n\}_{n\in\N}$ converges uniformly to $F|_X$.
\end{Lm}
\begin{proof}
Let us consider maps $F^j:=F|_{\bar S_j\times Y_j}$, $1\le j\le k$.
Each $F^j$ can be regarded as a continuous map of the clopen set $Y_j\subset\beta\N$ to the complex Banach Lie group $\mathcal A(\bar S_j,G)\subset C(\bar S_j, G)$ of continuous $G$-valued maps holomorphic on $S_j$ equipped with the topology of uniform convergence.  We endow $\mathcal A(\bar S_j,G)$ with a metric $d$ compatible with topology (existing by the Birkhoff-Kakutani theorem). Since $F^j$ is a continuous map of compact set $Y_j$ to metric space $(\mathcal A(\bar S_j,G),d)$, it is uniformly continuous. Hence for each $n\in\N$, $1\le j\le k$, there exists a partition $Y_j=\sqcup_{l=1}^{m_n} Y_{nl}^j$ into clopen subsets
such that
\begin{equation}\label{eq5.9}
d(F^j(x),F^j(y))\le\frac{1}{n}\quad {\rm for\ all}\quad x, y\in Y_{nl}^j,\ 1\le l\le m_n.
\end{equation}
Let us fix points $x_{nl}^j\in Y_{nl}^j$, $n\in\N$, $1\le j\le k$, $1\le l\le m_n$, and define 
\[
X_n:=\cup_{j=1}^k\!\cup_{l=1}^{m_n} S_j\times Y_{nl}^j
\]
and
\[
F_n(z,x):=F(z,x_{nl}^j),\qquad (z,x)\in S_j\times Y_{nl}^j,\quad 1\le j\le k,\quad 1\le l\le m_n.
\]
Due to \eqref{eq5.9}, sequences $\{X_n\}_{n\in\N}$ and $\{F_n\}_{n\in\N}$ satisfy the requirements of the lemma. 
\end{proof}
Lemma \ref{lem5.12} shows that in order to prove Theorem \ref{compactif} it suffices to assume that $F\in\mathcal O_f(X,G)$, where $X=\cup_{j=1}^k\, S_j\times Y_j$ is an elementary subset of $O_{b,\varepsilon}$.
In this case, $F|_{S_j\times Y_j}(s,y)=f_j(s)$, $(s,y)\in S_j\times Y_j$, for a holomorphic map $f_j:S_j\rightarrow G$, $1\le j\le k$.
Since each connected component of $S_j$ is homotopic to a finite one-dimensional $CW$-complex (cf. Lemma \ref{lem5.11}) and $G$ is simply connected, each $f_j$ is homotopic to the constant map $S_j\ni s\mapsto 1_G\in G$. Then by \cite[Th.\,2.1(b)]{R} (the Banach-valued version of the
Ramspott theorem) each $f_j$ 
is homotopic to this map inside  $\mathcal O(S_j,G)$. This implies that each $F|_{S_j\times Y_j}$ is homotopic inside $\mathcal O_f(S_j\times Y_j,G)$ to ${\mathbf 1}_G|_{S_j\times Y_j}$. Denote by $H_j: [0,1]\rightarrow  \mathcal O_f(S_j\times Y_j,G)$, $1\le j\le k$, the corresponding homotopies.
Then $H:[0,1]\rightarrow\mathcal A(K,G)$,
\[
H(t)(x):=H_j(t)(x)\quad {\rm if}\quad  x\in K\cap (S_j\times Y_j),\quad 1\le j\le k,
\]
is a path joining $F$ and $1_G|_K$. 

The proof of the theorem is complete.
\end{proof}

Let $\pi :P\rightarrow M(\mathscr A_Z)$ be a holomorphic principal bundle with fibre a simply connected complex Banach Lie group $G$. For a compact set $K\subset M(\mathscr A_Z))$ we say that $P|_K$ is trivial if it is trivial in an open neighbourhood of $K$ (cf. Section~2.2).
\begin{C}\label{cor5.13}
Suppose $V_i\Subset U_i\subset M(\mathscr A_Z)$, $i=1,2$, are open such that (a) $z\not\in \bar U_i\setminus U_i$, $i=1,2$; (b) $Q_Z^{-1}(\bar U_1\cap \bar U_2)\subset M_a\setminus Z$  is a subset of an open Blaschke set; (c) $\bar V_1$ is holomorphically convex. We set $V:=V_1\cup V_2$. If bundles $P|_{\bar U_i}$, $i=1,2$, are trivial, then $P|_V$ is trivial.
\end{C}
\begin{proof}
Without loss of generality we may assume that $\bar V_1\cap\bar V_2\ne\emptyset$ and $\bar U_1\cap\bar U_2$ is the proper subset of each $\bar U_i$, $i=1,2$ (for otherwise the statement is trivial). Then, by the hypothesis, $P|_{\bar U_1\cup\bar U_2}$ is determined by a cocycle $c\in \mathcal O(\bar U_1\cap\bar U_2, G)$.  Since $Q_Z^{-1}(\bar U_1\cap \bar U_2)\subset M_a\setminus Z$, pullback  $Q_Z^*$ determines an isomorphism of complex Banach Lie groups $\mathcal A(\bar U_1\cap\bar U_2, G)$ and $\mathcal A(Q_Z^{-1}(\bar U_1\cap \bar U_2),G)$. As $\bar U_1\cap \bar U_2$ is a compact subset of an open Blaschke set, Theorem \ref{compactif} implies that complex Banach Lie group $\mathcal A(\bar U_1\cap\bar U_2, G)$ is connected.
Hence we can apply Theorem \ref{cartan} to $c$ (with  $E=M(\mathscr A_Z)\times G$ and $W_i:=\bar V_i$, $i=1,2$) to find some $c_i\in \mathcal A(\bar V_i, G)$, $i=1,2$, such that $c_1^{-1}\, c_2= c$ on $\bar V_1\cap \bar V_2$. This shows that $P|_{V}$ is trivial.
\end{proof}
In the sequel, we also require the following result.
\begin{Proposition}\label{prop5.14}
Suppose $K\subset M(\mathscr A_Z)\setminus\{z\}$ is compact such that $Q_Z^{-1}(K)$ is the closure of an open Blaschke product $O_{b,\varepsilon}$. Then $K$ is holomorphically convex. 
\end{Proposition}
\begin{proof}
Assume that $x'\in M(H^\infty)$ is such that $x:=Q_Z(x')\not\in K\cup\{z\}$, $\{z\}:=Q_Z(Z)$. Since $Q_Z^{-1}(K)=\bar O_{b,\varepsilon}$ is holomorphically convex, there is a function $f\in \mathcal O(M(H^\infty))$ such that $f(x')=1$ and $\max_{Q_Z^{-1}(K)}|f|\le \frac12$. By definition $Z$ is the hull of ideal $I(Z)\ne (0)$. Hence, there exists a function $g\in\mathcal O(M(H^\infty))$ such that $g(x')=1$ and $g|_Z=0$. For a sufficiently large $n\in\N$ we obtain
\[
(f^n\, g)(x')=1\quad {\rm and}\quad \max_{Q_Z^{-1}(K)\cup Z}|f^n\, g|<1.
\]
Since $(f^n\, g)|_Z=0$, there exists a function $h\in \mathcal O(M(\mathscr A_Z))$ such that $Q_Z^*h=f^n\, g$. In particular,
\[
|h(x)|=1\quad {\rm and}\quad \max_{K\cup\{z\}} |h|<1.
\] 
The latter condition shows that $K\cup \{z\}\subset M(\mathscr A_Z)$ is a holomorphically convex set. Let $U_1$ and $U_2$ be open neighbourhoods of $K$ and $z$ such that $\bar U_1\cap\bar U_2=\emptyset$. Consider the holomorphic function on $U_1\sqcup U_2$ equals $0$ on $U_1$ and $1$ on $U_2$. By Theorem \ref{runge} (the Runge-type approximation theorem) it can be uniformly approximated on $K\cup \{z\}$ by holomorphic functions of $\mathcal O(M(\mathscr A_Z))$. In particular, there is a function $t\in\mathcal O(M(\mathscr A_Z))$ such that
$|t(z)|>\frac 12$ and $\max_K |t|<\frac 12$. This shows that $K$ is holomorphically convex.
\end{proof}

\sect{Proofs of Theorem \ref{te1.4}, Corollary \ref{cor1.5} and Proposition \ref{prop1.6}}
\subsection{Principal Bundles with Connected Fibres}
First, we prove a particular case of Theorem \ref{te1.4} for bundles with connected fibres. 

Let $V\Subset U\subset M(\mathscr A_Z)$ be open sets.
\begin{Th}\label{teo6.1}
Suppose $P$ is a  holomorphic principal bundle on $U$ with fibre a connected complex Banach Lie group $G$.
Then $P|_{V}$ is trivial.
\end{Th}
\begin{proof}
By $\psi: G_u\rightarrow G$ we denote the universal covering of $G$. The complex Banach Lie group $G_u$ is simply connected and the kernel of $\psi$ is a discrete central subgroup $Z_G\subset G_u$. 

Let $W$ be a relatively compact open subset of $U$ such that $V\Subset W$.
\begin{Proposition}\label{prop6.2}
There exist a principal $G_u$-bundle $P_u$ on $\bar W$ and  a morphism of  principal bundles $\Psi: P_u\rightarrow P|_{\bar W}$.
\end{Proposition} 
\begin{proof}
Since $\bar W$ is compact, $P|_{\bar W}$ is defined
on a {\em finite} open cover $\mathfrak W=(W_i)_{i\in I}$ of $\bar W$ by a cocycle $c=\{c_{ij}\in C (W_i\cap W_j, G)\}_{i,j\in I}$ such that $c_{ij}|_{W\cap W_i\cap W_j}\in \mathcal O(W\cap W_i\cap W_j, G)$ for all $i, j\in I$.
 To prove the result one constructs a  principal $G_u$-bundle $P_u$ on $\bar W$  defined on a finite open cover $\mathfrak V=(V_l)_{l\in L}$ of $\bar W$, a refinement of $\mathfrak W$, by cocycle $\tilde c=\{\tilde c_{lm}\in  C(V_l\cap V_m, G_u)\}_{l,m\in L}$ such that cocycle $\psi(\tilde c)=\{\psi(\tilde c_{lm})\in C(V_l\cap V_m, G)\}_{l,m\in L}$ coincides with the pullback of $c$ to $\mathfrak V$ with respect to the refinement map $\tau: L\rightarrow I$. (Recall that $V_l\subset W_{\tau(l)}$ for all $l\in L$ so that $\psi(\tilde c_{lm})=c_{\tau(l)\tau(m)}|_{V_l\cap V_m}$, $l,m\in L$.)
 The required morphism $\Psi: P_u\rightarrow P|_{\bar W}$ in local coordinates on $\mathfrak V$ is defined as $\Psi(x, g):=(x,\psi(g))$, $(x,g)\in V_l\times G_u$, $l\in L$. Note that since $\psi$ is locally biholomorphic, {\em $P_u$ is holomorphic on $W$}.\smallskip

The construction of such bundle $P_u$  repeats word-for-word the one presented on pages 135--136 of the proof of  \cite[Th.\,6.4]{Br2}, where instead of Lemma~6.5 there one uses 
\begin{Lm}\label{lem6.3}
For each finitely generated subgroup $Z'\subset Z_G$
the \v{C}ech cohomology group $H^2(\bar W,Z')=0$.
\end{Lm}
The proof of this lemma is the same  as that of Lemma~6.5 in \cite{Br2} and relies upon the facts ${\rm dim}\, \bar W=2$ and $H^2(\bar W,\Z)=0$ which follow immediately from similar results for $M(\mathscr A_Z)$ established in Proposition \ref{prop1.1} above. We leave the details to the readers.
\end{proof}

Proposition \ref{prop6.2} implies that in order to prove the theorem it suffices to show that {\em bundle $P_u|_V$ is trivial}. (In this case there exists a holomorphic section $s: V\rightarrow P_u$ and so $\Psi(s)$ is a holomorphic section of $P|_V$, i.e., $P|_V$ is trivial as well.) Let us prove this statement.\smallskip

Since $\bar W$ is compact,  $P_u|_W$ is defined by a holomorphic $G_u$-valued cocycle on a {\em finite} open cover $\mathfrak  V=(V_l)_{l\in L}$ of $W$; in particular, $P_u|_{V_l}=V_l\times G_u$ for all $l\in L$. Since $M_s$ is totally disconnected and $Q_Z$ maps $Z$ to $z\in M(\mathscr A_Z)$ and is one-to-one outside, $Q_Z(M_s)\cup\{z\}$ is totally disconnected as well (see, e.g., \cite[Ch.\,2,\,Th.\,9--11]{N}). In particular, there exists a finite open cover of $\bigl(Q_Z(M_s)\cup\{z\}\bigr)\cap\bar V$ by pairwise disjoint open subsets of $W$ such that $P_u$ is trivial over each of them. Thus $P_u$ is trivial over the union of these sets forming an open neighbourhood of $(Q_Z(M_s)\cup\{z\})\cap\bar V$. Without loss of generality we may assume that this neighbourhood coincides with $V_1$. 

Let us choose some open neighbourhood $V_0\Subset V_1$ of $(Q_Z(M_s)\cup\{z\})\cap\bar V$ and a sequence of open sets $W_1^i\Subset V$ containing $\bar V_0$ such that $W_1^{i+1}\Subset W_1^i$, $i\in\N$. Since $Q_Z^{-1}(\bar V\setminus V_0)$ is a compact subset of $M_a\setminus Z$, it can be covered  by open Blaschke sets $O_{b_1,\varepsilon_1},\dots, O_{b_N,\varepsilon_N}$ so  that  all $O_{b_j,2\varepsilon_j}\Subset (M_a\setminus Z)\cap W$ and each $Q_Z(O_{b_j,2\varepsilon_j})$ is a subset of one of $V_l$, $l\in L$.  For every $i\in\N$ we set
\[
W_{j+1}^i:=Q_Z(O_{b_j,\lambda_i\varepsilon_j}),\quad \lambda_i:=1+\frac{1}{i},\quad 1\le j\le N.
\]
Then $(W_j^i)_{1\le j\le N+1}$, $i\in\N$, are open covers of $\bar V$ such that (a) $P_u|_{W_j^i}$ are trivial for all $i,  j$; (b) all $Q_Z^{-1}(W_{j}^i)$, $2\le j\le N+1$, $i\in\N$, are open Blaschke sets relatively compact in $M_a\setminus Z$; (c) $W_{j}^{i+1}\Subset W_{j}^i$ for all $i,j$;
(d) all $\bar W_j^i$, $2\le j\le N+1$, $i\in\N$, are holomorphically convex (see Proposition \ref{prop5.14}).

Next, for $1\le k\le N+1$ and $i\in\N$ we set
\[
Z_{k}^i:=\bigcup_{j=1}^k W_{j}^i.
\]
Using induction on $k$, $1\le k\le N+1$, we prove that  {\em each bundle $P_u|_{Z_{k}^{k}}$ is trivial}.

Indeed, $P_u|_{Z_{1}^1}$ is trivial by the definition of $V_1$. Assuming that the statement is valid for $k-1$ let us prove it for $k$.

To this end, we apply Corollary \ref{cor5.13} with $U_1$ being an open relatively compact subset of $W_{k-1}^k$ containing $\bar W_k^k$, $U_2$ an open relatively compact subset of $Z_{k-1}^{k-1}$ containing $Z_{k-1}^k$, $V_1:=W_{k}^k$, $V_2:=Z_{k-1}^k$. Since $V_1\cup V_2=Z_k^k$, the required statement follows from the corollary.

For $k=N+1$ we obtain $\bar V \subset Z_{N+1}^{N+1}$. Thus $P_u|_{V}$ is trivial as required.

The proof of the theorem is complete.
\end{proof}

\subsection{Proof of Theorem \ref{te1.4}}
Suppose $\pi: P\rightarrow U$ is a holomorphic principal $G$-bundle trivial on an open neighbourhood $V$ of a compact set $K\subset U$. If  $P|_V$ is defined on an open cover $(V_l)_{l\in L}$ of $V$ by cocycle $g=\{g_{lm}\in\mathcal O(V_l\cap V_m, G)\}_{l,m\in L}$,   the associated principal $C(G)$-bundle $P_{C(G)}|_V$  is defined by cocycle $q(g)=\{q(g_{lm})\in  C(V_l\cap V_m, C(G))\}_{l,m\in L}$. (Recall that $C(G):=G/G_0$ is the discrete group of connected components of $G$ and  $q: G\rightarrow C(G)$ is the corresponding quotient homomorphism.)

Since $P|_V$ is trivial, there exist $g_l\in\mathcal O(V_l,G)$, $l\in L$, such that
\[
g_l^{-1}\, g_m=g_{lm}\quad\text{on}\quad V_l\cap V_m.
\]
This implies that 
\[
q(g_l)^{-1}\, q(g_m)=q(g_{lm})\quad\text{on}\quad V_l\cap V_m,
\]
i.e., bundle $P_{C(G)}|_{V}$ is topologically trivial and so $P_{C(G)}|_{K}$ as well.

Conversely, suppose $P_{C(G)}|_K$ is trivial. Then there is an open neighbourhood $V\Subset U$ of $K$ such that
$P_{C(G)}|_V$ is trivial  as well  (see, e.g., \cite[Lm.\,4]{Li} for the proof of this well-known fact). Thus if $P|_V$ is defined on a finite open cover $\mathfrak V=(V_l)_{l\in L}$ of $V$ by a cocycle $g=\{g_{lm}\in\mathcal O(V_l\cap V_m, G)\}_{l,m\in L}$,  there exist
$h_l\in C(V_l, C(G))$, $l\in L$, such that
\begin{equation}\label{eq6.1}
h_l^{-1}\, h_m=q(g_{lm})\quad\text{on}\quad U_l\cap U_m.
\end{equation}
Replacing $V$ by a relatively compact open subset containing $K$, if necessary, without loss of generality we may assume that each $h_l$ admits a continuous extension to compact set $\bar V_l$, $l\in L$. Since $h_l:\bar V_l\rightarrow C(G)$ is continuous and $C(G)$ is discrete, the image of $h_l$ is finite. In particular, there exists a continuous locally constant function $\tilde h_l :V_l\rightarrow G$ such that $q\circ \tilde h_l=h_l$. By definition each $\tilde h_l\in\mathcal O(V_l,G)$. Let us define cocycle $\tilde g=\{\tilde g_{lm}\in  \mathcal O (V_l\cap V_m, G)\}_{l,m\in L}$ by the formulas
\begin{equation}\label{eq6.2}
\tilde g_{lm}:=\tilde h_l\, g_{lm}\,\tilde h_m^{-1}\quad\text{on}\quad V_l\cap V_m.
\end{equation}
Then $\tilde g$ determines a holomorphic principal $G$-bundle $\tilde P$ on $V$ isomorphic to $P|_V$. Also, \eqref{eq6.1} implies that for all $l,m\in L$
\[
q(\tilde g_{lm})=1_{C(G)}.
\]
Thus each $\tilde g_{lm}$ maps $V_l\cap V_m$ into $G_0$. In particular, $\tilde g$ determines also a subbundle of $\tilde P$ with fibre $G_0$. According to Theorem \ref{teo6.1} this subbundle is trivial over an open neighbourhood $W\Subset V$ of $K$
(because $G_0$ is connected). Thus there exist $\tilde g_l\in\mathcal O(W\cap V_l,G_0)$, $l\in L_W:=\{k\in L\, :\, W\cap V_k\ne\emptyset\}$, such that
\[
\tilde g_l^{-1}\, \tilde g_m=\tilde g_{lm}\quad\text{on}\quad (W\cap V_l)\cap (W\cap V_m) .
\]
From here and \eqref{eq6.2} we obtain that for all $l,m\in L_W$
\[
(\tilde g_l\, h_l)^{-1}\, (\tilde g_m\, h_m)=g_{lm}\quad\text{on}\quad  (W\cap V_l)\cap (W\cap V_m) .
\]
This shows that holomorphic principal $G$-bundle $P|_W$ is trivial.

The proof of the theorem is complete.

\subsection{Proof of Corollary \ref{cor1.5}}
(1) In this case $C(G)$ is trivial so the statement follows directly from Theorem \ref{te1.4}.

(2) By the hypothesis the associated bundle $P_{C(G)}$ on $M(H^\infty)$ has a finite fibre. Then it is trivial by 
\cite[Lm.\,8.1]{Br3} and the result follows from Theorem \ref{te1.4}.

(3) Let $P|_K$ be defined on a finite open cover $(U_i)_{i\in I}$ of $K$ by a cocycle $g=\{g_{ij}\in C (U_i\cap U_j, G\}_{i,j\in I}$. Since $P|_K$ is topologically trivial, there exist $g_i\in C(U_i,G)$, $i\in I$, such that
\[
g_i^{-1}\, g_j=g_{ij}\quad\text{on}\quad U_i\cap U_j .
\]
This implies that 
\[
q(g_i)^{-1}\, q(g_j)=q(g_{ij})\quad\text{on}\quad U_i\cap U_j ,
\]
i.e., $P_{C(G)}|_K$ is topologically trivial. Thus to get the result we apply Theorem \ref{te1.4}.
%==
\subsection{Proof of Proposition \ref{prop1.6}}
We require
\begin{Lm}\label{lem7.4}
The \v{C}ech cohomology group $H^1(M(\mathscr A_Z),\Z)$ is nontrivial.
\end{Lm}
\begin{proof}
Due to the Arens-Royden theorem (see \cite{A2}, \cite{Ro}) it suffices to show that there exists a function in $\mathcal O(M(\mathscr A_Z),\mathbb C^*)$ whose pullback  to $M(H^\infty)$ by $Q_Z$ does not have a bounded logarithm.

To this end let us take a function $f\in I(Z)$ such that $\|f\|_{H^\infty}=1$. Then $\hat f(x)=1$ for some $x\in M(H^\infty)$.
By the definition $K:=\hat f(M(H^\infty))$ is a compact subset of the closed unit disk $\bar{\Di}$ containing $0$ and $1$, and $f(\Di)$ is the open dense subset of $K$. In particular, $z=1$ is the limit point of $f(\Di)$.
Consider holomorphic function
$h(z):=i\cdot{\rm Log}\bigl(\frac{1-z}{z+1}\bigr)$, $z\in\Di$; here ${\rm Log}$ is the principal value of the logarithmic function. It is unbounded on a neighbourhood of $z=1$ in $\Di$ and therefore $h|_{f(\Di)}$ is unbounded as well. Moreover, ${\rm Re}\ h=-{\rm Arg}\bigl(\frac{1-z}{z+1}\bigr)$ is bounded (taking values in $[-\frac{\pi}{2},\frac{\pi}{2}]$). Thus $g:=h\circ f$ is an unbounded holomorphic function on $\Di$ such that $e^{\pm g}\in H^\infty$.  Since $h$ is holomorphic  at $z=0$, function $g$ can be presented as a uniformly convergent series in variable $f$ on each set  $N_f(\varepsilon):=\{z\in\Di\, :\, |f(z)|<\varepsilon\}$, $\epsilon \in (0,1)$. Then due to \cite[Th.\,3.2]{S1}  $g|_{N_f(\varepsilon)}$ admits  a continuous extension $\hat g\in\mathcal O(N_{\hat f}(\varepsilon))$,  $N_{\hat f}(\varepsilon):=
\{z\in\Di\, :\, |\hat f(x)|<\varepsilon\}$, such that $\hat g|_{Z}=h(0)=0$. By uniqueness of the extension $\widehat{e^g}|_{N_{\hat f}(\varepsilon)}=e^{\hat g}$. Thus, $\widehat{e^g}\in \mathcal O(M(H^\infty),\mathbb C^*)$  and equals $1$ on $Z$. In particular, there exists $\tilde g\in\mathcal O(M(\mathscr A_Z),\mathbb C^*)$ such that $Q_Z^*\tilde g=\widehat{e^g}$.  Since by our construction function $\widehat{e^g}$ does not have a bounded logarithm, $\tilde g$ is as required.
\end{proof}

Let $G$ be a complex Banach Lie group such that group $C(G)$ has a nontorsion element. The construction of a nontrivial holomorphic principal $G$-bundle on $M(\mathscr A_Z)$ is similar to that  of \cite[Th.\,3.1]{Br3}.

Consider an integer-valued continuous cocycle $\{c_{ij}\}_{i,j\in I}$ defined on an open cover $(U_i)_{i\in I}$ of $M(\mathscr A_Z)$ representing a nonzero element of $H^1(M(\mathscr A_Z),\mathbb Z)$, see Lemma \ref{lem7.4}. Let $\langle a\rangle:=\{a^n\, :\, n\in\Z\}\subset C(G)$ be the subgroup generated  by a nontorsion element $a$. Then $\langle a\rangle\cong\mathbb Z$. Let $p: B\rightarrow M(\mathscr A_Z)$ be the principal bundle with (discrete) fibre $C(G)$
defined on the cover $(U_i)_{i\in I}$ by $C(G)$-valued cocycle $\{a^{c_{ij}}\}_{i,j\in I}$.
\begin{Lm}\label{lem7.5}
$B$ is a nontrivial  bundle on $M(\mathscr A_Z)$.
\end{Lm}
\begin{proof}
Suppose, on the contrary, that $B$ is trivial. Then there are $g_i\in C(U_i , C(G))$, $i\in I$, such that
\begin{equation}\label{equ7.3}
g_j(x)=g_i(x)\, a^{c_{ij}(x)}\quad\text{for all}\ x\in U_i\cap U_j,\quad i,j\in I.
\end{equation}
Next, the restriction of $\{Q_{Z}^*c_{ij}\}_{i,j\in I}$ to the open cover $(\Di\cap Q_Z^{-1}(U_i))_{i\in I}$ of $\Di$ represents an element of the cohomology group $H^1(\Di,\Z)$. Since $\Di$ is contractible, this group is trivial. In particular, there exist continuous functions $c_i: \Di\cap Q_Z^{-1}(U_i)\rightarrow \Z$ such that
\[
c_{j}(x)=c_i(x)+(Q_Z^*c_{ij})(x)\quad\text{for all}\ x\in (\Di\cap Q_Z^{-1}(U_i))\cap (\Di\cap Q_Z^{-1}(U_j)),\  i,j\in I.
\]
This and \eqref{equ7.3} imply that
\[
g(x):=(Q_Z^*g_i)(x)\,a^{-c_i(x)},\quad x\in \Di\cap Q_Z^{-1}(U_i),\quad i\in I,
\]
is a continuous map on $\Di$ with values in $C(G)$.  Since $C(G)$ is a discrete space and $\Di$ is connected, $g(x)=g(0)$ for all $x\in\Di$. Hence each map $g^{-1}(0)\, Q_Z^*g_i:\Di\cap Q_{Z}^{-1}(U_i)\rightarrow C(G)$ has range in $\langle a\rangle$. 
Further, since $\Di$ is an open dense subset of $M(H^\infty)$ and $Q_Z^{-1}(U_i)\subset M(H^\infty)$ is open, $\overline{\Di\cap Q_Z^{-1}(U_i)}$  contains $Q_Z^{-1}(U_i)$. Since $Q_Z^*g_i$ is continuous on $Q_Z^{-1}(U_i)$ and $\langle a\rangle$ is a closed subset of $C(G)$, extending  $(g^{-1}(0)\, Q_Z^*g_i)|_{\Di\cap Q_Z^{-1}(U_i)}$ by continuity we get that
each $g^{-1}(0)\, Q_Z^*g_i: Q_Z^{-1}(U_i)\rightarrow C(G)$ has range in $\langle a\rangle$. Therefore each $g^{-1}(0)\, g_i: U_i\rightarrow C(G)$ has range in $\langle a\rangle$ as well.

Then equations (cf. \eqref{equ7.3})
\[
g^{-1}(0)\, g_j(x)=g^{-1}(0)\, g_i(x)\, a^{c_{ij}(x)}\quad\text{for all}\ x\in U_i\cap U_j,\quad i,j\in I,
\] 
show that cocycle $\{a^{c_{ij}}\}_{i,j\in I}$ determines the trivial bundle in the category of principal bundles on $M(\mathscr A_Z)$ with fibre $\langle a\rangle\cong\Z$. Equivalently, cocycle
$\{c_{ij}\}_{i,j\in I}$ represents $0$ in $H^1(M(\mathscr A_Z),\Z)$, a contradiction. 
\end{proof}

Now, consider an element $b\in G$ such that $q(b)=a$. We set $g_{ij}:=b^{c_{ij}}$. Then $\{g_{ij}\}_{i,j\in I}$ is a holomorphic $G$-valued 1-cocycle on cover $(U_i)_{i\in I}$ of $M(\mathscr A_Z)$. It determines a holomorphic principal bundle $P$ on $M(\mathscr A_Z)$ with fibre $G$. By the definition the associated bundle $P_{C(G)}$ coincides with $B$ and therefore by Lemma \ref{lem7.5} it is nontrivial. Hence, $P$ is a nontrivial holomorphic $G$-bundle by Theorem \ref{te1.4}.

The proof of Proposition \ref{prop1.6} is complete.
%===
\sect{Proof of Theorem \ref{teo1.7}}
We will establish a more general result required in the proofs of theorems of Section~3.2.\smallskip

Let $P_1$ be a holomorphic principal bundle on $M(\mathscr A_Z)$ with fibre a complex Banach Lie group $G$ and $P_2$ be a holomorphic principal $G$-bundle on an open set $U\subset M(\mathscr A_Z)$. Let $V\Subset U$ be an open subset.
\begin{Th}\label{teo8.1}
If $P_1|_U$ and $P_2$ are topologically isomorphic,  then $P_1|_V$ and $P_2|_V$ are holomorphically isomorphic.
\end{Th}
\begin{proof}

Passing to suitable refinements of open covers of $M(\mathscr A_Z)$ and $U$ over which bundles $P_1$ and $P_2$ are trivialized and diminishing $U$, if necessary, without loss of generality we may assume that
$P_1$ is defined on a finite open cover $\mathfrak U_1=(U_i)_{i\in I_1}$ of $M(\mathscr A_Z)$ by a holomorphic $G$-valued cocycle $g=\{g_{ij}\in \mathcal O(U_i\cap U_j,G)\}_{i,j\in I_1}$ and there is a subset $I_2\subset I_1$ such that $P_2$ is defined on the finite open cover $\mathfrak U_2=(U_i)_{i\in I_2}$ of $U$ by a holomorphic $G$-valued cocycle $h=\{h_{ij}\in\mathcal O(U_i\cap U_j,G)\}_{i,j\in I_2}$. 
By the hypothesis of the theorem $P_1|_U$ and $P_2$ are topologically isomorphic, i.e., there exist $f_i\in C(U_i,G)$, $i\in I_2$, such that for all $i,j$
\begin{equation}\label{eq8.1}
f_i^{-1} g_{ij} f_j=h_{ij}\quad {\rm on}\quad  U_i\cap U_j.
\end{equation}
Applying to equations \eqref{eq8.1} 
the quotient homomorphism $q:G\rightarrow C(G):=G/G_0$ we get for all $i,j\in I_2$
\[
q(f_i)^{-1}\, q(g_{ij})\, q(f_j)=q(h_{ij})\quad {\rm on}\quad  U_i\cap U_j.
\]
This shows that the associated principal $C(G)$-bundles  $P_{1_{C(G)}}|_U$ and $P_{2_{C(G)}}$ are isomorphic.

Diminishing $U$, if necessary, without loss of generality we may assume that each $f_i$ admits a continuous extension to $\bar U_i$. Then the image of each $q(f_i)$ is finite.  In particular, there exist continuous locally constant maps $c_i: U_i\rightarrow G$ such that $q\circ c_i=q(f_i)$, $i\in I_2$. We set $c_i:=1_G$ for all $i\in I_1\setminus I_2$ and define a  holomorphic principal $G$-bundle $\widetilde P_1$ on $M(\mathscr A_Z)$ by cocycle $\tilde g=\{\tilde g_{ij}:=c_i^{-1}g_{ij}c_j\in\mathcal O(U_i\cap U_j,G)\}_{i,j\in I_1}$. Then $\widetilde P_1$ is holomorphically isomorphic to $P_1$ and $q(\tilde g)|_{\mathfrak U_2}=q(h)$.
To avoid abuse of notation in what follows we assume that cocycle $g$ initially has this property. 

Let $A(P_1)$ be a holomorphic fibre bundle on $M(\mathscr A_Z)$ with fibre $G_0$ defined on  the cover $\mathfrak U_1$ by holomorphic cocycle ${\rm Ad}(g)=\{{\rm Ad}(g_{ij})\in\mathcal O(U_i\cap U_j, {\rm Aut}(G_0))\}_{i,j\in I_1}$ (here ${\rm Ad}(u)(v):=u^{-1} v u$, $u\in G$, $v\in G_0$.) Consider the family of holomorphic sections $s=\{s_{ij}\in \Gamma_{\mathcal O}(U_i\cap U_j, A(P_1))\}_{i,j\in I_2}$ of $A(P_1)$ such that in the trivialization of $A(P_1)$ on $U_i$, 
\[
s_{ij}(x)=(x,h_{ij}(x)g_{ji}(x)),\quad x\in U_i.
\]
(Note that $s_{ij}$ is well-defined as $q(h_{ij}g_{ji})=1_{C(G)}$.)
\begin{Lm}\label{lem8.1}
Family $s$ forms a holomorphic $A(P_1)$-valued $1$-cocycle on  cover $\mathfrak U_2$.
\end{Lm}
\begin{proof}
We must show that $s_{ij} s_{jk} s_{ki}=1_{A(P_1)}$ on $U_i\cap U_j\cap U_k\ne\emptyset$, $i,j,k\in I_2$.  It suffices to check this in local coordinates of $A(P_1)$ on $U_i$. Then, for $x\in U_i\cap U_j\cap U_k$,
\[
\begin{array}{l}
\displaystyle
s_{ij}(x)=\bigl(x, h_{ij}(x)g_{ji}(x)\bigr),\quad s_{jk}(x)=\bigl(x,{\rm Ad}(g_{ij}^{-1}(x))(h_{jk}(x)g_{kj}(x))\bigr),\medskip\\
\displaystyle  s_{ki}(x)=\bigl(x,{\rm Ad}(g_{ik}^{-1}(x))(h_{ki}(x)g_{ik}(x))\bigr).
\end{array}
\]
Using these and that $\{g_{ij}\}_{i,j\in I_1}$ and $\{h_{ij}\}_{i,j\in I_2}$ are cocycles we obtain
\[
\begin{array}{l}
s_{ij}(x)s_{jk}(x)s_{ki}(x)\medskip\\
\qquad\qquad\qquad =\bigl(x,(h_{ij}(x)g_{ji}(x))\cdot (g_{ij}(x)h_{jk}(x)g_{kj}(x)g_{ji}(x))\cdot (g_{ik}(x)h_{ki}(x)g_{ik}(x)g_{ki}(x))\bigr)\medskip\\
\qquad\qquad\qquad=\bigl(x, h_{ij}(x)h_{jk}(x)h_{ki}(x)\bigr)=(x,1_G)=: 1_{A(P_1)}(x).
\end{array}
\]
\end{proof}
Let $\psi: G_u\rightarrow G_0$ be the universal covering of $G_0$. By the covering homotopy theorem, for each automorphism ${\rm Ad}(g)$, $g\in G$, of $G_0$ there is an automorphism $A(g)$ of $G_u$ such that $\psi\circ A(g)={\rm Ad}(g)\circ \psi$. Also, since the fundamental group of $G_0$ is naturally isomorphic to a subgroup of the center of $G_u$, it is readily seen that $A(g)={\rm Ad}(\tilde g)$ for all $\tilde g\in G_u$ such that $\psi(\tilde g)=g$. By definition, $A(g_1g_2)=A(g_1)  A(g_2)$ for all $g_1,g_2\in G$. Hence, $A: G\rightarrow {\rm Inn}(G_u)$, $g\mapsto A(g)$, is a homomorphism from $G$ to the group of inner automorphisms of $G_u$. Since $\psi$ is locally biholomorphic, the covering homotopy theorem implies that map $A(\cdot)(h): G\rightarrow G_u$ is holomorphic for each $h\in G_u$. 

 Let us consider a topological bundle $A_u(P_1)$ on $M(\mathscr A_Z)$ with fibre $G_u$ defined on open cover $\mathfrak U_1$ by  holomorphic transitions functions $\{A(g_{ij})\in\mathcal O(U_i\cap U_j\times G_u, G_u)\}_{i,j\in I_1}$ (cf. Section~6.5 above). The quotient map $\psi: G_u\rightarrow G_0$ induces a holomorphic bundle map
$\Psi: A_u(P_1)\rightarrow A(P_1)$ defined in local coordinates on $U_i$ as $\Psi(x,g):=(x,\psi(g))$, $(x,g)\in U_i\times G_u$, $i\in I$.
\begin{Proposition}\label{prop8.2}
There exists a holomorphic $A_u(P_1)$-valued $1$-cocycle $\tilde s=\{\tilde s_{ij}\}$ on $\mathfrak U_2$  such that $\Psi\circ \tilde s_{ij}=s_{ij}$ for all $i,j\in I_2$.
\end{Proposition}
\begin{proof}
Since $q(g_{ij})=q(h_{ij})$ for all $i,j\in I_2$, the hypothesis of the theorem implies that there exists $f_i\in C(U_i,G_0)$, $i\in I_2$, such that for all $i,j$
\[
f_i^{-1}g_{ij} f_j=h_{ij}\quad {\rm on}\quad  U_i\cap U_j.
\]
Let $s_k$ be a continuous section of $A(P_1)|_{U_k}$ such that $s_k(x)=(x,f_k(x))$, $x\in U_k$, in local coordinates on $U_k$. Then for $x\in U_i\cap U_j$ we obtain (using local coordinates on $U_i$)
\[
\begin{array}{l}
\displaystyle
s_i(x)s_{ij}(x)s_j^{-1}(x)=\bigl(x,f_i(x) (h_{ij}(x)g_{ji}(x)) ({\rm Ad}(g_{ij}^{-1}(x)))(f_j^{-1}(x))\bigr)\medskip\\
\displaystyle = \bigl(x,f_i(x)h_{ij}(x)f_j^{-1}(x)g_{ij}^{-1}(x)\bigr)=(x,1_{G})=1_{A(P_1)}(x).
\end{array}
\]
Thus for all $i,j\in I_2$
\[
s_{ij}=s_i^{-1} s_j\quad {\rm on}\quad U_i\cap U_j.
\]
Passing to a refinement of $\mathfrak U_1$ and then diminishing $U$, if necessary, without loss of generality we may assume that for each $i\in I_2$ the closure of $f_i(U_i)$ belongs to a simply connected open subset $V_i\subset G_0$. In particular, the holomorphic map $\psi^{-1}_i: V_i\rightarrow G_u$ inverse to $\psi|_{V_i}$ is defined.
Let $\tilde s_k$ be a continuous section of $A_u(P_1)|_{U_k}$ such that $\tilde s_ k(x)=(x, \psi^{-1}_k(f_k(x)))$, $x\in U_k$, in local coordinates on $U_k$. We set for all $i,j\in I_2$
\[
\tilde s_{ij}:=\tilde s_i^{-1}\tilde s_j\quad {\rm on}\quad U_i\cap U_j.
\]
Then
\[
\Psi\circ \tilde s_{ij}=s_{i}^{-1} s_{j}=s_{ij}\quad {\rm on}\quad U_i\cap U_j.
\]
Since map $\psi$ is locally biholomorphic, each $\tilde s_{ij}\in \Gamma_{\mathcal O}(U_i\cap U_j, A_u(P_1))$.
By definition,  family $\tilde s=\{\tilde s_{ij}\}_{i,j\in I_2}$ is an $A_u(P_1)$-valued $1$-cocycle on $\mathfrak U_2$.
\end{proof}

We set $I:=\{i\in I_2\, :\, V\cap U_i\ne\emptyset\}$  and $V_i:=V\cap U_i$, $i\in I$.

\begin{Proposition}\label{prop8.3}
 There exist sections $t_i\in \Gamma_\mathcal O (V_i,A_u(P_1))$, $i\in I$, such that for all $i,j\in I$
 \[
 t_i^{-1}t_j=\tilde s_{ij}\quad {\rm on}\quad V_i\cap V_j.
 \]
 \end{Proposition}
 \begin{proof}
 As in the proof of Theorem \ref{teo6.1} let us construct open covers
 $\mathfrak W^i=(W_j^i)_{1\le j\le N+1}$, $i\in\N$, of $\bar V$ consisting of relatively compact open subsets of sets of cover $\mathfrak U_2$ of $U$ satisfying 
 
 \noindent (a) all $Q_Z^{-1}(W_j^i)$, $2\le j\le N+1$, are open Blaschke sets  relatively compact in $M_a\setminus Z$; 
 
 \noindent (b) $W_{j}^{i+1}\Subset W_{j}^i$ for all $i,j$;
(c) all sets $\bar W_j^i$, $2\le j\le N+1$, $i\in\N$, are holomorphically convex.

For $1\le k\le N+1$ and $i\in\N$ we set
\[
Z_{k}^i:=\bigcup_{j=1}^k W_{j}^i.
\]
By $\{\tilde s_{ij}^1\}_{1\le i,j\le N+1}$ we denote the pullback of cocycle $\tilde s=\{\tilde s_{ij}\}_{i,j\in I_2}$ to cover $\mathfrak W^1$
by the refinement map. Using induction on $k\in\{1,\dots, N+1\}$ we prove the following statement:\smallskip

\noindent $(\circ)$ {\em There exist sections $t_i^k\in \Gamma_\mathcal O(W_i^k,A_u(P_1))$, $1\le i\le k$, such that for all $1\le i,j\le k$}
\begin{equation}\label{eq8.2}
(t_i^k)^{-1}\, t_j^k= \tilde s_{ij}^1\quad on\quad W_i^k\cap W_j^k.
\end{equation}

If $k=1$, the statement is trivial (we just set $t_1^1:=1_{A_u(P_1)}|_{W_1^1}$).

Assuming that the statement is valid for $k\in\{1,\dots, N\}$ let us prove it for $k+1$. 

We set $t_j^k:=1_{A_u(P_1)}|_{W_j^k}$ for $k+1\le j\le N+1$ and define a new holomorphic $A_u(P_1)$-valued $1$-cocycle $\bar s^k=\{\bar s_{ij}^k\}_{1\le i, j\le N+1}$ on $\mathfrak W^k$ by the formulas
\[
\bar s_{ij}^k:=t_i^k\, \tilde s_{ij}^1\, (t_j^k)^{-1}\quad {\rm on}\quad W_i^k\cap W_j^k,\quad 1\le i,j\le N+1.
\]
Then for all $1\le i,j\le k$ we have $\bar s_{ij}^k=1_{A_u(P_1)}$ on $W_i^k\cap W_j^k$. Moreover, restrictions of $\bar s_{i\,k+1}^k$, $1\le i\le k$, to $Z_k^k\cap W_{k+1}^k$ glue together to define a section $t_{k+1}\in \Gamma_\mathcal O(Z_k^k\cap W_{k+1}^k,A_u(P_1))$. Indeed, if $x\in (W_i^k\cap W_j^k)\cap W_{k+1}^k$, $1\le i,j\le k$,  then since $\bar s^k$ is a cocycle,
\[
\bar s_{i\,k+1}^k(x)\, (\bar s_{j\, k+1}^k(x))^{-1}=\bar s_{ij}^k(x)=1_{A_u(P_1)}(x).
\]
Hence,
\[
\bar s_{i\,k+1}^k=\bar s_{j\,k+1}^k\quad {\rm on}\quad (W_i^k\cap W_j^k)\cap W_{k+1}^k
\]
as required.

Next, consider the open cover $\{Z_k^{k+1}, W_{k+1}^{k+1}\}$ of $Z_{k+1}^{k+1}$.
By our definition, 
$\bar Z_k^{k+1}\cap\bar W_{k+1}^{k+1}\Subset Z_k^k\cap W_{k+1}^k$ and so $Q_Z^{-1}(\bar Z_k^{k+1}\cap\bar W_{k+1}^{k+1})$ is a compact subset of open Blaschke set $W_{k+1}^k$. Moreover, $\bar W_{k+1}^{k+1}$ is holomorphically convex. 
Choosing some open sets $U_i$, $i=1,2$, such that $Z_{k}^{k+1}\Subset U_1\Subset Z_{k}^k$ and $W_{k+1}^{k+1}\Subset U_2\Subset W_{k+1}^k$ and using that $t_{k+1}|_{\bar U_1\cap \bar U_2}$ belongs to the connected component of $1_{A_u(P_1)}|_{\bar U_1\cap\bar U_2}$ (this follows from Theorem \ref{compactif} if we write $t_{k+1}$ in local coordinates on $W_{k+1}^k$) we obtain by Theorem \ref{cartan} that there exist some sections $ t_{1\,k+1}\in \Gamma_\mathcal O (Z_{k}^{k+1}, A_u(P_1))$ and $t_{2\, k+1}\in \Gamma_\mathcal O (W_{k+1}^{k+1}, A_u(P_1))$ such that
\[
(t_{1\,k+1})^{-1}\,  t_{2\,k+1}= t_{k+1}\quad {\rm on}\quad Z_k^{k+1}\cap W_{k+1}^{k+1}.
\]
Let us define 
\[
t_i^{k+1}(x):=\left\{
\begin{array}{ccccc}
\displaystyle
t_{1\,k+1}(x)\, t_i^k(x),&x\in W_i^{k+1},&1\le i\le k,\medskip\\
t_{2\,k+1}(x),&x\in W_{k+1}^{k+1}.
\end{array}
\right.
\]
Then by the induction hypothesis for $x\in W_i^{k+1}\cap W_j^{k+1}$, $1\le i,j\le k$,
\[
(t_{i}^{k+1}(x))^{-1}\, t_j^{k+1}(x)=(t_{i}^k(x))^{-1}\, (t_{1\,k+1}(x))^{-1}\, t_{1\, k+1}(x)\,  t_j^k(x)=\tilde s_{ij}^1(x).
\]
Also, for $x\in W_i^{k+1}\cap W_{k+1}^{k+1}$, $1\le i\le k$,
\[
\begin{array}{l}
\displaystyle
(t_{i}^{k+1}(x))^{-1}\, t_j^{k+1}(x)=(t_{i}^k(x))^{-1}\, (t_{1\,k+1}(x))^{-1}\, t_{2\, k+1}(x)=(t_{i}^k(x))^{-1}\, t_{k+1}(x)\medskip
\\
\displaystyle \ \quad\qquad\qquad\qquad\quad\ =(t_{i}^k(x))^{-1}\,\bar s_{i\,k+1}^k(x)\, t_{k+1}^k(x)=\tilde s_{i\,k+1}^1.
\end{array}
\]

This completes the proof of the induction step and therefore  statement $(\circ)$  is valid for all $k\in\{1,\dots, N+1\}$.

Applying it to $k=N+1$ we obtain
\begin{equation}\label{eq8.3}
(t_i^{N+1})^{-1}\, t_{j}^{N+1}=\tilde s_{ij}^1\quad {\rm on}\quad W_i^{N+1}\cap W_j^{N+1}\quad {\rm for\ all}\quad 1\le i,j\le N+1.
\end{equation}

Let $\tau: \{1,\dots, N+1\}\rightarrow I_2$ be the refinement map of the refinement $\mathfrak W^{N+1}$ of $\mathfrak U_2$.
Suppose $\bar V_i:=\overline{V\cap U_i}$ is covered by sets $W_{i_s}^{N+1}$, $1\le s\le k_i$. We define $t_i\in \Gamma_\mathcal O(V_i,A_u(P_1))$ by the formulas
\begin{equation}\label{eq8.4}
t_i(x):=t_{i_s}^{N+1}(x)\, \tilde s_{\tau(i_s)i}(x)\quad {\rm for}\quad x\in W_{i_s}^{N+1}\cap V_i, \  s\in \{1,\dots,k_i\}.
\end{equation}
Let us show that $t_i$ is well-defined. Indeed, if $x\in V_i\cap (W_{i_{s_1}}^{N+1}\cap W_{i_{s_2}}^{N+1}) $, then by \eqref{eq8.3}
\[
\begin{array}{r}
\displaystyle
(\tilde s_{\tau(i_{s_1})i}(x))^{-1}\, (t_{i_{s_1}}^{N+1}(x))^{-1}\, t_{i_{s_2}}^{N+1}(x)\, \tilde s_{\tau(i_{s_2})i}(x)=\tilde s_{i\tau(i_{s_1})}(x)\,\tilde s_{\tau(i_{s_1})\tau(i_{s_2})}(x)\, s_{\tau(i_{s_2})i}(x)\medskip\\
\displaystyle =s_{ii}(x)=1_{A_u(P_1)}(x).\quad\qquad\qquad\quad\,
\end{array}
\]
Thus, the right-hand sides of \eqref{eq8.4} glue together over all nonempty intersections to define a  global section of $A_u(P_1)$ over $V_i$ as required. 

Further, for $x\in (V_i\cap W_{i_s}^{N+1})\cap (V_j\cap W_{j_l}^{N+1})$ we obtain using \eqref{eq8.3}
\[
\begin{array}{l}
\displaystyle
(t_i(x))^{-1}\, t_j(x)= (\tilde s_{\tau(i_s)i}(x))^{-1}\, (t_{i_s}^{N+1}(x))^{-1} \, t_{j_l}^{N+1}(x)\, \tilde s_{\tau(j_l)j}(x)\medskip\\
\displaystyle \ \quad\qquad\qquad\quad =\tilde s_{i\tau(i_s)}(x)\,  \tilde s_{\tau(i_s)\tau(j_l)}\, \tilde s_{\tau(j_l)j}(x)=\tilde s_{ij}(x).
\end{array}
\]

Since $V_i\cap V_j$ is the union off all nonempty sets of the form $(V_i\cap W_{i_s}^{N+1})\cap (V_j\cap W_{j_l}^{N+1})$, this completes the proof of the proposition.
 \end{proof}

To finish the proof of the theorem we set
\[
u_i:=\Psi\circ t_i\in \Gamma_\mathcal O(V_i, A(P_1)),\quad i\in I.
\]
Then Propositions \ref{prop8.2} and \ref{prop8.3} imply that for all $i,j\in I$
\begin{equation}\label{eq8.5}
u_i^{-1}u_j=s_{ij}\quad {\rm on}\quad V_i\cap V_j.
\end{equation}
Writing each $u_i$ in local coordinates on $V_i$,
\[
u_i(x)=(x, g_i(x)),\quad x\in U_i,\quad g_i\in \mathcal O(U_i,G),
\]
and arguing as in the proof of Proposition \ref{prop8.2} (but in the reverse order) for all $i,j\in I$ we obtain
\[
g_i^{-1}g_{ij}g_j=h_{ij}\quad {\rm on}\quad V_i\cap V_j.
\]
These identities show that bundles $P_1|_V$ and $P_2|_V$
are holomorphically isomorphic.

The proof of Theorem \ref{teo8.1} (and therefore of Theorem \ref{teo1.7}) is complete.
\end{proof}
%===
\sect{Proofs of Theorems \ref{teo2.9} and \ref{teo1.10}}
\subsection{Auxiliary result}
Let $\pi: P\rightarrow M(\mathscr A_Z)$ be a topological principal bundle with fibre a complex Banach Lie group $G$. 
\begin{Proposition}\label{prop9.1}
There exist a compact polyhedron $Q$ of dimension $\le 2$, a surjective continuous map $M(\mathscr A_Z)\rightarrow Q$ and a topological principal $G$-bundle $p: E\rightarrow Q$ such that  pullback $f^*E$ is isomorphic to $P$.
\end{Proposition}
The proof of this technical result would be much shorter if we knew that the classifying space of group $G$ is homotopy equivalent to an absolute neighbourhood retract. (In general, it is unknown.)
\begin{proof}
Let $\exp_G:\mathfrak g\rightarrow G_0$ be the exponential map of the Lie algebra $\mathfrak g$ of $G$ and $U'\subset\mathfrak g$ be an open ball centered at $0$ such that $\exp_G: U'\rightarrow V'=:\exp_G(U')$ is biholomorphic. We fix a smaller open ball $U\subset U'$ centered at $0$ such that $V\cdot V\subset V'$, where $V:=\exp_G(U)$.\smallskip

Passing to a refinement, if necessary, without loss of generality we may assume that $P$ is defined on a finite open cover $\mathfrak U=(U_i)_{i\in I}$, $I:=\{1,2,\dots,N\}\subset\N$, of $M(\mathscr A_Z)$ by a cocycle $c=\{c_{ij}\in C(\bar U_i\cap \bar U_j,G)\}_{i,j\in I}$ such that  each $c_{ij}$ with $i<j$  has a form $c_{ij}(x)=g_{ij}\, \tilde c_{ij}(x)$, $x\in \bar U_i\cap \bar U_j$, $g_{ij}\in G$, and $\tilde c_{ij}\in C(\bar U_i\cap\bar U_j, V)$. (Then for all $i>j$ we have $c_{ij}=c_{ji}^{-1}=\tilde c_{ji}^{-1}\, g_{ji}^{-1}$.)\smallskip

Next, since ${\rm dim}\, M(\mathscr A_Z)=2$, by the Marde\v{s}i\'c  theorem \cite[Th.\,1]{Ma} $M(\mathscr A_Z)$ can be presented as the inverse limit of the inverse system of metrizable compacta $\{Q_b, p_{bb'}\}$ with ${\rm dim}\, Q_b\le 2$; here $b$ ranges over  a directed set $B$ and $p^b: M(\mathscr A_Z)\rightarrow Q_b$ are the corresponding inverse limit projections. In turn, by the Freudenthal theorem, each $Q_b$ can be presented as the inverse limit of a sequence of compact polyhedra $\{Q_{bn}, p_{bnm}\}$ with ${\rm dim}\, Q_{bn}\le 2$; here $p^n_b: Q_b\rightarrow Q_{bn}$ are the corresponding inverse limit projections. Then by the Stone-Weierstrass theorem the union of pullback algebras $(p^b\circ  p_b^n)^*C(Q_{bn})$, $b\in B$, $n\in\N$, is dense in $C(M(\mathscr A_Z))$.\smallskip

Let us consider $\mathfrak g$-valued functions $c_{ij}':=\exp_G^{-1}(\tilde c_{ij})\in C(\bar U_i\cap\bar U_j,\mathfrak g)$, $i,j\in I$, $i<j$. Since the algebra of complex-valued  continuous functions on a compact Hausdorff spaces has the Grothendieck approximation property \cite{G}, each  $c_{ij}'$ can be uniformly approximated on $\bar U_i\cap\bar U_j$ by continuous functions from the algebraic tensor product $C(\bar U_i\cap\bar U_j)\otimes\mathfrak g$. In turn, extending functions in $C(\bar U_i\cap\bar U_j)\otimes\mathfrak g$ to continuous functions in $C(M(\mathscr A_Z))\otimes\mathfrak g$ by the Tietze-Urysohn extension theorem, we obtain that each $c_{ij}'$ can be uniformly approximated on $\bar U_i\cap\bar U_j$ by functions in algebras $\bigl((p^b\circ  p_b^n)^*C(Q_{bn})\bigr)\otimes\mathfrak g$, $b\in B$, $n\in\N$. From here, using that $\{c_{ij}\}_{i,j\in I}$ is a cocycle, we conclude  that there exist some $b_0\in B$, $n_0\in\N$ and functions $d_{ij}'\in \bigl((p^{b_0}\circ  p_{b_0}^{n_0})^*C(Q_{b_0n_0})\bigr)\otimes\mathfrak g$ sufficiently close to $c_{ij}'$ (with $i,j\in I$, $i<j$) such that for $x\in \bar U_i\cap\bar U_j$, $t\in [0,1]$ and 
\[
d_{ij}^{\,t}(x):=\left\{
\begin{array}{ccc}
g_{ij}\exp_G\bigl(td_{ij}'(x)+(1-t)c_{ij}'(x)\bigr)&{\rm if}&i<j\medskip\\ 
\exp_G\bigl(-td_{ji}'(x)-(1-t)c_{ji}'(x)\bigr) g_{ji}^{-1}&{\rm if}&i>j\medskip\\
1_G&{\rm if}&i=j,
\end{array}
\right.
\]
\begin{equation}\label{eq9.1}
\begin{array}{l}
c_{ij}(x) (d_{ij}^{\,t}(x))^{-1}\in V,\quad x\in U_i\cap U_j,\quad {\rm for\ all}\quad i, j\in I\quad {\rm and} \\ \\
\displaystyle
d_{ij}^{\,t} (x)\, d_{jk}^{\,t} (x)\, d_{ki}^{\,t} (x)\in V,\quad x\in  U_i\cap U_j\cap U_k,\quad
{\rm for\ all}\quad  i,j,k\in I.
\end{array}
\end{equation}

Further, by the definition of the inverse limit topology and since ${\rm dim}\,M(\mathscr A_Z)=2$, there exist $b_1\in B$, $b_1\ge b_0$ and $n_1\in\N$, $n_1\ge n_0$ and a finite open cover $\mathfrak V=(V_l)_{l\in L}$ of $Q_{b_1n_1}$  of order at most $3$ such that cover $\bigl((p^{b_0}\circ  p_{b_0}^{n_0})^{-1}(\bar V_l)\bigr)_{l\in L}$ is a refinement of $\mathfrak U$. Since by the definition of the inverse limit, 
$(p^{b_0}\circ  p_{b_0}^{n_0})^*C(Q_{b_0n_0})\subset (p^{b_1}\circ  p_{b_1}^{n_1})^*C(Q_{b_1n_1})$, there exist functions $e_{ij}'\in C(Q_{b_1n_1})\otimes \frak g$ such that $(p^{b_1}\circ  p_{b_1}^{n_1})^*(e_{ij}')=d_{ij}'$ for all $i,j\in I$ with $i<j$. 

We will assume without loss of generality that $Q_{b_1n_1}$ is a simplicial complex in some $\RR^d$. By $\|\cdot\|_2$ we denote the Euclidean norm on $\RR^d$. Then  uniform continuity of $e_{ij}'$ implies that for every $\varepsilon>0$ there exists $\delta>0$ such that  for all $x,y\in Q_{b_1n_1}$ with $\|x-y\|_2<\delta$ and all $i,j\in I$ with $i<j$,  
\begin{equation}\label{eq9.2}
\| e_{ij}'(x)- e_{ij}'(y)\|_{\mathfrak g}<\varepsilon.
\end{equation}
\begin{Lm}\label{lem9.2}
There exist a compact polyhedron $Q$ of dimension $\le 2$, a continuous surjective map $f: M(\mathscr A_Z)\rightarrow Q$ and a continuous map $g: Q\rightarrow Q_{b_1n_1}$ such that the finite open cover
$\bigl((g\circ f)^{-1}(V_l)\bigr)_{l\in L}$ of $M(\mathscr A_Z)$ is a refinement of $\mathfrak U$ and the analog of \eqref{eq9.1} is valid with $d_{ij}^{\, t}(x)$ replaced by 
\[
e_{ij}^{\, t}(x):=\left\{
\begin{array}{ccc}
g_{ij} \exp_G\bigl(t((g\circ f)^* e_{ij}')(x)+(1-t)c_{ij}'(x)\bigr)&{\rm if}&i<j\medskip\\ 
\exp_G\bigl(-t((g\circ f)^* e_{ji}')(x)-(1-t)c_{ji}'(x)\bigr) g_{ji}^{-1}&{\rm if}&i>j\medskip\\
1_G&{\rm if}&i=j,
\end{array}
\right.
\]
$x\in \bar U_i\cap\bar U_j$,  $t\in [0,1]$, $i, j\in I$.
\end{Lm}
\begin{proof}
By \cite[Lm.\,1]{Ma} for every $\delta >0$ there exist a compact polyhedron $Q_\delta$ of dimension $\le 2$,  a continuous surjective map $f_\delta:M(\mathscr A_Z)\rightarrow Q$ and a continuous map $g_\delta: Q\rightarrow Q_{b_1n_1}$ such that for all $x\in M(\mathscr A_Z)$
\begin{equation}\label{eq9.3}
\|(g_\delta\circ f_\delta)(x)-(p^{b_1}\circ p_{b_1}^{n_1})(x)\|_2< \delta.
\end{equation}
This and \eqref{eq9.2} imply that for all $x\in M(\mathscr A_Z)$ and $i,j\in I$ with $i<j$
\[
\| d_{ij}'(x)-((g_\delta\circ f_\delta )^*e_{ij}')(x)\|_{\mathfrak g}<\varepsilon.
\]
Choosing here sufficiently small $\varepsilon$ and $\delta$ we obtain 
by continuity of $\exp_G$ and operations on $G$ that the analog of \eqref{eq9.1} is valid with $d_{ij}^{\, t}$ replaced by $e_{ij}^{\, t}$ for all $i, j\in I$.

Next, let $\tau : L\rightarrow I$ be the refinement map for the refinement $\bigl((p^{b_0}\circ  p_{b_0}^{n_0})^{-1}(\bar V_l)\bigr)_{l\in L}$ of $\mathfrak U$. Since $(p^{b_0}\circ  p_{b_0}^{n_0})^{-1}(\bar V_l)$ is a compact subset of $U_{\tau(l)}$, inequality \eqref{eq9.3} implies easily that for a sufficiently small $\delta$ and all $l\in L$,
$(g_\delta\circ f_\delta)^{-1}(V_l)\Subset U_{\tau(l)}$ as well. For such  $\delta$ we set $Q:=Q_\delta$, $f:=f_\delta$ and $g:=g_\delta$.
\end{proof}

We set $\widetilde{\mathfrak V}:=(\tilde V_l)_{l\in L}$, $\tilde V_l:=g^{-1}(V_l)$.  Then $\widetilde{\mathfrak V}$ is a finite open cover of $Q$ of order $\le 3$.
Since ${\rm dim}\, Q\le 2$, according to the Ostrand theorem on colored dimension \cite[Th.\,3]{O} there exists a finite open cover $\mathfrak W$ of $Q$ which can be represented as the union of families $\mathcal W_1$, $\mathcal  W_2$, $\mathcal W_3$, where $\mathcal W_i=\{W_{i,l}\}_{l\in L}$ are such that closures of their subsets are pairwise disjoint  and $\bar W_{i,l}\subset \tilde V_l$ for each $i\in \{1,2,3\}$ and $l\in L$.

\noindent As before, $\tau : L\rightarrow I$ stands for the refinement map for the refinement $\bigl((p^{b_0}\circ  p_{b_0}^{n_0})^{-1}(\bar V_l)\bigr)_{l\in L}$ of $\mathfrak U$.
Lemma \ref{lem9.2} implies that for all $i,j\in I$ maps $\tilde e_{ij}\in C(Q,G)$, 
\[
\tilde e_{ij}:=\left\{
\begin{array}{ccc}
g_{ij}\bigl(\exp_G\circ (g^*e_{ij}')\bigr)&{\rm if}&i<j\medskip\\
\bigl(\exp_G\circ (-g^*e_{ji}')\bigr)g_{ji}^{-1}&{\rm if}&i>j\medskip\\
1_G&{\rm if}&i=j,
\end{array}
\right.
\] 
satisfy
\[
\begin{array}{r}
\tilde e_{\tau(l)\tau(m)}(x)\, \tilde e_{\tau(m)\tau(n)}(x)\, \tilde e_{\tau(n)\tau(l)}(x)\in V,\quad x\in  
\tilde V_l\cap \tilde V_m\cap \tilde V_n,
\quad {\rm for\ all}\quad  l,m,n\in L.
\end{array}
\]
We set $W_i=\cup_{l\in L}W_{i,l}$, $i=1,2,3$, and consider the closed cover $(\bar W_i)_{i=1}^3$ of $Q$. By definition,
\[
\bar W_i\cap \bar W_j=\bigsqcup_{k\in L}\left(\,\bigsqcup_{l\in L}\bar W_{i,k}\cap  \bar W_{j,l}\right);
\]
here each nonempty $\bar W_{i,k}\cap  \bar W_{j,l}$ is a compact subset of $\tilde V_k\cap \tilde V_l$. 

\noindent For  $i, j\in \{1,2,3\}$ we define maps $\bar e_{ij}\in C(\bar W_i\cap\bar W_j,G)$ by the formulas
\begin{equation}\label{eq9.4}
\bar e_{ij}(x):=\tilde e_{\tau(k)\tau(l)}(x),\quad x\in \bar W_{i,k}\cap  \bar W_{j,l},\quad k,l\in L.
\end{equation}
\begin{Lm}\label{lem9.3} 
Suppose $\bar W_1\cap \bar W_2\cap \bar W_3\ne\emptyset$. Then there exists a map $\bar e\in C(\bar W_1\cap\bar W_3,V)$ such that
\[
\bar e_{12}(x)\,\bar e_{23}(x)\, (\bar e_{31}(x)\, \bar e(x))=1_G\quad {\rm for\ all}\quad x\in \bar W_1\cap \bar W_2\cap \bar W_3.
\]
\end{Lm}
\begin{proof}
We set $\bar e_{123}(x):=\bar e_{12}(x)\,\bar e_{23}(x)\, \bar e_{31}(x)$, $x\in \bar W_1\cap \bar W_2\cap \bar W_3$. By definition, $\bar e_{123}\in C(\bar W_1\cap \bar W_2\cap \bar W_3, V)$; hence map $\exp_G^{-1}\circ \bar e_{123}\in C(\bar W_1\cap \bar W_2\cap \bar W_3, U)$ is well defined. Since $U\subset\mathfrak g$ is a ball, by the Dugundji extension theorem there exists some $e'\in C(\bar W_1\cap\bar W_3,U)$ such that $e'=-\exp_G^{-1}\circ \bar e_{123}$ on $\bar W_1\cap \bar W_2\cap \bar W_3$. We set $\bar e:=\exp_G\circ e'$. Then the required identity is fulfilled.  
\end{proof}

For $\bar W_1\cap \bar W_2\cap \bar W_3=\emptyset$, we define $\bar e\in  C(\bar W_1\cap\bar W_3,V)$ to be the constant map with constant value $1_G$.

Next, we define $h_{ij}:=\bar e_{ij}$ for $i,j\in\{1,2,3\}$ with $(i,j)\not\in\{(1,3), (3,1)\}$ and $h_{31}:=\bar e_{31}\,\bar e$,  $h_{13}=\bar e^{-1}\, \bar e_{13}$.
\begin{Lm}\label{lem9.4}
Family $\bigl\{h_{ij}\in C( \bar W_{i}\cap   \bar W_{j}, G)\bigr\}_{1\le i,j\le 3}$ is a $G$-valued $1$-cocycle on the cover $(\bar W_i)_{i=1}^3$ of $Q$. 
\end{Lm}
\begin{proof}
By the definition of $\bar e_{ij}$, cf \eqref{eq9.4}, $h_{ij}=h_{ji}^{-1}$ for all $i,j\in\{1,2,3\}$. Clearly, it suffices to prove the result for $\bar W_1\cap\bar W_2\cap \bar W_3\ne\emptyset$. In this case it suffices to check only that
\[
h_{12}\,h_{23}\, h_{31}\equiv 1_G\quad {\rm on}\quad \bar W_1\cap\bar W_2\cap \bar W_3.
\]
This follows straightforwardly from Lemma \ref{lem9.3} and  the definition of $h_{ij}$.
\end{proof}

Let us define the required principal $G$-bundle $p:E\rightarrow Q$ of the proposition by cocycle $h=\{h_{ij}\}$ on the cover $(\bar W_i)_{i=1}^3$ of $Q$ (see \eqref{equ1.3} above).  Then we must prove the following result.
\begin{Lm}\label{lem9.5}
Topological principal $G$-bundles $f^*E$  and $P$ on $M(\mathscr A_Z)$ are isomorphic.
\end{Lm}
\begin{proof}
We set $\mathfrak W^*=( W_{i,l}^*)_{l\in  L, 1\le i\le 3}$, where  $ W_{i,l}^*:=f^{-1}(W_{i,l})$. The finite open cover $\mathfrak W^*$ of $M(\mathscr A_Z)$ is a refinement of the cover $((g\circ f)^{-1}(V_l))_{l\in L}$ which, in turn, is a refinement of the cover $\mathfrak U$ (see Lemma \ref{lem9.2}). Also,  we consider the closed cover $(\bar W_i^*)_{i=1}^3$ of $M(\mathscr A_Z)$ where $W_i^*:=f^{-1}(W_i)$, $i=1,2,3$. 

For $i, j\in \{1,2,3\}$ we define maps $\ell_{ij}\in C(\bar W_i^*\cap\bar W_j^*, C([0,1],G))$ by the formulas
\[
(\ell_{ij}(x))(t):=e_{\tau(k)\tau(l)}^{\, t}(x),\quad x\in \bar W_{i,k}^*\cap\bar W_{j,l}^*,\quad k, l\in L.
\]
By definition, see Lemma \ref{lem9.2} and the text before \eqref{eq9.4},
\begin{equation}\label{eq9.5}
(\ell_{ij}(x))(0)=c_{\tau(k)\tau(l)}(x),\quad 
(\ell_{ij}(x))(1)=(f^*\tilde e_{\tau(k)\tau(l)})(x),\quad x\in \bar W_{i,k}^*\cap\bar W_{j,l}^*,\ k,l\in L.
\end{equation}

Note that $C([0,1],G)$ is a complex Banach Lie group with respect to the pointwise multiplication of maps whose Lie algebra is $C([0,1],\frak g)$.

For $\bar e$ as in Lemma \ref{lem9.3} we define a map $\ell\in C(\bar W_1^*\cap\bar W_3^*, C([0,1],V))$ by the formula
\begin{equation}\label{eq9.6}
(\ell(x))(t):= \exp_G\bigl(t \exp_G^{-1}(\bar e(f(x)))\bigr), \quad x\in \bar W_1^*\cap\bar W_3^*,\quad
t\in [0,1].
\end{equation}
Next, we consider a continuous $C([0,1],G)$-valued map
\[
\ell_{123}(x):=\ell_{12}(x)\,\ell_{23}(x)\, (\ell_{31}(x)\, \ell(x)),\quad x\in \bar W_1^*\cap\bar W_2^*\cap\bar W_3^*.
\]
Due to Lemmas \ref{lem9.2}, \ref{lem9.3} and since $\{c_{ij}\}_{i,j\in I}$ is a cocycle on $\mathfrak U$, the image of $\ell_{123}$ consists of continuous maps  $F\in C([0,1],V')$ (recall that $V\cdot V\subset V'$) such that $F(0)=F(1)=1_G$. Therefore  map $\ell_{123}'\in C( \bar W_1^*\cap\bar W_2^*\cap\bar W_3^*, C([0,1],U')$,
\[
(\ell_{123}'(x))(t):=\exp_{G}^{-1}\bigl((\ell_{123}(x))(t)\bigr),\quad x\in  \bar W_1^*\cap\bar W_2^*\cap\bar W_3^*,\quad t\in [0,1],
\]
is well defined and its image consists of maps $F\in C([0,1],U')$ such that $F(0)=F(1)=0$. The space $C_0([0,1],\mathfrak g)$ of continuous maps $F: [0,1]\rightarrow \mathfrak g$ with $F(0)=F(1)=0$ is a closed subspace  of the Banach space $C([0,1],\mathfrak g)$ and for a closed ball $B\subset\mathfrak g$ centered at $0$ set
$C_0([0,1],\mathfrak g)\cap C([0,1], B)$ is a closed ball centered at zero map in $C_0([0,1],\mathfrak g)$. In particular, we can apply to $\ell_{123}'$ the Arens extension theorem, see, e.g.,  \cite[Th.\,4.1]{A1}, to find some  $\ell'\in C\bigl(\bar W_1^*\cap\bar W_3^*,C_0([0,1],\mathfrak g)\cap C([0,1],U')\bigr)$ which extends $\ell'_{123}$.
We define $\bar\ell\in C(\bar W_1^*\cap\bar W_3^*,C([0,1],V'))$ by the formula
\[
(\bar\ell(x))(t):=\exp_G\bigl((-(\ell'(x))(t)\bigr),\quad x\in \bar W_1^*\cap\bar W_3^*,\quad t\in [0,1].
\]
Then 
\begin{equation}\label{eq9.7}
\ell_{12}(x)\,\ell_{23}(x)\, (\ell_{31}(x)\, \ell(x)\,\bar \ell(x))={\bf 1}_G\quad {\rm for\ all}\quad x\in \bar W_1^*\cap\bar W_2^*\cap\bar W_3^*,
\end{equation}
where ${\bf 1}_{G}(t)=1_G$ for all $t\in [0,1]$.

Next, let us define $\kappa_{ij}:=\ell _{ij}$ for $i,j\in\{1,2,3\}$ with $(i,j)\not\in\{(1,3), (3,1)\}$ and $\kappa_{31}:=\ell_{31}\, \ell\,\bar \ell$,  $\kappa_{13}=\bar \ell^{-1}\,\ell^{-1}\,\ell_{13}$.
Then similarly to Lemma \ref{lem9.4} we obtain that
family $\kappa:=\bigl\{\kappa_{ij}\in C( \bar W_{i}\cap   \bar W_{j}, C([0,1],G))\bigr\}_{1\le i,j\le 3}$ is a continuous $C([0,1],G)$-valued $1$-cocycle on the cover $(\bar W_i^*)_{i=1}^3$ of $M(\mathscr A_Z)$. Since $\mathfrak W^*$ is the natural refinement of the cover $(\bar W_i^*)_{i=1}^3$, the pullback of $\kappa$ to $\mathfrak W^*$ by the refinement map determines a family of continuous $G$-valued $1$-cocycles $c(t)$ on $\mathfrak W^*$ depending continuously on $t\in [0,1]$. By definitions of $\kappa_{ij}$ and $\bar\ell$, see \eqref{eq9.5} \eqref{eq9.6}, $c(0)$ is the pullback to $\mathfrak W^*$ by means of the refinement map $\tau$ of the cocycle $c$ determining bundle $P$. On the other hand, $c(1)$ is the pullback to $\mathfrak W^*$ of the cocycle $f^*h=\{f^*h_{ij}\}$ determining $f^*E$. Thus, the general result of the fibre bundles theory \cite[Ch.\,4,\,Cor.\,9.7]{Hus} implies that $P$ and $f^*E$ are isomorphic.
\end{proof}

The proof of the proposition is complete.
\end{proof}
\subsect{Proof of Theorem \ref{teo2.9}}
Let $A\in\mathscr A_Z$ and
$D$ be the set of all finite subsets of $A$ directed by inclusion $\subseteq$. Then $M(\mathscr A_Z)$ is naturally identified with $M(A)$ and is presented as the inverse limit of the inverse system of compacta $\{M(A_\alpha), F_\beta^\alpha\}$, where $\alpha$ ranges over  $D$ and $A_\alpha$ is the unital closed subalgebra of $A$ generated by $\alpha$, see Section~2.2.
Recall that if $\alpha=\{f_1,\dots,f_n\}\in D$,  then $F_\alpha=(\hat f_1,\dots, \hat f_n) : M(\mathscr A_Z)\rightarrow \Co^{n}$ is the inverse limit projection and $M(A_\alpha)$ is the polynomially convex hull of its image.
Also, if $\alpha,\beta\in D$ with $\alpha\supseteq\beta$, then $F_\beta^\alpha:\Co^{\#\alpha}\rightarrow\Co^{\#\beta}$, $\Co^{\#\alpha}\ni (z_1,\dots, z_{\#\alpha})\mapsto (z_1,\dots, z_{\#\beta})\in\Co^{\#\beta}$. 

On the other hand, $M(\mathscr A_Z)$ is the inverse limit of inverse system $\{K_\alpha, F_\beta^\alpha\}$, where $K_\alpha:=F_\alpha (M(\mathscr A_Z))\, (\subset M(A_\alpha))$. Therefore by the definition of the inverse limit topology, for each $\beta\in D$ and an open neighbourhood $O\Subset\Co^{\#\beta}$ of $K_\beta$ there exists $\alpha\supseteq\beta$ such that
\begin{equation}\label{eq9.8}
M(A_\alpha)\subset (F_\beta^\alpha)^{-1}(O).
\end{equation}

Let  $\pi: P\rightarrow M(\mathscr A_Z)$, $Q$ and $f\in C(M(\mathscr A_Z),Q)$ be as in Proposition \ref{prop9.1}.
\begin{Proposition}\label{prop9.6}
There exists $\beta\in D$ and a continuous map $g: K_\beta\rightarrow Q$ such that maps $f,\, g\circ F_\beta\in C(M(\mathscr A_Z),Q)$  are homotopic.
\end{Proposition}
\begin{proof}
Without loss of generality we may assume that $Q$ is a (finite) simplicial complex in some $\RR^d$ (recall that $Q$ is a compact polyhedron). Then there exists an open neighbourhood $U\Subset\RR^d$ and a continuous retraction $r: U\rightarrow Q$.  Let $A_\beta^*\subset C(M(\mathscr A_Z))$ be the (nonclosed) subalgebra generated by functions in $A_\beta$ and their complex conjugate. By the Stone-Weierstrass theorem algebra $\cup_{\beta\in D}A_\beta^*$ is dense in $C(M(\mathscr A_Z))$. In particular, regarding $f$ as a map in $C(M(\mathscr A_Z), \RR^d)$ we can find $\beta\in D$ and a map $f_\beta\in C(M(\mathscr A_Z),\RR^d)$ with coordinates in $A_\beta^*$ such that 
\begin{equation}\label{eq9.9}
t f_\beta(x)+(1-t) f(x)\in U\quad {\rm for\ all}\quad x\in M(\mathscr A_Z),\ t\in [0,1].
\end{equation}
Next, by definitions of $K_\beta$ and $F_\beta$, there exists a map $g_\beta\in C(K_\beta,\RR^d)$ such that $f_\beta=g_\beta\circ F_\beta$. We set $g:=r\circ g_\beta\in C(K_\beta,Q)$. Then by \eqref{eq9.9} 
\[
H(x,t):=r(t f_\beta(x)+(1-t) f(x)),\quad  x\in M(\mathscr A_Z),\ t\in [0,1],
\]
is a homotopy between $f$ and $g\circ F_\beta$.
\end{proof}

Using the Tietze-Urysohn extension theorem composed with retraction $r\in C(U,Q)$ we  extend the map $g$ of the proposition to a map $\tilde g\in C(O,Q)$ for some open neighbourhood $O$ of $K_\beta$. Let $\alpha\supseteq\beta$ be such that $M(A_\alpha)\subset (F_\beta^\alpha)^{-1}(O)$ (cf. \eqref{eq9.8}). Since $M(A_\alpha)$ is polynomially convex, there is a Stein neighbourhood $N$ of $M(A_\alpha)$ contained in $(F_\beta^\alpha)^{-1}(O)$.

Next, let $p: E\rightarrow Q$ be the topological principal $G$-bundle of Proposition \ref{prop9.1}.  
Consider principal $G$-bundle $(\tilde g\circ F^\alpha_\beta|_N)^*E$ on $N$. Since $N$ is Stein,  by the Bungart theorem \cite[Th.\,8.1]{Bu} this bundle is isomorphic to a holomorphic principal $G$-bundle $\tilde E$ on $N$. Further, since by Proposition \ref{prop9.6} maps $f$ and $g\circ F_\beta$ are homotopic,  bundles $f^*E$ and $(g\circ F_\beta)^*E=F_\alpha^*((g\circ F^\alpha_\beta)^*E)$ on $M(\mathscr A_Z)$ are isomorphic (see, e.g., \cite[Ch.\,4]{Hus}). These imply that the topological principal $G$-bundle $P\,(\cong f^*E)$ is
isomorphic to the holomorphic principal $G$-bundle $F_\alpha^*\tilde E$.

The proof of the theorem is complete.
\subsection{Proof of Theorem \ref{teo1.10}}
First, we show that {\em for each holomorphic principal $G$-bundle  $P$ on $M(\mathscr A_Z)$ there exist $\alpha\in D$ and a holomorphic principal $G$-bundle $\tilde P$ defined on a neighbourhood of $M(A_\alpha)$ such that $P$ and $F_\alpha^*\tilde P$ are holomorphically isomorphic.}

Indeed, as in the proof of Theorem \ref{teo2.9} we  find such $\tilde P$ and $\alpha\in D$ that $F_\alpha^*\tilde P$  and $P$ are isomorphic as topological bundles. Then due to Theorem \ref{teo1.7} they are holomorphically isomorphic as required.\smallskip

Second, we show that  {\em if holomorphic principal $G$-bundles $P_1,P_2$ defined on a neighbourhood of $M(A_\beta)$ are such that $F_\beta^*P_1$ and $F_\beta^*P_2$ are holomorphically isomorphic, then there exists $\alpha\ge\beta$ such that the holomorphic principal $G$-bundles $(F_\beta^\alpha)^*P_1$ and $(F_\beta^\alpha)^*P_2$ defined on a neighbourhood of $M(A_\alpha)$ are isomorphic.}

This will complete the proof of the theorem.\smallskip

In the proof of the preceding statement with use the following general result.

Let a compact Hausdorff space $X$ be the inverse limit of the inverse limit system of compacta $\{X_\alpha, p_{\beta}^\alpha\}$ where $\alpha$ ranges over a directed set $\Lambda$ and $p_\alpha: X\rightarrow X_\alpha$ are the corresponding inverse limit projections. Let $E_1$ and $E_2$ be topological principal $G$-bundles on some $X_\beta$ with fibre $G$ a complex Banach Lie group.
\begin{Proposition}\label{prop10.1}
Suppose principal $G$-bundles $p_\beta^*E_1$ and $p_\beta^*E_2$ on $X$ are isomorphic. Then there exists some $\alpha\ge \beta$ such that principal $G$-bundles $(p_\beta^\alpha)^*E_1$ and $(p_\beta^\alpha)^*E_2$ on $X_\alpha$ are isomorphic.
\end{Proposition}
As in the case of Proposition  \ref{prop9.1}, the proof would be  much shorter if we knew that the classifying space of group $G$ is homotopy equivalent to an absolute neighbourhood retract. 
\begin{proof}
Let $U\subset\mathfrak g$ be an open ball centered at $0$ of the Lie algebra $\mathfrak g$ of $G$ such that the exponential map $\exp_G:\mathfrak g\rightarrow G$ is biholomorphic on $U$. We set $V:=\exp_G(U)$.

Without loss of generality we may assume that $E_1$ and $E_2$ are defined on a finite open cover $\mathfrak U=(U_i)_{i\in I}$ of $X_\beta$ by cocycles  $\{c_{ij}^k\in C(U_i\cap U_j,G)\}_{i,j\in I}$, $k=1,2$. By the hypothesis, there exist some $c_i\in C(p_\beta^{-1}(U_i),G)$, $i\in I$, such that
\begin{equation}\label{eq10.1}
c_i^{-1}\cdot p_\beta^* c_{ij}^1\cdot c_j= p_\beta^* c_{ij}^2\quad {\rm on}\quad p_\beta^{-1}(U_i)\cap p_\beta^{-1}(U_j),\quad i,j\in I.
\end{equation}
Then there exist a finite open refinement $\mathfrak V=(V_l)_{l\in L}$ of the cover $(p_\beta^{-1}(U_i))_{i\in I}$ of $X$ with refinement map $\tau: L\rightarrow I$,  maps $e_l\in C(V_l,V)$ and elements $d_l\in G$ such that
\begin{equation}\label{eq10.2}
c_{\tau(l)}=d_{l} \cdot e_l\quad {\rm on}\quad V_l,\quad l\in L.
\end{equation}
By the definition of the inverse limit topology, there exist some $\tilde\beta\in\Lambda$  and a finite open cover $\mathfrak W=(W_k)_{k\in K}$ of $X_{\tilde \beta}$ such that $(p_{\tilde\beta}^{-1}(W_k))_{k\in K}$ is a refinement of $\mathfrak V$ with refinement map $\sigma: K\rightarrow L$ and $p_{\tilde\beta}^{-1}(\bar W_k)\subset V_{\sigma(k)}$ for all $k\in K$.

We set for all $k,m\in K$, $i=1,2$,
\begin{equation}\label{equ8.12}
f_{km}^i:=c^i_{(\tau\circ\sigma)(k)(\tau\circ\sigma)(m)},\quad f_k:=c_{(\tau\circ\sigma)(k)},\quad g_k:=d_{\sigma(k)},\quad h_k:=e_{\sigma(k)}.
\end{equation}
Then \eqref{eq10.1}, \eqref{eq10.2} imply
\begin{equation}\label{eq10.4}
f_k^{-1}\cdot p_\beta^* f_{km}^1\cdot f_m= p_\beta^* f_{km}^2\quad {\rm on}\quad p_{\tilde\beta}^{-1}(\bar W_k)\cap p_{\tilde \beta}^{-1}(\bar W_m),\quad k,m\in K.
\end{equation}
\begin{equation}\label{eq10.5}
f_{k}=g_k\cdot h_k\quad {\rm on}\quad p_{\tilde\beta}^{-1}(\bar W_k),\quad k\in K.
\end{equation}
Here $h_k\in C(p_{\tilde\beta}^{-1}(\bar W_k),V)$, $k\in K$.

Since by the Stone-Weierstrass theorem algebra $\cup_{\alpha\in\Lambda}\,p_\alpha^* C(X_\alpha)$ is dense in $C(X)$, by the Tietze-Urysohn extension theorem each $\exp_G^{-1}(h_k)$ can be uniformly approximated by maps in algebras $p_\alpha^* C(X_\alpha)\otimes\mathfrak g$.

Let us construct a sequence of open balls $U_s\subset U$, $1\le s\le \bar k:=\#K$, centered at $0$ such that for $V_s:=\exp_G(U_s)$ 
\begin{equation}\label{eq10.6}
V_{s}\cdot f_{km}^2(x)\cdot  V_{s}\cdot V_s\subset f_{km}^2(x)\cdot V_{s+1}\quad {\rm for\ all}\quad x\in p_{\tilde\beta}^{-1}(\bar W_k)\cap p_{\tilde \beta}^{-1}(\bar W_m),\ \, k,m\in K.
\end{equation}

We set $U_{\bar k}:=U$. Since each $f_{km}^{2}\cdot V$ is an open neighbourhood of $f_{km}^2$ in complex Banach Lie group $C\bigl(p_{\tilde\beta}^{-1}(\bar W_k)\cap p_{\tilde \beta}^{-1}(\bar W_m),G\bigr)$,  set $K$ is finite and each $p_{\tilde\beta}^{-1}(\bar W_k)\cap p_{\tilde \beta}^{-1}(\bar W_m)$ is  compact, due to  continuity of the product on $G$ there is a ball $U_{\bar k-1}\subset U_{\bar k}$ such that $V_{\bar k-1}\cdot f_{km}^2\cdot  V_{\bar k-1}\cdot  V_{\bar k-1}\subset f_{km}^2\cdot V_{\bar k}$ for all $k,m\in K$. Applying the same argument to $f_{km}^2\cdot V_{\bar k-1}$ we construct the required ball $U_{\bar k-2}$, etc.\smallskip

Next, let us choose $\tilde\alpha\ge\tilde\beta$ and elements $\ell_k'\in C(X_{\tilde\alpha})\otimes\mathfrak g$ such that $p_{\tilde\alpha}^*\ell_k|_{p_{\tilde\beta}^{-1}(\bar W_k)}$ with $\ell_k:=g_k\cdot\exp_{G}(\ell_k')$ are so close to $f_k$ that
(cf. \eqref{eq10.4})
\begin{equation}\label{eq10.7}
p_{\tilde\alpha}^*\ell_k^{-1}\cdot p_\beta^* f_{km}^1\cdot p_{\tilde\alpha}^*\ell_m\subset p_\beta^* f_{km}^2\cdot V_1\quad {\rm on}\quad p_{\tilde\beta}^{-1}(\bar W_k)\cap p_{\tilde \beta}^{-1}(\bar W_m),\ \, k,m\in K.
\end{equation}
\begin{Lm}\label{lem10.2}
There exist $\alpha\ge\tilde\alpha$ and maps $r_k\in C\bigl((p_{\tilde\beta}^\alpha)^{-1}(\bar W_k),V\bigr)$, $k\in K$, such that for all $k,m\in K$
\[
r_k^{-1}\cdot (p^\alpha_{\tilde\alpha})^*\ell_k^{-1}\, (p^\alpha_\beta)^* f_{km}^1\cdot (p^\alpha_{\tilde\alpha})^*\ell_m\cdot r_m=(p^\alpha_\beta)^* f_{km}^2\quad {\rm on}\quad (p^\alpha_{\tilde\beta})^{-1}(\bar W_k)\cap (p^\alpha_{\tilde\beta})^{-1}(\bar W_m).
\]
\end{Lm}
\begin{proof}
Due to \eqref{eq10.7} for all $k,m\in K$,
\[
\ell_k^{-1}\cdot (p^{\tilde\alpha}_\beta)^* f_{km}^1\cdot \ell_m\subset  (p^{\tilde\alpha}_\beta)^* f_{km}^2\cdot V_1\quad {\rm on}\quad \bigl((p^{\tilde\alpha}_{\tilde\beta})^{-1}(\bar W_k)\cap (p^{\tilde\alpha}_{\tilde \beta})^{-1}(\bar W_m)\bigr)\cap p_{\tilde\alpha}(X).
\]
By continuity, there is an open neighbourhood $N\subset X_{\tilde\alpha}$ of $p_{\tilde\alpha}(X)$ such that analogous implications are still valid on all  $\bigl(p_{\tilde\beta}^{-1}(\bar W_k)\cap p_{\tilde \beta}^{-1}(\bar W_m)\bigr)\cap N$. By the definition of the inverse limit topology, there exists $\alpha\ge\tilde\alpha$ such that $X_\alpha=(p^\alpha_{\tilde\alpha})^{-1}(N)$. Hence, for all $k,m\in K$,
\begin{equation}\label{eq10.8}
(p^\alpha_{\tilde\alpha})^*\ell_k^{-1}\cdot (p^{\alpha}_\beta)^* f_{km}^1\cdot (p^\alpha_{\tilde\alpha})^*\ell_m\subset  (p^{\alpha}_\beta)^* f_{km}^2\cdot V_1\quad {\rm on}\quad \bigl((p^{\alpha}_{\tilde\beta})^{-1}(\bar W_k)\cap (p^{\alpha}_{\tilde \beta})^{-1}(\bar W_m)\bigr).
\end{equation}
In particular, there exist some $h_{km}\in C(\bigl((p^{\alpha}_{\tilde\beta})^{-1}(\bar W_k)\cap (p^{\alpha}_{\tilde \beta})^{-1}(\bar W_m)\bigr),V_1)$ such that
\[
(p^\alpha_{\tilde\alpha})^*\ell_k^{-1}\cdot (p^{\alpha}_\beta)^* f_{km}^1\cdot (p^\alpha_{\tilde\alpha})^*\ell_m=  (p^{\alpha}_\beta)^* f_{km}^2\cdot h_{km}.
\]
Without loss of generality we may assume that $K=\{1,\dots, \bar k\}\subset\N$. We set
\[
Z_l=\bigcup_{k=1}^l (p^{\alpha}_{\tilde\beta})^{-1}(\bar W_k).
\]
Using induction on $l\in \{1,\dots, \bar k\}$ we prove the following statement:\smallskip

\noindent ($\circ$) {\em There exist  maps $r_k^l\in C\bigl((p^{\alpha}_{\tilde\beta})^{-1}(\bar W_k),V_k\bigr)$, $1\le  l\le\bar k$, such that for all $1\le k,m\le l$}
\[
(r_k^l)^{-1}\cdot (p^{\alpha}_\beta)^* f_{km}^2\cdot h_{km}\cdot r_m^l= (p^\alpha_\beta)^* f_{km}^2 \quad {\rm on}\quad (p^\alpha_{\tilde\beta})^{-1}(\bar W_k)\cap (p^\alpha_{\tilde\beta})^{-1}(\bar W_m).
\]

If $l=1$, the statement  trivially holds with $r_1^1=1_G$.

Assuming that it is valid for $l\in\{1,\dots, \bar k-1\}$ let us prove it for $l+1$.

We set $r_k^l:=1_{G}$ for $l+1\le k\le \bar k$ and define a new  $G$-valued $1$-cocycle $c^l=\{c_{km}^l\}_{1\le k, l\le N+1}$ on $\widetilde{\mathfrak W}:=((p_{\tilde\beta}^\alpha)^{-1}(\bar W_k))_{k\in K}$ by the formulas
\[
c_{km}^l:=(r_k^l)^{-1}\cdot (p^{\alpha}_\beta)^* f_{km}^2\cdot h_{km}
\cdot r_m^l\quad {\rm on}\quad (p^\alpha_{\tilde\beta})^{-1}(\bar W_k)\cap (p^\alpha_{\tilde\beta})^{-1}(\bar W_m),\quad 1\le k,m\le \bar k.
\]
Then for all $1\le k,m\le l$ we have 
\[
c_{km}^l=(p^{\alpha}_\beta)^* f_{km}^2\quad {\rm on}\quad (p^\alpha_{\tilde\beta})^{-1}(\bar W_k)\cap (p^\alpha_{\tilde\beta})^{-1}(\bar W_m).
\]
Due to \eqref{eq10.6}, $c_{km}^l(x)\in f_{km}^2(x)\cdot V_{l+1}$ for all $x$ and $k, m\in K$.
In particular, restrictions of $c_{k\,l+1}^{l}$, $1\le k\le l$, to $(p^\alpha_{\tilde\beta})^{-1}(\bar W_k)\cap (p^\alpha_{\tilde\beta})^{-1}(\bar W_{l+1})$ have the form
\[
c_{k\,l+1}^k=(p^{\alpha}_\beta)^* f_{k\,l+1}^2\cdot t_{k\,l+1}\ \, {\rm for\ some}\ \, t_{k\,l+1}\in C\bigl((p^\alpha_{\tilde\beta})^{-1}(\bar W_k)\cap (p^\alpha_{\tilde\beta})^{-1}(\bar W_{l+1}), V_{l+1}\bigr).
\]
Note that maps $t_{k\,l+1}$ glue together over all nonempty intersections to give a map $t_{l+1}\in C(Z_l\cap (p^\alpha_{\tilde\beta})^{-1}(\bar W_{l+1}),V_{l+1})$. Indeed, if $x\in \bigl((p^\alpha_{\tilde\beta})^{-1}(\bar W_{k})\cap (p^\alpha_{\tilde\beta})^{-1}(\bar W_{m})\bigr)\cap (p^\alpha_{\tilde\beta})^{-1}(\bar W_{l+1})$, $1\le k,m\le l$, then since $c^l$ and $\{(p^{\alpha}_\beta)^*f_{km}^2\}_{k,m\in K}$ are cocycles on $\widetilde{\mathfrak W}$,
\[
\begin{array}{l}
\displaystyle
(p^{\alpha}_\beta)^*f_{k\,l+1}^2(x)\cdot ((p^{\alpha}_\beta)^* f_{k\,l+1}^2(x))^{-1}=
(p^{\alpha}_\beta)^*f_{km}^2(x)=c_{km}^l(x)=c_{k\,l+1}(x)\cdot (c_{m\,l+1}(x))^{-1}\medskip\\
\displaystyle =(p^{\alpha}_\beta)^* f_{k\,l+1}^2(x)\cdot t_{k\,l+1}(x)\cdot (t_{m\,l+1}(x))^{-1}\cdot ((p^{\alpha}_\beta)^* f_{k\,l+1}^2(x))^{-1}.
\end{array}
\]
This implies that $t_{k\,l+1}(x)= t_{m\,l+1}(x)$ as required.

Using  the Arens extension theorem we extend $t_{l+1}$ to a map  $\tilde t_{l+1}\in C((p^\alpha_{\tilde\beta})^{-1}(\bar W_{l+1}),V_{l+1})$ and define 
\[
r_{l+1}^{l+1}:=\tilde t_{l+1}^{-1}\quad {\rm and} \quad r_{k}^{l+1}:=r_k^l,\quad 1\le k\le l.
\]
Then for $1\le k\le l$,
\[
(r_k^{l+1})^{-1}\cdot (p^\alpha_{\tilde\beta})^*f_{k\,l+1}\cdot h_{k\,l+1}\cdot r_{l+1}^{l+1}=c_{k\,l+1}^l\cdot t_{k\,l+1}^{-1}= (p^\alpha_{\beta})^*f_{k\,l+1}
\]
on $(p^\alpha_{\tilde\beta})^{-1}(\bar W_k)\cap (p^\alpha_{\tilde\beta})^{-1}(\bar W_{l+1})$,
and for others $k,m\in \{1,\dots, l+1\}$ analogous identities on $(p^\alpha_{\tilde\beta})^{-1}(\bar W_k)\cap (p^\alpha_{\tilde\beta})^{-1}(\bar W_{m})$ are
valid by the induction hypothesis and  the definition of cocycle.

This completes the proof of the induction step.

Choosing in the statement $l=\bar k$ and defining $r_k:=r_k^{\bar k}$, $k\in K$, we obtain the assertion of the lemma.
\end{proof}

To finish the proof of the proposition observe that bundles $(p_{\beta}^\alpha)^*E_i$ can be defined by cocycles $\{(p^{\alpha}_\beta)^*f_{km}^i\}_{k,m\in K}$, $i=1,2$, on cover $\widetilde{\mathfrak W}$ , see \eqref{equ8.12}. Then Lemma \ref{lem10.2} shows that these cocycles determine the same element of cohomology set $H^1(X_\alpha,G)$. This implies that topological principal $G$-bundles $(p_{\beta}^\alpha)^*E_1$ and $(p_{\beta}^\alpha)^*E_2$ on $X_\alpha$ are isomorphic.

The proof of the proposition is complete.
\end{proof}

To complete the proof of Theorem \ref{teo1.10} we apply Proposition \ref{prop10.1} to
$X=M(\mathscr A_Z)$ and  $X_\alpha$ the closed $\frac{1}{\#\alpha}$-neighbourhood of $M(A_\alpha)$ in $\Co^{\#\alpha}$, i.e., $X_\alpha:=\cup_{z\in M(A_\alpha)}\bar{B}_z(\frac{1}{\#\alpha})\subset\Co^{\#\alpha}$, where $\bar{B}_z(\frac{1}{\#\alpha})\Subset\Co^{\#\alpha}$ is the closed Euclidean ball of radius $\frac{1}{\#\alpha}$ centered at $z$. Then the inverse limit of the the inverse limit system of compacta $\{X_\alpha,F_\beta^\alpha\}$, where $\alpha$ ranges over the directed set $D$, coincides with $M(\mathscr A_Z)$ as well. Suppose bundles $P_1,P_2$ of the second statement of Theorem \ref{teo1.10} are
defined on an open neighbourhood $N\Subset\Co^{\#\beta}$ of $M(A_\beta)$. By the definition of the inverse limit topology there exists some $\tilde\beta\supseteq\beta$ such that $X_{\tilde \beta}\Subset (F_\beta^{\tilde\beta})^{-1}(N)$. In particular, bundles $(F_\beta^{\tilde\beta})^*P_i$, $i=1,2$, are defined on an open neighbourhood of $X_{\tilde\beta}\Subset\Co^{\#\tilde\beta}$. Since pullbacks of these bundles to $M(\mathscr A_Z)$ are  isomorphic topological principal $G$-bundles, Proposition \ref{prop10.1} is applicable.
Due to the proposition, there is $\alpha\supseteq\tilde\beta$ such that bundles
$(F_\beta^\alpha)^*E_i$, $i=1,2$, on $X_\alpha$ are isomorphic. As $\mathring{X}_\alpha$ is an open neighbourhood of the polynomially convex set
$M(A_\alpha)$, there exists a Stein neighbourhood $O\Subset X_\alpha$ of $M(A_\alpha)$. Then holomorphic bundles $(F_\beta^\alpha)^*E_i$ restricted to $O$ are topologically isomorphic and so by Bungart's theorem \cite[Th.\,8.1]{Bu} they are holomorphically isomorphic as well. This completes the proof of the second statement of Theorem \ref{teo1.10} and therefore of the theorem.

\sect{Proofs of Theorems  \ref{teo3.4}--\ref{teo3.7} and Corollary \ref{cor3.6}}
 
\subsection{Proof of Theorem \ref{teo3.4}}
Since $\mathcal O(K,X)$ is dense in the complex Banach Lie group $\mathcal A(K,X)$, it suffices to show that if $f\in\mathcal O(K,X)$, then $f|_K$ can be joined by a path in $\mathcal O(K,X)$ with the constant map of the value $1_G$ (the unit of group $\mathcal A(K,X)$).

To this end, let  $U\subset\mathfrak g$ be an open ball centered at $0$ such that $\exp_G|_U: U\rightarrow G$ is biholomorphic.
Suppose $f$ is holomorphic on an open neighbourhood  $W$ of $K$. We choose a finite cover $\mathfrak V=(V_l)_{l\in L}$ of $K$ by relatively compact open subsets of $W$ so that 
\[
f|_{V_l}=g_lf_l\quad {\rm for\ some}\quad g_l\in G,\  \, f_l\in \mathcal O(V_l,\exp_G(U)),\ l\in L.
\]
By definitions of $g_l$ and $f_l$ there are paths $u_l: [0,1]\rightarrow G$
and $v_l: [0,1]\rightarrow \mathcal O(V_l,\exp_G(U))$
 such that $u_l(0)=v_l(0)(x)= 1_G$ for all $x\in V_l$, and $u_l(1)=g_l$, $v_l(1)=f_l$. We set
 \[
 z_l(t)=u_l(t)\, v_l(t),\quad t\in [0,1].
\]
Then $z_l$ can be regarded as a map in $\mathcal O(V_l, C([0,1],G))$, $l\in L$. Let us consider cocycle $\{c_{lm}\}_{l,m\in L}$ on $\mathfrak V$ defined by the formulas
\begin{equation}\label{equ9.2}
c_{lm}:=z_l^{-1}\, z_m\quad {\rm on}\quad V_l\cap V_m.
\end{equation}
Each $c_{lm}$ has range in the complex Lie group $C_0([0,1], G)$ of maps in $C([0,1], G)$ equal to $1_G$ at the endpoints. By definition, each element of $C_0([0,1], G)$ is a loop in $G$ with the basepoint $1_G$. Since $G$ is simply connected, such a loop can be joined in $C_0([0,1], G)$ with the constant loop of the value $1_G$. This shows that group $C_0([0,1], G)$ is connected. 

Next, cocycle $\{c_{lm}\}_{l,m\in L}$ determines a holomorphic principal $C_0([0,1], G)$-bundle $P$ on $\cup_{l\in L}V_l$. Since the fibre of $P$ is connected, Corollary \ref{cor1.5}(1) implies that $P|_{W'}$ is trivial on an open neighbourhood $W'\Subset W$ of $K$. Hence, there exist $c_l\in\mathcal O(W'\cap V_l,C_0([0,1],G))$ (we assume that all $K\cap V_l\ne\emptyset$ so that all $W'\cap V_l\ne\emptyset$ as well) such that
\begin{equation}\label{equ9.3}
c_l^{-1}\, c_m=c_{lm}\quad {\rm on}\quad (W'\cap V_l)\cap (W'\cap V_m),\quad l,m\in L.
\end{equation}
Comparing \eqref{equ9.2} and \eqref{equ9.3} we get a map $H\in \mathcal O(W',C([0,1],G))$ such that
\[
H|_{W'\cap V_l}:=z_l\, c_l^{-1},\quad l\in L.
\]
Clearly, $H(x,0)=1_G$ and $H(x,1)=f(x)$ for all $x\in W'$.  Thus $H|_K$ determines a path in $\mathcal O(K,G)$ joining $f$ with the constant map of the value $1_G$. This shows that $\mathcal A(K,G)$ is connected.

The proof of the theorem is complete.
\subsection{Proof of Theorem \ref{teo3.5}} (1) Let $X$ be a complex Banach homogeneous space under the action of a complex Banach Lie group $G$. We use the following result see, e.g., \cite[Prop.\,1.4]{R}:
\begin{itemize}
\item[($\star$)]
For each $p\in X$, $g\in G$ there exist neighbourhoods $U$ of $g$ in $G$ and $V$ of $g\cdot p$ in $X$ and a holomorphic map $f_{g,p}:V\rightarrow U$ such that $\pi^p\circ f_{g,p}$ is the identity on $V$.
\end{itemize}
Let $G_0$ be the connected component of the unit of $G$ and $p$ be a point in a connected component $X'$ of $X$. Then ($\star$) implies that  $O_p:=\{g\cdot p\}_{g\in G_0}$ is an open connected subset of $X'$. Clearly, for distinct $p_1,p_2\in X'$, sets $O_{p_k}$, $k=1, 2$, either coincide of disjoint. Therefore if $X'\ne O_p$, it can presented as disjoint union of open connected subsets which contradicts connectedness of $X'$. Hence, each connected component $X'$ of $X$ is a complex Banach homogeneous space under the action of $G_0$.

Further, let $\psi: G_u\rightarrow G_0$ be the universal covering of $G_0$. Since $\psi$ is a locally biholomorphic epimorphism of groups,  each connected component $X'$ of $X$ can be considered as a complex Banach homogeneous space under the action $G_u\times X'\rightarrow X'$,
$(g,p)\mapsto\psi(g)\cdot p$.
Finally, since space $M(\mathscr A_Z)$ is connected, images of homotopic maps $M(\mathscr A_Z)\rightarrow X$ belong to the same connected component. These arguments show that it suffices to prove the following statement:\smallskip

\noindent ($\circ$) {\em Suppose $X$ is a connected complex Banach homogeneous space under the action of a simply connected complex Banach Lie group $G$. If $f_1,f_2\in \mathcal O(M(\mathscr A_Z),X)$ are homotopic, then they are homotopic in $\mathcal O(M(\mathscr A_Z),X)$.}\smallskip

To this end, we fix $p\in X$ and consider an open cover $\mathfrak U=(U_i)_{i\in I}$ such that there exist maps $s_i\in \mathcal O(U_i,G)$ with $s_i\cdot p={\rm id}_{U_i}$, $i\in I$, cf. ($\star$). Then $s_i(x)\cdot p=s_j(x)\cdot p=x$ for all $x\in U_i\cap U_j$. This implies that $c_{ij}(x):=s_{i}^{-1}(x) s_j(x)$ satisfy $c_{ij}(x)\cdot p=p$ for all $x\in U_i\cap U_j$,
i.e., $c_{ij}\in\mathcal O(U_i\cap U_j, G(p))$. Thus, cocycle $c=\{c_{ij}\}_{i,j\in I}$ on $\mathfrak U$ determines a holomorphic principal $G(p)$-bundle $P$ on $X$. Since $f_1$ and $f_2$ are homotopic, pullbacks $f_1^*P$ and $f_2^*P$ are topologically isomorphic holomorphic principal $G(p)$-bundles on $M(\mathscr A_Z)$, see, e.g., \cite[Ch.\,4]{Hus}. Thus, by Theorem \ref{teo2.9} they are holomorphically isomorphic. Hence, there exists a common finite open refinement $\mathfrak V=(V_l)_{l\in L}$ of covers $(f_k^{-1}(U_i))_{i\in I}$ with refinements maps $\tau_k: L\rightarrow I$, $k=1,2$, and maps $c_l\in \mathcal O(V_l\cap V_m,G(p))$, $l\in L$, such that
\[
c_l^{-1} (f_1^* c_{\tau_1(l)\tau_1(m)})\, c_m=f_2^* c_{\tau_2(l)\tau_2(m)}\quad {\rm on}\quad V_l\cap V_m,\ \, l,m\in L.
\]
From here, using the definition of cocycle $c$,  we get
\begin{equation}\label{equ9.4}
c_l^{-1}\,(f_1^* s_{\tau_1(l)})^{-1} (f_1^*s_{\tau_1(m)})\, c_m=(f_2^* s_{\tau_2(l)})^{-1} (f_2^*s_{\tau_2(m)})\quad {\rm on}\quad V_l\cap V_m,\ \, l,m\in L.
\end{equation}
This determines a map $g\in\mathcal O(M(\mathscr A_Z), G)$ such that
\[
g|_{V_l}=(f_2^* s_{\tau_2(l)})\,c_l^{-1}\,(f_1^* s_{\tau_1(l)})^{-1}, \ \, l\in L.
\]
Then from \eqref{equ9.4} we obtain
\[
f_2=f_2^*s_{\tau_2(l)}\cdot p=(g\,f_1^*s_{\tau_1(l)}\,c_l)\cdot p=
(g\,f_1^*s_{\tau_1(l)})\cdot p=g\cdot f_1\quad {\rm on}\quad V_l,\ \, l\in L.
\]
Since $G$ is simply connected $g$ can be joined by a path $g_t$, $t\in [0,1]$, with the constant map $g_0$ of value $1_G$, see Theorem \ref{teo3.4}. Then $g_t\cdot f_1$ determines homotopy in $\mathcal O(M(\mathscr A_Z), X)$ joining $f_1$ and $f_2$.

This proves injectivity of map $\mathscr O: [M(\mathscr A_Z),X]_{\mathcal O}\rightarrow [M(\mathscr A_Z),X]$ of the theorem.\smallskip

(2) Suppose that $f\in C(M(\mathscr A_Z))$. For $A\in\mathscr A_Z$,  consider $M(\mathscr A_Z)$ as the inverse limit of the inverse limiting system $\{M(A_\alpha),F_\beta^\alpha\}$, where $\alpha$ runs over directed set $D$, see Section~2.2. Then by the Stone-Weierstrass theorem and since $X$ is an absolute neighbourhood retract (see \cite{P}), there are $\alpha\in D$ and a map $g\in C(U,X)$ defined on a neighbourhood of $M(A_\alpha)$ such that maps $f$ and $F_\alpha^*g$ are homotopic. Further, as $M(A_\alpha)\Subset\Co^{\#\alpha}$ is polynomially convex, there exists an open Stein neighbourhood $N$ of $M(A_\alpha)$ containing in $U$. Then due to Ramspott's ``Oka principle'' 
\cite{Ra}, \cite[Th.\,2.1]{R} there is a map $\tilde g\in \mathcal O(N,X)$ homotopic to $g|_{N}$. Hence, map $f$ is homotopic to map $F_\alpha^{*\,}\tilde g\in\mathcal O(M(\mathscr A_Z),X)$.

This proves surjectivity of map $\mathscr O: [M(\mathscr A_Z),X]_{\mathcal O}\rightarrow [M(\mathscr A_Z),X]$ and completes the proof of the theorem.

\subsection{Proof of Corollary \ref{cor3.6}}
Since ${\rm dim}\, M(\mathscr A_Z)=2$ and $H^{2}(M(\mathscr A_Z),\pi_2(X))=0$, see Lemma \ref{lem6.3}, under the hypotheses of the corollary there is a one-to-one correspondence between elements of $[M(\mathscr A_Z),X]$ and the \v{C}ech cohomology group $H^1(M(\mathscr A_Z),\pi_1(X))$, see \cite[Th.\,(11.4)]{Hu2}. Since by Theorem \ref{teo3.5} $\mathscr O: [M(\mathscr A_Z),X]_{\mathcal O}\rightarrow [M(\mathscr A_Z),X]$ is a bijection, this gives the required statement.
\begin{R}\label{rem10.1}
{\rm The correspondence of the corollary is defined as follows.

Let $r: X_u\rightarrow  X$ be the universal covering of $X$.  It can be considered as a principal bundle $P$ on $X$ with fibre $\pi_1(X)$ defined by a cocycle $c$ with values in $\pi_1(X)$ on an open cover of $X$. By definition, $c$ determines an element $\{c\}\in H^1(X,\pi_1(X))$. Now, if $f\in C(M(\mathscr A_Z), X)$ then pullback $f^*P$ is a principal $\pi_1(X)$-bundle on $M(\mathscr A_Z)$ defined by cocycle $f^*c$ representing class $\{f^*c\}\in H^1(M(\mathscr A_Z),\pi_1(X))$. If $f,g\in C(M(\mathscr A_Z), X)$ are homotopic, then $\{f^*c\}=\{g^*c\}$. Map $\mathcal O(M(\mathscr A_Z),X)\ni f\mapsto \{f^*c\}\in H^1(M(\mathscr A_Z),\pi_1(X))$ determines the required correspondence.
}
\end{R}

\subsection{Proof of Theorem \ref{teo3.6}}
(1) Let $f\in \mathcal O(M(\mathscr A_Z),X)$. Due to the arguments of the proof of Theorem \ref{teo3.5}(2), there exist $\alpha\in D$ and a holomorphic map $h\in\mathcal O(N,X)$ defined in a neighbourhood $N$ of $M(A_\alpha)$ such that holomorphic maps $f$ and $F_\alpha^*h$ are homotopic. Then by Theorem \ref{teo3.5} these maps are homotopic in $\mathcal O(M(\mathscr A_Z),X)$. This shows that map
$\mathfrak F_A: \displaystyle \lim_{\longrightarrow}\,[M(A_\alpha), X]_{\mathcal O}\rightarrow [M(\mathscr A_Z),X]_{\mathcal O}$ is surjective.\smallskip

(2) Suppose holomorphic maps $f_1,f_2$ into $X$ defined on a neighbourhood of $M(A_\beta)$ are such that  $F_\beta^*f_1$ and $F_\beta^*f_2$ are homotopic in $\mathcal O(M(\mathscr A_Z),X)$. Then by the definition of the inverse limit topology and since $X$ is an absolute neighbourhood retract, there exists $\alpha\supseteq\beta$ such that maps $(F_\beta^\alpha)^*f_1$ and $(F_\beta^\alpha)^*f_2$ are defined and homotopic on a neighbourhood $U$ of $M(A_\alpha)$, see, e.g., \cite[Lm.\,1]{Li} and the reference therein. As $M(A_\alpha)$ is polynomially convex, there is a Stein neighbourhood $N$ of $M(A_\alpha)$ containing in $U$. Then maps $(F_\beta^\alpha)^*f_1|_N$ and $(F_\beta^\alpha)^*f_2|_N$ are homotopic and so the Ramspott ``Oka principle'' \cite{Ra}, \cite[Th.\,2.1]{R} implies that they are homotopic in $\mathcal O(N,G)$. 

This shows that map
$\mathfrak F_A: \displaystyle \lim_{\longrightarrow}\,[M(A_\alpha), X]_{\mathcal O}\rightarrow [M(\mathscr A_Z),X]_{\mathcal O}$ is injective and completes the proof of the theorem.

\subsection{Proof of Theorem \ref{teo3.7} }
Let $f\in\mathcal O(M(\mathscr A_Z), X)$. Using ($\star$) as in the proof of Theorem \ref{teo3.5} we construct a holomorphic principal $G(p)$-bundle $P$ on the connected component of $X$ containing  the image of $f$. Since the fibre $G(p)$ of holomorphic principal bundle $f^*P$ on $M(\mathscr A_Z)$ is  connected, $f^*P$ is trivial by Corollary \ref{cor1.5}(1). Thus it has a holomorphic section. Repeating literally the arguments of the proof of Theorem \ref{teo3.5} we construct by means of this section a map $\tilde f\in\mathcal O(M(\mathscr A_Z),X)$ such that $f(x)=\tilde f(x)\cdot p$ for all $x\in M(\mathscr A_Z)$. We leave the details to the reader.

The proof of the theorem is complete.

\sect{Proofs of Theorems \ref{nonlinrunge} and \ref{te2.2}}
\begin{proof}[Proof of Theorem \ref{nonlinrunge}]
(1) Suppose $f\in\mathcal O(U,X)$ can be uniformly approximated on the compact subset $K\subset U$ by maps in $C(M(\mathscr A_Z),X)$.
Since $X$ is an absolute neighbourhood retract, see \cite{P}, the latter implies that  $f$ is homotopic to the restriction to $U$ of a map in $C(M(\mathscr A_Z),X)$. Then according to Theorem 
\ref{teo3.5} it is homotopic to the restriction to $U$ of a map $g\in\mathcal O(M(\mathscr A_Z),X)$. In particular, since $M(\mathscr A_Z)$ is connected, images of $g$ and $f$ belong to a connected component $X_0$ of $X$.  As in the proof of Theorem \ref{teo3.5},  $X_0$ is a complex Banach homogeneous space under the action of a simply connected complex Banach Lie group $G$. Let $H\subset G$ be the stationary subgroup of a point in $X_0$ under this action. Then $G$ is biholomorphic to the holomorphic principal bundle $P$ on $X_0$ with fibre $H$ (cf. ($\star$) in the proof of Theorem \ref{teo3.5}).  Since $f$ and $g|_U$ are homotopic, the holomorphic principal bundles $f^*P$ and $g^*P|_U$ on $U$ are topologically isomorphic. Then by Theorem \ref{teo8.1} they are holomorphically isomorphic on an open neighbourhood $V\Subset U$ of $K$. From here as in the proof of Theorem \ref{teo3.5}(1) we obtain that there is a map $h\in\mathcal O(V,G)$ such that $f=h\cdot g|_V$. Since $G$ is simply connected, for an open neighbourhood $W\Subset V$ of $K$ map $h|_W$ can be joined by a path in $\mathcal O(W,G)$ with the constant map of the value $1_G$, see the proof of Theorem \ref{teo3.4}.
Since $K$ is holomorphically convex,  the latter and the Runge approximation theorem (see Theorem \ref{runge}) imply, by a standard argument, that $h|_K$ can be uniformly approximated on $K$ by  maps $h_n\in\mathcal O(M(\mathscr A_Z),G)$, $n\in\N$. In particular, maps $h_n\cdot g$ converges to $f$ uniformly on $K$.

This completes the proof of part (1) of the theorem.\smallskip

\noindent (2) Suppose $X$ is simply connected. Then it is the complex Banach homogeneous space under the action of a simply connected complex Banach Lie group $G$ and the stationary subgroup $H\subset G$ of a point $p\in X$ under this action is  {\em connected}. (This follows, e.g., from the long exact sequence of homotopy groups of fibration $G\rightarrow X$ with fibre $H$.) In particular, in the notation of the first part of the proof, holomorphic bundle $f^*P$ on $U$, $f\in\mathcal O(U,X)$, has connected fibre $H$. Thus, by Corollary \ref{cor1.5}\,(1), $f^*P|_V$ is 
trivial on an open neighbourhood $V\Subset U$ of $K$. From here as in the proof of Theorem \ref{teo3.5} (cf. also Theorem \ref{teo3.7}) we obtain that $f|_V=\tilde f\cdot p$ for some $\tilde f\in\mathcal O(V,G)$. Since $G$ is simply connected, as in the proof of the first part of the theorem we get that $\tilde f|_K$ can be uniformly approximated on $K$ by maps $\tilde f_n\in\mathcal O(M(\mathscr A_Z),G)$, $n\in\N$. In particular, maps $\tilde f_n\cdot p$ converges to $f$ uniformly on $K$.

The proof of the theorem is complete.
\end{proof}

\begin{proof}[Proof of Theorem \ref{te2.2}]
First, we prove the particular case of part (2) of the theorem for $X$ being a connected complex Banach Lie group $G$.

Suppose that $f\in\mathcal O(U,G)$ for an open neighbourhood $U$ of a hull $Z$. Let $Q_Z: M(H^\infty)\rightarrow M(\mathscr A_Z)$ be the quotient map of Proposition \ref{prop1.1} sending $Z$ to a point $z$.  Then $Q_Z(U)\subset M(\mathscr A_Z)$ is an open neighbourhood of $z$. Since $Q_Z$ is one-to-one outside $Z$, there is a map $h\in \mathcal O(M(\mathscr A_Z)\setminus\{z\},G)$ such that $Q_Z^*h=f$ on $M(H^\infty)\setminus Z$. Consider the open cover $\mathfrak U:=\bigl\{M(\mathscr A_Z)\setminus\{z\}, Q_Z(U)\bigr\}$ of $M(\mathscr A_Z)$. Let $P$ be the holomorphic principal $G$-bundle on $M(\mathscr A_Z)$ defined on  $\mathfrak U$ by cocycle $c:=h|_{Q_Z(U)\setminus\{z\}}\in \mathcal O (Q_Z(U)\setminus\{z\},G)$. Since group $G$ is connected, Corollary \ref{cor1.5}(1) implies that bundle $P$ is trivial. Equivalently, there exists a global holomorphic section of $P$ which in local coordinates on $\mathfrak U$ is given by maps $c_1\in \mathcal O(M(\mathscr A_Z)\setminus\{z\},G)$ and $c_2\in\mathcal O(Q_Z(U),G)$ such that
\begin{equation}\label{equ9.1}
c_1 c_2^{-1}=c\quad {\rm on}\quad Q_Z(U)\setminus\{z\}.
\end{equation}
Replacing $c_i$ by $c_i c_2(z)^{-1}$, $i=1,2$, if necessary, without loss of generality we may assume that
$c_2(z)=1_G$. Then \eqref{equ9.1} implies that  maps $Q_Z^*c_1\in\mathcal O(M(H^\infty)\setminus Z,G)$ and $g\, Q_Z^*c_2\in\mathcal O(U,G)$ are equal on $U\setminus Z$ and therefore they glue together to determine a map $\tilde f\in\mathcal O(M(H^\infty),G)$. By definition, $\tilde f|_{Z}=f$, as required.

Next, we establish the following result.
\begin{Lm}\label{lem11.1}
Under hypotheses (1), (2) of the theorem there is an open neighbourhood $V\Subset U$ of $Z$ and a map $g\in\mathcal O(M(H^\infty),X)$ such that maps $f|_V$ and $g|_V$ are homotopic.
\end{Lm}
\begin{proof}
For part (1), there is a map in $C(M(H^\infty),X)$ extending $f|_Z$. Such a map is homotopic to a map $g\in\mathcal O(M(H^\infty), X)$ by Theorem \ref{teo3.5}. This gives the required result with $V:=U$.

Now suppose that  $f$ satisfies conditions of part (2) of the theorem.  Since due to Proposition \ref{prop1.1} ${\rm dim}\, Z\le 2$ and  $H^2(Z,\Z)=0$,
there is a one-to-one correspondence between homotopy classes of continuous maps $Z\rightarrow X$ and elements of $H^1(Z,\pi_1(X))$, see \cite[Th.\,(11.4)]{Hu2}, described similarly to that of Remark \ref{rem10.1}.
On the other hand, by Treil's theorem \cite{T6} and its corollary \cite[Th\,1.3]{S1}  inclusion $Z\hookrightarrow M(H^\infty)$ induces an epimorphism of groups $H^1(M(H^\infty),\Z)\rightarrow H^1(Z,\Z)$.  Thus Corollary \ref{cor3.6} and Remark \ref{rem10.1} imply that there is a map $g\in\mathcal O(M(H^\infty),X)$ such that maps $f|_Z$ and $g|_Z$ are homotopic. Further, since $Z$
is the inverse limit of the inverse system of compacta $\{\bar V,\subseteq\,\}$, where $V$ runs over the directed set (with the order given by inclusions of sets) of open neighbourhoods of $Z$, the previous statement implies that $f|_V$ and $g|_V$ are homotopic on an open neighbourhood $V\Subset U$ of $Z$, see, e.g., \cite[Lm.\,1]{Li} and the reference there.
\end{proof}

Using this lemma, as in the proof of Theorem \ref{nonlinrunge} we obtain that for some open neighbourhood $W\Subset V$ of $Z$ there exists $h\in\mathcal O(W,G)$ such that $f|_W=h\cdot g|_W$. Here $G$ is a simply connected complex Banach Lie group such that $X$ is the complex Banach homogeneous space with respect to the action of $G$. Further, according to the particular case of part (2) of the theorem established above, there is a map $\tilde h\in\mathcal O(M(H^\infty),G)$ such that $\tilde h|_Z=h|_Z$. We set $\tilde f:=\tilde h\cdot g$. Then $\tilde f\in \mathcal O(M(H^\infty),X)$ and $\tilde f|_Z= h|_Z\cdot g|_Z=f_Z$. 

The proof of the theorem is complete.
\end{proof}

%======
\sect{Proofs of Theorems \ref{bmoa1}, \ref{teo4.1}, \ref{teo4.5}, \ref{teo4.6} and Proposition \ref{prop4.4}}
\begin{proof}[Proof of Theorem \ref{bmoa1}]
Let $\mathfrak U=(U_i)_{i\in I}$ be a finite open cover of $\Co\mathbb T^n$ by simply connected sets. Then  
there are maps $\varphi_i\in\mathcal O(U_i,\Co^n)$ such that $\pi\circ\varphi_i={\rm id}_{U_i}$, $i\in I$. In particular, $c_{ij}:=\varphi_{i}-\varphi_j\in C(U_i\cap U_j,\Gamma)$ so that $c=\{c_{ij}\}_{i,j\in I}$ is an additive $1$-cocycle on $\mathfrak U$ with values in $\Gamma$. 

Suppose $f\in\mathcal O(\Di,\Co\mathbb T^n)$ is the restriction of a continuous map $\hat f\in C(M(H^\infty),\Co\mathbb T^n)$.  We set $\mathfrak V=(V_i)_{i\in I}$, $V_i:=f^{-1}(U_i)$, $i\in I$, and consider additive $\Gamma$-valued $1$-cocycle  $\hat f^*c=\{\hat f^*c_{ij}\}_{i,j\in I}$ on cover $\mathfrak V$ of  $M(H^\infty)$. Since $\Gamma$ is isomorphic to $\Z^{2n}$, there are linearly independent over $\RR$ vectors $v_1,\dots, v_{2n}\in\Co^n$ such that $\Gamma=\Z v_1 \oplus \Z v_2 \oplus\cdots\oplus \Z v_{2n} $. Hence, each $\hat f^*c_{ij}=\sum_{k=1}^{2n}c_{ij}^k v_k$ for some $c_{ij}^k\in C(V_i\cap V_j,\Z)$ and $c^k:=\{c_{ij}^k\}_{i,j\in I}$ are integer-valued $1$-cocycles on $\mathfrak V$, $1\le k\le 2n$. Cocycles $c^k$ determine holomorphic principal $\Co$-bundles on $M(H^\infty)$ which are trivial by Corollary \ref{cor1.5}\,(1). Thus there exist functions $g_i^k\in \mathcal O(V_i)$ such that  
\begin{equation}\label{eq11.1}
g_i^k-g_j^k=c_{ij}^k\qquad {\rm on}\qquad  V_i\cap V_j,\quad i,j\in I,\ \, 1\le k\le 2n.
\end{equation}

On the other hand, since $\Di$ is contractible, restrictions $c^k|_{\Di}$ are trivial cocycles, i.e., there exist functions $n_i^k\in C(V_i\cap \Di,\Z)$ such that
\begin{equation}\label{eq11.2}
 n_i^k-n_j^k=c_{ij}^k\qquad{\rm on}\qquad (V_i\cap\Di)\cap ( V_j\cap\Di), \quad i, j\in I,\ \, 1\le k\le 2n.
\end{equation} 

Equations \eqref{eq11.1} and \eqref{eq11.2} show that functions 
$g_{i}^k-n_i^k\in \mathcal O(V_i\cap\Di)$, $i\in I$, glue together over all nonempty 
intersections $(V_i\cap\Di)\cap ( V_j\cap\Di)$ to determine functions $h^k\in \mathcal O(\Di)$, $k\in\{1,\dots,2n\}$. By definition, each $e^{2\pi\sqrt{-1}\, h^k}\in\mathcal O(\Di,\Co^*)$ extends to a function in $\mathcal O(M(H^\infty),\Co^*)$. In particular, the imaginary part of each $h^k$ is bounded and so each $h_k\in {\rm BMOA}$, see, e.g., \cite[Ch.\,VI,\,Th.\,1,5]{Ga}.  

Next, let us consider maps $\hat f^*\varphi_i-\sum_{k=1}^{2n}g_i^k v_k\in\mathcal O(V_i,\Co^n)$, $i\in I$. By the definition of cocycle $c$ and due to equation \eqref{eq11.1}, these maps glue together over all nonempty intersections $V_i\cap V_j$ to determine a map $g\in\mathcal O(M(H^\infty),\Co^n)$. These imply that map $\tilde f\in \mathcal O(\Di,\Co^n)$, 
\begin{equation}\label{eq11.3}
\tilde f(x)=(\hat f^*\varphi_i)(x)-\sum_{k=1}^{2n}n_i^k(x)v_k,\qquad x\in V_i\cap\Di,\ \, i\in I,
\end{equation}
is well-defined and $\tilde f=g+\sum_{k=1}^{2n}h^k v_k$ on $\Di$. Therefore, since all $h^k\in {\rm BMOA}$, all coordinates $\tilde f_k$ of  map $\tilde f=(\tilde f_1,\dots,\tilde f_n)$ belong to ${\rm BMOA}$. Also, according to \eqref{eq11.3},
\[
f=\pi\circ\tilde f.
\]
This completes the proof of the first part of the theorem.\smallskip

Conversely, suppose that $f=\pi\circ\tilde f$ for some $\tilde f=(\tilde f_1,\dots,\tilde f_n)\in\mathcal O(\Di,\Co^n)$ with all $\tilde f_k\in {\rm BMOA}$. We must show that $f$ extends to a map in $C(M(H^\infty),\Co\mathbb T^n)$. 

In fact, as follows from \cite[Th.1.11]{Br4}
there exist an open finite cover $\mathfrak U=(U_i)_{i\in I}$ of $M(H^\infty)$ and locally constant continuous maps $c_i\in C(U_i\cap\Di,\mathbb C^n)$ such that 
\begin{equation}\label{eq11.4}
\sup_{x\in U_i\cap\Di,\, i\in I}\|\tilde f(x)-c_i(x)\|_{\Co^n}<\infty .
\end{equation}
Let $\Pi\Subset\Co^n$ be the symmetric convex hull of vectors $v_1,\dots, v_{2n}$ generating $\Gamma$. Then $\Pi$ is the fundamental compact under the action of $\Gamma$ on $\Co^n$, i.e.,
\[
\Co^n=\bigcup_{z\in\Gamma}(z+\Pi) .
\]
Let $V$ be a connected component of  $U_i\cap\Di$. By definition, $c_i|_V$ is a constant vector. Thus there is a vector $d_{i,V}\in\Gamma$ such that $c_i|_V-d_{i,V}\in\Pi$. In this way we construct  locally constant maps $d_i\in C(U_i\cap\Di,\Gamma)$,  $d_i|_V:=d_{i,V}$, $V\subset U_i\cap\Di$ is clopen, such maps $c_i-d_i$ have ranges in $\Pi$, $i\in I$. From here and \eqref{eq11.4} we get
\begin{equation}\label{eq11.5}
\sup_{x\in U_i\cap\Di,\, i\in I}\|\tilde f(x)-d_i(x)\|_{\Co^n}<\infty .
\end{equation}

Next, since each map $\tilde f|_{U_i\cap\Di}-d_i$ is bounded holomorphic, according to Su\'{a}rez theorem \cite[Th.3.2]{S1}, it admits an extension $g_i\in\mathcal O(U_i,\Co^n)$. Then continuous maps $\pi\circ g_i\in \mathcal O(U_i,\Co\mathbb T^n)$ satisfy $\pi\circ g_i=f$ on $U_i\cap\Di$,  $i\in I$. In particular,
\begin{equation}\label{eq11.6}
\pi\circ g_i-\pi\circ g_j=0\qquad {\rm on}\qquad (U_i\cap U_j)\cap\Di,\ \, i,j\in I.
\end{equation}
Since, due to the Carleson corona theorem, each $(U_i\cap U_j)\cap\Di$ is dense in $U_i\cap U_j$, equation \eqref{eq11.6} implies that  maps $\pi\circ g_i$ glue together over all nonempty intersections $U_i\cap U_j$ to determine a map $\hat f\in \mathcal (M(H^\infty),\Co\mathbb T^n)$ such that
$\hat f|_{\Di}=f$. 

This completes the proof of the theorem.
\end{proof}
\begin{proof}[Proof of Theorem \ref{teo4.1}]
Parts (a) and (b) follow straightforwardly from  Theorems \ref{teo3.5} and \ref{teo3.7}. For part (c) we apply Theorem \ref{teo3.6} to find for each $f\in\mathcal O(M(\mathscr A_Z),{\rm id}\,\mathfrak A)$ 
some $\alpha\in D$, a map $\tilde h\in \mathcal O(U, {\rm id}\,\mathfrak A)$ defined on a neighbourhood $U$ of $M(A_\alpha)$ and a map $g\in \mathcal O(M(\mathscr A_Z),\mathfrak A_0^{-1})_0$ such that $gfg^{-1}=F_\alpha^*\tilde h$. Using that $M(A_\alpha)$ is polynomially convex, we choose a {\em Weil polynomial polyhedron} $\Pi$ such that $M(A_\alpha)\Subset \Pi\Subset U$. Then the application of the Weil integral representation formula \cite{W} to $\tilde h|_{\Pi}$ implies that this map is  the uniform limit of a sequence of maps of the algebraic tensor product $\Co[z_1,\dots, z_{\#\alpha}]\otimes\mathfrak A$ restricted to $\Pi$. (Here $\Co[z_1,\dots, z_{\#\alpha}]$ is the ring of holomorphic polynomials on $\Co^{\#\alpha}$.) Hence, map $h:=F_\alpha^*\tilde h$ 
is uniformly approximated by a sequence of maps in $\widehat{A}_\alpha\otimes\mathfrak A$ as required.\end{proof}
\begin{proof}[Proof of Proposition \ref{prop4.4}]
For a fixed $p\in\mathfrak A_l^{-1}$ consider  holomorphic  map $\pi^p: \mathfrak A_0^{-1}\rightarrow \mathfrak A_l^{-1}$, $\pi^p(g):=gp$. Its differential
$d\pi_{1_{\mathfrak A}}^p\in L(\mathfrak A)$  at $1_{\mathfrak A}$ is given by the formula  $d\pi_{1_{\mathfrak A}}^p(a):=ap$.\\
Let us show that  $d\pi_{1_{\mathfrak A}}^p$ is surjective and its kernel is a complemented subspace of $\mathfrak A$. 

Indeed, if $q\in\mathfrak A$ is such that $q p=1_{\mathfrak A}$, then linear map $r:\mathfrak A\rightarrow\mathfrak A$, $r(a):=aq$,
is a continuous right inverse of $d\pi_{1_{\mathfrak A}}^p$ which gives the required statement. 

This and the implicit function theorem (see, e.g., \cite[Prop.\,1.2]{R}) imply that there exist open neighbourhoods $U_p$ of $0\in\mathfrak A$ and $V_p$ of $p$ such that $\pi^p(U_p)=V_p$. From here one deduces easily that  $\mathfrak A_0^{-1}$ acts transitively on each connected component of $\mathfrak A_l^{-1}$.

In fact, suppose  $\gamma : [0,1]\rightarrow \mathfrak A_l^{-1}$ is a path joining left-invertible elements $p$ and $q$. Then for each $t\in [0,1]$ there is $\varepsilon_t>0$ such that $\gamma([0,1]\cap (t-\varepsilon_t, t+\varepsilon_t))\subset V_{\gamma(t)}$. This and compactness of $[0,1]$ show that there exist points $0=t_1<t_2<\cdots <t_{k+1}=1$ such that $\gamma([t_i,t_{i+1}])\subset V_{\gamma(t_i)}$ for all $1\le i\le k$. In turn, there are elements $g_i\in U_{\gamma(t_i)}\subset \mathfrak A_0^{-1}$ such that $\gamma(t_{i+1})=g_i\gamma(t_i)$ for all $1\le i\le k$. This yields $q= (g_{k}\cdots g_1)\,p$ completing the proof of the claim.

Thus according to the definition of Section 3.2, each connected component of $\mathfrak A_l^{-1}$ is a complex Banach homogeneous space under the action of $\mathfrak A_0^{-1}$.
\end{proof}
\begin{proof}[Proof of Theorem \ref{teo4.5}]
The results follow from Theorems \ref{teo3.6} and \ref{teo3.7} and the Weil integral representation formula as in the proof of Theorem \ref{teo4.1} above.
\end{proof}
\begin{proof}[Proof of Theorem \ref{teo4.6}]
Let us show that if $f\in \mathcal O(M(\mathscr A_Z),\mathfrak A_l^{-1})$, then 
$ f\in\mathcal O(M(\mathscr A_Z),\mathfrak  A)_l^{-1}$. 
Indeed, according to Theorem \ref{teo4.5}\,(b) it suffices to assume also that $f\in \cup_{\alpha\in D}\,\widehat{A}_\alpha\otimes_{\varepsilon} \mathfrak A\, \bigl(\subset \mathcal O(M(\mathscr A_Z))\otimes_{\varepsilon}\mathfrak A\bigr)$. Then $f\in \mathcal O(M(\mathscr A_Z),\mathfrak A)_l^{-1}$ by a special case of the  Bochner-Phillips-Allan-Markus-Sementsul theory, see \cite[Th.\,2.2]{V} and references therein. 

Thus, $ \mathcal O(M(\mathscr A_Z),\mathfrak A_l^{-1})\subset \mathcal O(M(\mathscr A_Z),\mathfrak  A)_l^{-1}$. 

The converse implication $\mathcal O(M(\mathscr A_Z),\mathfrak  A)_l^{-1}\subset \mathcal O(M(\mathscr A_Z),\mathfrak A_l^{-1})$ is obvious. 

The statement of the theorem for  algebra $C(M(\mathscr A_Z),\mathfrak A)$ follows from \cite[Th.\,2.2]{V}  as $C(M(\mathscr A_Z),\mathfrak A)=C(M(\mathscr A_Z))\otimes_{\varepsilon}\mathfrak A$ by the approximation property of $C(M(\mathscr A_Z))$.

Finally, the statement of the second part of theorem follows from Theorem \ref{teo3.5}.
\end{proof}

\end{document}